\documentclass[12pt]{amsart}

\usepackage{amsmath}
\usepackage{amsfonts}
\usepackage{latexsym}
\usepackage{graphicx}
\usepackage{amssymb}
\usepackage{amsthm}
\usepackage[margin=2cm]{geometry}
\usepackage{url}
\usepackage{color}
\usepackage{enumerate}
\usepackage[shortlabels]{enumitem} %pour commencer les enumerations a des nombres differents
\usepackage[small]{caption}
\usepackage{cite}       % sort and compress references number in latex
\usepackage{pgfplots}
\usepackage{longtable}
\usepackage{stmaryrd}

\usepackage{algorithm}     % the environment
\usepackage[noend]{algpseudocode}
\algrenewcommand\algorithmicrequire{\textbf{Precondition:}}
\algrenewcommand\algorithmicensure{\textbf{Postcondition:}}

\usepackage{todonotes}

\usepackage{hyperref}

\usetikzlibrary{patterns}
\usetikzlibrary{arrows}

\usepackage[T1]{fontenc}     % Hyphénation des mots accentués
\usepackage{lmodern}         % Polices vectorielles
\usepackage[utf8]{inputenc}  % Codage UNICODE (UTF-8)

\newtheorem{definition}{Definition}
\newtheorem{lemma}[definition]{Lemma}
\newtheorem{proposition}[definition]{Proposition}

\newtheorem{corollary}[definition]{Corollary}

\newtheorem{theorem}[definition]{Theorem}
\newtheorem{remark}[definition]{Remark}

\newcommand{\N}{\mathbb{N}}
\newcommand{\Z}{\mathbb{Z}}

\newcommand{\R}{\mathbb{R}}

\newcommand{\A}{\mathcal{A}}
\newcommand{\B}{\mathcal{B}}

\newcommand{\F}{\mathcal{F}}

\renewcommand{\S}{\mathcal{S}}
\newcommand{\T}{\mathcal{T}}
\newcommand{\U}{\mathcal{U}}
\newcommand{\V}{\mathcal{V}}
\newcommand{\W}{\mathcal{W}}
\newcommand{\X}{\mathcal{X}}
\renewcommand{\L}{\mathcal{L}}
\newcommand{\ba}{\mathbf{a}}
\newcommand{\be}{\mathbf{e}}
\newcommand{\bk}{\mathbf{k}}
\newcommand{\bn}{\mathbf{n}}
\newcommand{\bm}{\mathbf{m}}
\newcommand{\bp}{\mathbf{p}}
\newcommand{\bu}{\mathbf{u}}
\newcommand{\bx}{\mathbf{x}}
\newcommand{\zero}{\mathbf{0}}

\newcommand{\shape}{\textsc{shape}}
\newcommand{\SFT}{\textsc{SFT}}
\newcommand{\sctop}{\textsc{top}}
\newcommand{\scbottom}{\textsc{bottom}}
\newcommand{\scleft}{\textsc{left}}
\newcommand{\scright}{\textsc{right}}

\newcommand\tile[4]{
    \raisebox{-3mm}{
\begin{tikzpicture}[scale=0.9]
\draw (0, 0) -- (1, 0);
\draw (0, 0) -- (0, 1);
\draw (1, 1) -- (1, 0);
\draw (1, 1) -- (0, 1);
\node[rotate=0,font=\footnotesize] at (0.8, 0.5) {#1};
\node[rotate=0,font=\footnotesize] at (0.5, 0.8) {#2};
\node[rotate=0,font=\footnotesize] at (0.2, 0.5) {#3};
\node[rotate=0,font=\footnotesize] at (0.5, 0.2) {#4};
\end{tikzpicture}}}

\newcommand\unitsquare[1]{
\begin{tikzpicture}[scale=.2]
\draw[fill=#1] (0,0) rectangle (1,1);
\end{tikzpicture}}

\newcommand\dominoV[3]{
\begin{tikzpicture}[scale=.2]
\draw[fill=#1] (0,0) rectangle (1,1);
\draw[fill=#2] (0,1) rectangle (1,2);
\foreach \x/\y in {#3}
{\draw[thick] (\x,\y) -- (\x+1,\y+1);
 \draw[thick] (\x,\y+1) -- (\x+1,\y);}
\end{tikzpicture}}
\newcommand\dominoH[3]{
\begin{tikzpicture}[scale=.2]
\draw[fill=#1] (0,0) rectangle (1,1);
\draw[fill=#2] (1,0) rectangle (2,1);
\foreach \x/\y in {#3}
{\draw[thick] (\x,\y) -- (\x+1,\y+1);
 \draw[thick] (\x,\y+1) -- (\x+1,\y);}
\end{tikzpicture}}

\def\p{1.61803398874989}   % phi

\begin{document}

\title{A self-similar aperiodic set of 19 Wang tiles}
\author[S.~Labb\'e]{S\'ebastien Labb\'e}
\address[S.~Labb\'e]{CNRS, LaBRI, UMR 5800, F-33400 Talence, France}
\email{sebastien.labbe@labri.fr}
% Adresse LaBRI selon
% https://www.labri.fr/index.php?n=LaBRI.HowToSigne

\date{\today}
\begin{abstract}
We define a Wang tile set $\mathcal{U}$ of cardinality 19 and show that the set $\Omega_\mathcal{U}$ of all valid Wang tilings $\mathbb{Z}^2\to\mathcal{U}$ is self-similar, aperiodic and is a minimal subshift of $\mathcal{U}^{\mathbb{Z}^2}$. Thus $\mathcal{U}$ is the second smallest self-similar aperiodic Wang tile set known after Ammann's set of 16 Wang tiles.  The proof is based on the unique composition property.  We prove the existence of an expansive, primitive and recognizable $2$-dimensional morphism $\omega:\Omega_\mathcal{U}\to\Omega_\mathcal{U}$ that is onto up to a shift. The proof of recognizability is done in two steps using at each step the same criteria (the existence of marker tiles) for proving the existence of a recognizable one-dimensional substitution that sends each tile either on a single tile or on a domino of two tiles.
\end{abstract}

\maketitle

\section{Introduction}

\emph{Wang tiles} are unit square tiles with colored edges
as in Figure~\ref{fig:the-tile-set-U}.
Given a finite set
of Wang tiles, we consider \emph{tilings} of the Euclidean plane using
arbitrarily many copies of the tiles. Tiles are placed on the integer lattice
points of the plane with their edges oriented horizontally and vertically. The
tiles may not be rotated. The tiling is \emph{valid} if every
contiguous edges have the same color.
Given a finite set $\T$ of Wang tiles, we denote by $\Omega_\T$ the set of all
valid tilings $f:\Z^2\to\T$. The set
$\Omega_\T$ is a $2$-dimensional \emph{subshift} as it
is invariant under translations and closed under taking limits.
Hence $\Omega_\T$ is called the \emph{Wang shift} of $\T$.
A nonempty Wang shift $\Omega_\T$ is \emph{aperiodic} if none of the tilings in
$\Omega_\T$ have a nontrivial period.
%A tiling $f\in\Omega_\T$ is \emph{periodic} if there exists a \emph{period}
%$\ba\in\Z^2\setminus\{(0,0)\}$ such that $f(\bx)=f(\bx+\ba)$ for every
%$\bx\in\Z^2$. 
Chapters 10 and 11 of \cite{MR857454} and the more recent
book on aperiodic order \cite{MR3136260} give an excellent overview of what is
known on aperiodic tilings (in $\Z^d$ as in $\R^d$) together with their
applications to physics and crystallography.

\begin{figure}[h]
\begin{center}
\input{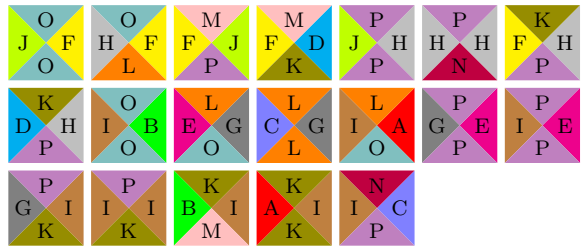}
\end{center}
    \caption{The set $\U$ of 19 Wang tiles.}
    \label{fig:the-tile-set-U}
\end{figure}

A nonexhaustive list of aperiodic Wang tile sets is shown in
Table~\ref{tab:list-aperiodic-sets}.  Examples of small aperiodic Wang tile
sets include Ammann's 16 tiles \cite[p. 595]{MR857454}, Kari's 14 tiles \cite{MR1417578}  and Culik's 13 tiles \cite{MR1417576}.  The
question of finding the smallest aperiodic set of Wang tiles was open until
Jeandel and Rao proved 
\cite{jeandel_aperiodic_2015}
the existence of an aperiodic set of 11 Wang tiles
and that no set of Wang tiles of cardinality $\leq10$ is aperiodic.
The proofs of their two results is based on transducers, which has proven to be
an excellent approach for the search of the smallest aperiodic Wang tile set.
A transducer is a finite-state machine with two memory tapes:
an input tape and an output tape. There is a natural way to construct a
transducer from a set of Wang tiles (see Figure~\ref{fig:T7transducer})
which computes some row of tiles that can be placed above some other row of
tiles.

%Goodman:\cite{MR1702375}

% Creation of the images and inputs:
% sage for_article1.sage
\newcommand\UTileO{
\begin{tikzpicture}
[scale=0.900000000000000]
\tikzstyle{every node}=[font=\tiny]
% tile at position (x,y)=(0, 0)
\node at (0.5, 0.5) {0};
\draw (0, 0) -- ++ (0,1);
\draw (0, 0) -- ++ (1,0);
\draw (1, 0) -- ++ (0,1);
\draw (0, 1) -- ++ (1,0);
\node[rotate=0] at (0.800000000000000, 0.5) {F};
\node[rotate=0] at (0.5, 0.800000000000000) {O};
\node[rotate=0] at (0.200000000000000, 0.5) {J};
\node[rotate=0] at (0.5, 0.200000000000000) {O};
\end{tikzpicture}
} % end of newcommand
\newcommand\UTileI{
\begin{tikzpicture}
[scale=0.900000000000000]
\tikzstyle{every node}=[font=\tiny]
% tile at position (x,y)=(0, 0)
\node at (0.5, 0.5) {1};
\draw (0, 0) -- ++ (0,1);
\draw (0, 0) -- ++ (1,0);
\draw (1, 0) -- ++ (0,1);
\draw (0, 1) -- ++ (1,0);
\node[rotate=0] at (0.800000000000000, 0.5) {F};
\node[rotate=0] at (0.5, 0.800000000000000) {O};
\node[rotate=0] at (0.200000000000000, 0.5) {H};
\node[rotate=0] at (0.5, 0.200000000000000) {L};
\end{tikzpicture}
} % end of newcommand
\newcommand\UTileII{
\begin{tikzpicture}
[scale=0.900000000000000]
\tikzstyle{every node}=[font=\tiny]
% tile at position (x,y)=(0, 0)
\node at (0.5, 0.5) {2};
\draw (0, 0) -- ++ (0,1);
\draw (0, 0) -- ++ (1,0);
\draw (1, 0) -- ++ (0,1);
\draw (0, 1) -- ++ (1,0);
\node[rotate=0] at (0.800000000000000, 0.5) {J};
\node[rotate=0] at (0.5, 0.800000000000000) {M};
\node[rotate=0] at (0.200000000000000, 0.5) {F};
\node[rotate=0] at (0.5, 0.200000000000000) {P};
\end{tikzpicture}
} % end of newcommand
\newcommand\UTileIII{
\begin{tikzpicture}
[scale=0.900000000000000]
\tikzstyle{every node}=[font=\tiny]
% tile at position (x,y)=(0, 0)
\node at (0.5, 0.5) {3};
\draw (0, 0) -- ++ (0,1);
\draw (0, 0) -- ++ (1,0);
\draw (1, 0) -- ++ (0,1);
\draw (0, 1) -- ++ (1,0);
\node[rotate=0] at (0.800000000000000, 0.5) {D};
\node[rotate=0] at (0.5, 0.800000000000000) {M};
\node[rotate=0] at (0.200000000000000, 0.5) {F};
\node[rotate=0] at (0.5, 0.200000000000000) {K};
\end{tikzpicture}
} % end of newcommand
\newcommand\UTileIV{
\begin{tikzpicture}
[scale=0.900000000000000]
\tikzstyle{every node}=[font=\tiny]
% tile at position (x,y)=(0, 0)
\node at (0.5, 0.5) {4};
\draw (0, 0) -- ++ (0,1);
\draw (0, 0) -- ++ (1,0);
\draw (1, 0) -- ++ (0,1);
\draw (0, 1) -- ++ (1,0);
\node[rotate=0] at (0.800000000000000, 0.5) {H};
\node[rotate=0] at (0.5, 0.800000000000000) {P};
\node[rotate=0] at (0.200000000000000, 0.5) {J};
\node[rotate=0] at (0.5, 0.200000000000000) {P};
\end{tikzpicture}
} % end of newcommand
\newcommand\UTileV{
\begin{tikzpicture}
[scale=0.900000000000000]
\tikzstyle{every node}=[font=\tiny]
% tile at position (x,y)=(0, 0)
\node at (0.5, 0.5) {5};
\draw (0, 0) -- ++ (0,1);
\draw (0, 0) -- ++ (1,0);
\draw (1, 0) -- ++ (0,1);
\draw (0, 1) -- ++ (1,0);
\node[rotate=0] at (0.800000000000000, 0.5) {H};
\node[rotate=0] at (0.5, 0.800000000000000) {P};
\node[rotate=0] at (0.200000000000000, 0.5) {H};
\node[rotate=0] at (0.5, 0.200000000000000) {N};
\end{tikzpicture}
} % end of newcommand
\newcommand\UTileVI{
\begin{tikzpicture}
[scale=0.900000000000000]
\tikzstyle{every node}=[font=\tiny]
% tile at position (x,y)=(0, 0)
\node at (0.5, 0.5) {6};
\draw (0, 0) -- ++ (0,1);
\draw (0, 0) -- ++ (1,0);
\draw (1, 0) -- ++ (0,1);
\draw (0, 1) -- ++ (1,0);
\node[rotate=0] at (0.800000000000000, 0.5) {H};
\node[rotate=0] at (0.5, 0.800000000000000) {K};
\node[rotate=0] at (0.200000000000000, 0.5) {F};
\node[rotate=0] at (0.5, 0.200000000000000) {P};
\end{tikzpicture}
} % end of newcommand
\newcommand\UTileVII{
\begin{tikzpicture}
[scale=0.900000000000000]
\tikzstyle{every node}=[font=\tiny]
% tile at position (x,y)=(0, 0)
\node at (0.5, 0.5) {7};
\draw (0, 0) -- ++ (0,1);
\draw (0, 0) -- ++ (1,0);
\draw (1, 0) -- ++ (0,1);
\draw (0, 1) -- ++ (1,0);
\node[rotate=0] at (0.800000000000000, 0.5) {H};
\node[rotate=0] at (0.5, 0.800000000000000) {K};
\node[rotate=0] at (0.200000000000000, 0.5) {D};
\node[rotate=0] at (0.5, 0.200000000000000) {P};
\end{tikzpicture}
} % end of newcommand
\newcommand\UTileVIII{
\begin{tikzpicture}
[scale=0.900000000000000]
\tikzstyle{every node}=[font=\tiny]
% tile at position (x,y)=(0, 0)
\node at (0.5, 0.5) {8};
\draw (0, 0) -- ++ (0,1);
\draw (0, 0) -- ++ (1,0);
\draw (1, 0) -- ++ (0,1);
\draw (0, 1) -- ++ (1,0);
\node[rotate=0] at (0.800000000000000, 0.5) {B};
\node[rotate=0] at (0.5, 0.800000000000000) {O};
\node[rotate=0] at (0.200000000000000, 0.5) {I};
\node[rotate=0] at (0.5, 0.200000000000000) {O};
\end{tikzpicture}
} % end of newcommand
\newcommand\UTileIX{
\begin{tikzpicture}
[scale=0.900000000000000]
\tikzstyle{every node}=[font=\tiny]
% tile at position (x,y)=(0, 0)
\node at (0.5, 0.5) {9};
\draw (0, 0) -- ++ (0,1);
\draw (0, 0) -- ++ (1,0);
\draw (1, 0) -- ++ (0,1);
\draw (0, 1) -- ++ (1,0);
\node[rotate=0] at (0.800000000000000, 0.5) {G};
\node[rotate=0] at (0.5, 0.800000000000000) {L};
\node[rotate=0] at (0.200000000000000, 0.5) {E};
\node[rotate=0] at (0.5, 0.200000000000000) {O};
\end{tikzpicture}
} % end of newcommand
\newcommand\UTileX{
\begin{tikzpicture}
[scale=0.900000000000000]
\tikzstyle{every node}=[font=\tiny]
% tile at position (x,y)=(0, 0)
\node at (0.5, 0.5) {10};
\draw (0, 0) -- ++ (0,1);
\draw (0, 0) -- ++ (1,0);
\draw (1, 0) -- ++ (0,1);
\draw (0, 1) -- ++ (1,0);
\node[rotate=0] at (0.800000000000000, 0.5) {G};
\node[rotate=0] at (0.5, 0.800000000000000) {L};
\node[rotate=0] at (0.200000000000000, 0.5) {C};
\node[rotate=0] at (0.5, 0.200000000000000) {L};
\end{tikzpicture}
} % end of newcommand
\newcommand\UTileXI{
\begin{tikzpicture}
[scale=0.900000000000000]
\tikzstyle{every node}=[font=\tiny]
% tile at position (x,y)=(0, 0)
\node at (0.5, 0.5) {11};
\draw (0, 0) -- ++ (0,1);
\draw (0, 0) -- ++ (1,0);
\draw (1, 0) -- ++ (0,1);
\draw (0, 1) -- ++ (1,0);
\node[rotate=0] at (0.800000000000000, 0.5) {A};
\node[rotate=0] at (0.5, 0.800000000000000) {L};
\node[rotate=0] at (0.200000000000000, 0.5) {I};
\node[rotate=0] at (0.5, 0.200000000000000) {O};
\end{tikzpicture}
} % end of newcommand
\newcommand\UTileXII{
\begin{tikzpicture}
[scale=0.900000000000000]
\tikzstyle{every node}=[font=\tiny]
% tile at position (x,y)=(0, 0)
\node at (0.5, 0.5) {12};
\draw (0, 0) -- ++ (0,1);
\draw (0, 0) -- ++ (1,0);
\draw (1, 0) -- ++ (0,1);
\draw (0, 1) -- ++ (1,0);
\node[rotate=0] at (0.800000000000000, 0.5) {E};
\node[rotate=0] at (0.5, 0.800000000000000) {P};
\node[rotate=0] at (0.200000000000000, 0.5) {G};
\node[rotate=0] at (0.5, 0.200000000000000) {P};
\end{tikzpicture}
} % end of newcommand
\newcommand\UTileXIII{
\begin{tikzpicture}
[scale=0.900000000000000]
\tikzstyle{every node}=[font=\tiny]
% tile at position (x,y)=(0, 0)
\node at (0.5, 0.5) {13};
\draw (0, 0) -- ++ (0,1);
\draw (0, 0) -- ++ (1,0);
\draw (1, 0) -- ++ (0,1);
\draw (0, 1) -- ++ (1,0);
\node[rotate=0] at (0.800000000000000, 0.5) {E};
\node[rotate=0] at (0.5, 0.800000000000000) {P};
\node[rotate=0] at (0.200000000000000, 0.5) {I};
\node[rotate=0] at (0.5, 0.200000000000000) {P};
\end{tikzpicture}
} % end of newcommand
\newcommand\UTileXIV{
\begin{tikzpicture}
[scale=0.900000000000000]
\tikzstyle{every node}=[font=\tiny]
% tile at position (x,y)=(0, 0)
\node at (0.5, 0.5) {14};
\draw (0, 0) -- ++ (0,1);
\draw (0, 0) -- ++ (1,0);
\draw (1, 0) -- ++ (0,1);
\draw (0, 1) -- ++ (1,0);
\node[rotate=0] at (0.800000000000000, 0.5) {I};
\node[rotate=0] at (0.5, 0.800000000000000) {P};
\node[rotate=0] at (0.200000000000000, 0.5) {G};
\node[rotate=0] at (0.5, 0.200000000000000) {K};
\end{tikzpicture}
} % end of newcommand
\newcommand\UTileXV{
\begin{tikzpicture}
[scale=0.900000000000000]
\tikzstyle{every node}=[font=\tiny]
% tile at position (x,y)=(0, 0)
\node at (0.5, 0.5) {15};
\draw (0, 0) -- ++ (0,1);
\draw (0, 0) -- ++ (1,0);
\draw (1, 0) -- ++ (0,1);
\draw (0, 1) -- ++ (1,0);
\node[rotate=0] at (0.800000000000000, 0.5) {I};
\node[rotate=0] at (0.5, 0.800000000000000) {P};
\node[rotate=0] at (0.200000000000000, 0.5) {I};
\node[rotate=0] at (0.5, 0.200000000000000) {K};
\end{tikzpicture}
} % end of newcommand
\newcommand\UTileXVI{
\begin{tikzpicture}
[scale=0.900000000000000]
\tikzstyle{every node}=[font=\tiny]
% tile at position (x,y)=(0, 0)
\node at (0.5, 0.5) {16};
\draw (0, 0) -- ++ (0,1);
\draw (0, 0) -- ++ (1,0);
\draw (1, 0) -- ++ (0,1);
\draw (0, 1) -- ++ (1,0);
\node[rotate=0] at (0.800000000000000, 0.5) {I};
\node[rotate=0] at (0.5, 0.800000000000000) {K};
\node[rotate=0] at (0.200000000000000, 0.5) {B};
\node[rotate=0] at (0.5, 0.200000000000000) {M};
\end{tikzpicture}
} % end of newcommand
\newcommand\UTileXVII{
\begin{tikzpicture}
[scale=0.900000000000000]
\tikzstyle{every node}=[font=\tiny]
% tile at position (x,y)=(0, 0)
\node at (0.5, 0.5) {17};
\draw (0, 0) -- ++ (0,1);
\draw (0, 0) -- ++ (1,0);
\draw (1, 0) -- ++ (0,1);
\draw (0, 1) -- ++ (1,0);
\node[rotate=0] at (0.800000000000000, 0.5) {I};
\node[rotate=0] at (0.5, 0.800000000000000) {K};
\node[rotate=0] at (0.200000000000000, 0.5) {A};
\node[rotate=0] at (0.5, 0.200000000000000) {K};
\end{tikzpicture}
} % end of newcommand
\newcommand\UTileXVIII{
\begin{tikzpicture}
[scale=0.900000000000000]
\tikzstyle{every node}=[font=\tiny]
% tile at position (x,y)=(0, 0)
\node at (0.5, 0.5) {18};
\draw (0, 0) -- ++ (0,1);
\draw (0, 0) -- ++ (1,0);
\draw (1, 0) -- ++ (0,1);
\draw (0, 1) -- ++ (1,0);
\node[rotate=0] at (0.800000000000000, 0.5) {C};
\node[rotate=0] at (0.5, 0.800000000000000) {N};
\node[rotate=0] at (0.200000000000000, 0.5) {I};
\node[rotate=0] at (0.5, 0.200000000000000) {P};
\end{tikzpicture}
} % end of newcommand

In this contribution we introduce a set $\U$
of 19 Wang tiles (see Figure~\ref{fig:the-tile-set-U}).
The main result of the current contribution is that
$\U$ is aperiodic and that it generates
tilings (Figure~\ref{fig:U_image_5th_of4}) of $\Z^2$ with a self-similar
structure.
With 19 tiles, this makes $\U$ the second smallest self-similar aperiodic
set after Ammann's tile set which has 16 only tiles. 
The tile set $\U$ comes from the study of the structure
of Jeandel-Rao aperiodic tilings. The link with Jeandel-Rao tilings needs more
tools and space and will be done in a forthcoming paper
\cite{labbe_structure_2018}.
In this contribution, we prove the following
result.

\begin{theorem}\label{thm:main}
    The Wang shift $\Omega_\U$ is self-similar, aperiodic and minimal.
    %The tile set $\U$ is self-similar and aperiodic.
\end{theorem}

To prove aperiodicity, we use the \emph{composition-decomposition method}
also called the \emph{unique composition property}
for tilings of $\R^d$ described in
\cite[Thm 10.1.1]{MR857454},
\cite[Sec. 5.7.1]{MR3136260} and \cite[Lemma 2.7]{MR1452190}.
Let us cite \cite{MR1156132} which summarizes informally the
composition-decomposition method as two conditions to be satisfied
for a tile set $\T$: \textit{if}
\begin{enumerate}[(C1)]
\item \label{cond:C1} 
    \textit{in every tiling admitted by $\T$ there is a unique way in which the
        tiles can be grouped into patches which lead to a tiling by supertiles;
        and}
\item \label{cond:C2} 
    \textit{
    the markings on the supertiles, inherited from the original tiles,
imply a matching condition for the supertiles which is exactly equivalent to
        that originally specified for the tiles,}
\end{enumerate}
\textit{then $\T$ is aperiodic.}
Condition \ref{cond:C1} was formalized with many vocabularies in the literature
including the term \emph{unambiguous} in \cite{MR2507046}. In this contribution
(see Proposition~\ref{prop:expansive-recognizable-aperiodic}), we use the
notion of \emph{recognizability}.
This method was used to prove aperiodicity of the
first discovered aperiodic examples (see Table~\ref{tab:list-aperiodic-sets}).
The proof of aperiodicity of 104 Berger's tile set is
obtained by the use of two-by-two substitutions of the form $\square\mapsto\boxplus$, mapping each tile to a two-by-two square of 4 tiles.
The same holds for Knuth's and Robinson's tile sets, their constructions being
obtained as a simplification of Berger's tile set (see also \cite{MR2507046}).
Note that the composition-decomposition method does not apply to all aperiodic
tile set, e.g. aperiodicity of Kari and Culik tile sets
\cite{MR1417578,MR1417576} follows from arithmetic properties.

% https://en.wikipedia.org/wiki/List_of_aperiodic_sets_of_tiles
\begin{table}[h]
\begin{center}
\begin{tabular}{lcccccl}
    Author & Size & Year & Proof of aperiodicity & s.s. &$\lambda$ & References\\
    \hline
    %Berger & 20426 & 1966 (lowered down later to 104)
    Berger & 104 & 1964 & unique composition property & yes & 2 & \cite{MR2939561,MR2507046}\\
    Knuth & 92 & 1968 & '' & yes & 2 & \cite[2.3.4.3]{MR0286317}\\
    Robinson & 56 & 1971 &  '' & yes & 2 & \cite{MR0297572}\\
    Gr\"unbaum et al. & 24 & 1987 & '' & yes & $\frac{1+\sqrt{5}}{2}$ & \cite{MR857454}\\
    Ammann & 16 & 1971 & '' & yes & $\frac{1+\sqrt{5}}{2}$ & \cite[p.
    595]{MR857454} \\
    Kari & 14 & 1996 & arithmetic properties & no & -- & \cite{MR1417578}  \\
    Culik & 13 & 1996 & '' & no & -- & \cite{MR1417576}  \\
    Jeandel, Rao & 11 & 2015 & transducers and Fibonacci word & no & -- & \cite{jeandel_aperiodic_2015} \\
\end{tabular}
\end{center}
    \caption{A list of small aperiodic sets of Wang tiles. 
    Whether the tile set is self-similar (s.s.) is indicated.
    When applicable, $\lambda$ denotes the expansion factor of the similarity
    involved in the composition (see Section~\ref{sec:stone} on stone
    inflations). Ammann's tile set is the smallest known self-similar aperiodic
    Wang tile set. The history includes many unpublished results and a more
    complete story can be found in \cite[Sec. 1.2]{jeandel_aperiodic_2015}.}
    \label{tab:list-aperiodic-sets}
\end{table}

As it is the case for Ammann set of 16 Wang tiles \cite[p. 595]{MR857454}, the
substitutive structure of tilings in $\Omega_\U$ have the form
$\square\mapsto\square$,
$\square\mapsto\dominoV{none}{none}{}$,
$\square\mapsto\dominoH{none}{none}{}$,
$\square\mapsto\boxplus$, that is, the image of a tile is a tile, a vertical or
horizontal domino of two tiles, or a two-by-two square of 4 tiles.
More precisely, in this contribution, we show that any tiling of $\Z^2$ by the
19 tiles in $\U$ can be decomposed uniquely into a tilings by the supertiles shown in
Figure~\ref{fig:Usupertiles}. It can be seen by comparing
Figure~\ref{fig:the-tile-set-U} and
Figure~\ref{fig:Usupertiles} that the markings on the supertiles imply a
matching condition which is equivalent to those specified for~$\U$ (horizontal
color $KO$ in the supertiles plays the role of color $P$ in the tile set $\U$,
etc.), thus satisfying \ref{cond:C1} and \ref{cond:C2}.

\begin{figure}
\begin{center}
\input{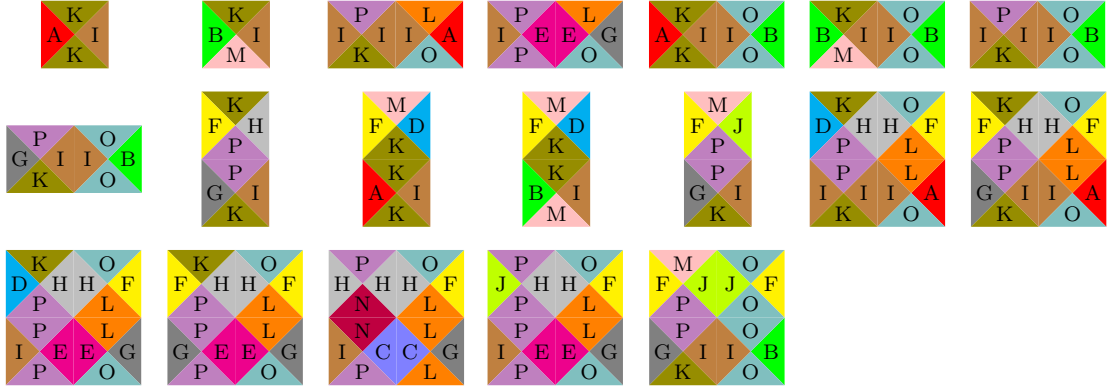}
\end{center}
    \caption{Any tiling in $\Omega_\U$ can be decomposed uniquely into a tiling
    by these 19 supertiles. Each of these supertiles is equivalent to one of the
    tiles in $\U$ which explains the self-similar structure of tilings in
    $\Omega_\U$. It is reminiscent of Figure 11.1.16 in \cite{MR857454}.}
    \label{fig:Usupertiles}
\end{figure}

\begin{figure}[h]
%\begin{center}
%    \includegraphics{article1_image_5th_of4.pdf}
%\end{center}
\begin{center}
    \input{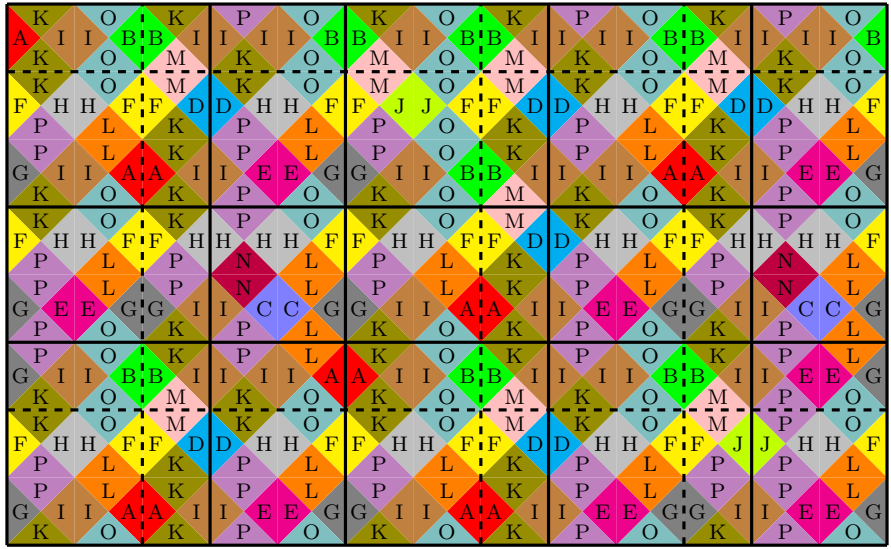}
\end{center}
    \caption{A $13\times 8$ rectangle tiled by tiles from $\U$. 
    It corresponds to $\omega^5(u_4)$ where $\omega=\alpha\beta\gamma$ is
    defined in Section~\ref{sec:aperiodicityofU}.
    Notice that the sequence of colors at the top correspond to the sequence at
    the bottom allowing a periodic tiling of a infinite vertical strip of
    width~13.
    Dashed lines identify the decomposition into supertiles.
    Solid lines identify the decomposition into supertiles of the next level.}
    \label{fig:U_image_5th_of4}
\end{figure}

The particularity and the importance of this contribution stands in breaking
down the proof of unique composition property into smaller steps.
Thus, each step consists in proving recognizability and surjectivity up to a shift for a
substitution of the form $\square\mapsto\square,\square\mapsto\dominoV{none}{none}{}$
or of the form $\square\mapsto\square,\square\mapsto\dominoH{none}{none}{}$.
Each step has the effect of identifying special rows (or columns) made of marker tiles
(see Definition~\ref{def:markers}) that have to be adjacent in one direction but can not
be adjacent in the other direction.
The wider applicability of this method (Theorem~\ref{thm:exist-homeo}) allowing
to be automated by computer check will be used in \cite{labbe_structure_2018}.

Following the theory of $S$-adic systems on $\Z$ \cite{MR3330561} and
hierarchical tilings of $\R^d$ \cite{MR3226791},
we desubstitute tilings in $\Omega_\U$ into tilings in
$\Omega_\V$ by a $2$-dimensional morphism $\alpha:\Omega_\V\to\Omega_\U$ for some Wang tile set $\V$
and we desubstitute tilings of $\Omega_\V$ into tilings in $\Omega_\W$ 
by a $2$-dimensional morphism $\beta:\Omega_\W\to\Omega_\V$ for some Wang tile set $\W$.
Finally we prove that $\alpha$ and $\beta$ are recognizable 
thus satisfying \ref{cond:C1}
and that there is a
letter to letter bijection $\gamma:\Omega_\U\to\Omega_\W$
that satisfies \ref{cond:C2}.
\begin{center}
\begin{tikzpicture}[auto]
    \node (U) at (0,0) {$\Omega_\U$};
    \node (V) at (5,0) {$\Omega_\V$};
    \node (W) at (10,0) {$\Omega_\W$};
    \node (U2) at (13,0) {$\Omega_\U$};
    \draw[to-,very thick] (U) to node {$\alpha:
    \unitsquare{none}\mapsto\unitsquare{none},
    \unitsquare{none}\mapsto\dominoV{none}{none}{}$} (V);
    \draw[to-,very thick] (V) to node {$\beta:
    \unitsquare{none}\mapsto\unitsquare{none},
    \unitsquare{none}\mapsto\dominoH{none}{none}{}$} (W);
    \draw[to-,very thick] (W) to node {$\gamma:
    \unitsquare{none}\mapsto\unitsquare{none}$} (U2);
\end{tikzpicture}
\end{center}
A key result in this contribution is Theorem~\ref{thm:exist-homeo} since
it is used twice to prove that we can desubstitute uniquely.  It can be seen as
a $2$-dimensional generalization to Wang subshifts of the notion of
\emph{derived sequences} introduced in \cite{MR1489074} for $1$-dimensional
substitutive sequences. Remark that another generalization of derived
sequences to tilings of $\R^2$ based on Voronoï tessellations is presented in
\cite{MR1961011}. The advantage of markers and Theorem~\ref{thm:exist-homeo} is
that the rectangular lattice structure of Wang tilings is preserved under each
derivation.

The paper is structured as follows.
In Section~\ref{sec:preliminaries}, we present the necessary definitions and
notations on Wang tiles including self-similarity, recognizability and
aperiodicity.
In Section~\ref{sec:sufficient-condition},
we present a sufficient condition for the existence of
a recognizable $2$-dimensional morphism allowing to desubstitute any tiling of
a Wang shift.
In Section~\ref{sec:desubstitionofU} and Section~\ref{sec:desubstitionofV}, we prove the existence of
tile sets $\V$ and $\W$ and $2$-dimensional morphisms 
$\alpha:\Omega_\V\to\Omega_\U$ and
$\beta:\Omega_\W\to\Omega_\V$
that are recognizable and surjective up to a shift.
In Section~\ref{sec:aperiodicityofU}, we prove that $\U$ and $\W$ are
equivalent tile sets and we prove self-similarity, aperiodicity and minimality
of $\Omega_\U$.
In Section~\ref{sec:stone}, we present a stone inflation version of
$\omega=\alpha\beta\gamma:\Omega_\U\to\Omega_\U$.
A minimal number of lemmas are proved using Sage
and their proofs are provided in Section~\ref{sec:appendix}.

\section{Preliminaries}\label{sec:preliminaries}

In this section, we introduce 
subshifts, shifts of finite type,
Wang tiles,
fusion of Wang tiles,
the transducer representation of Wang tiles,
$d$-dimensional words, morphisms and languages.
We recall the notions of 
self-similarity and
recognizability for proving aperiodicity.

We denote by $\Z=\{\dots,-1,0,1,2,\dots\}$ the integers and
by $\N=\{0,1,2,\dots\}$ the nonnegative integers.
If $d\geq1$ is an integer and $1\leq k\leq d$, we denote by
$\be_k=(0,\dots,0,1,0,\dots,0)\in\Z^d$ the vector of the canonical
basis of $\Z^d$ with a $1$ as position $k$ and $0$ elsewhere.

\subsection{Subshifts and shifts of finite type}

% following KLAUS SCHMIDT

We follow the notations of \cite{MR1861953}.
Let $\A$ be a finite set, $d\geq 1$, and let $\A^{\Z^d}$ be the set of all maps
$x:\Z^d\to\A$, furnished with the compact product topology. We write a typical
point $x\in\A^{\Z^d}$ as $x=(x_\bm)=(x_\bm:\bm\in\Z^d)$, where $x_\bm\in\A$
denotes the value of $x$ at $\bm$. The \emph{shift action} $\sigma:\bn\mapsto
\sigma^\bn$ of $\Z^d$ on $\A^{\Z^d}$ is defined by
\begin{equation}\label{eq:shift-action}
    (\sigma^\bn(x))_\bm = x_{\bm+\bn}
\end{equation}
for every $x=(x_\bm)\in\A^{\Z^d}$ and $\bn\in\Z^d$. A subset $X\subset
\A^{\Z^d}$ is \emph{shift-invariant} if $\sigma^\bn(X)=X$ for every
$\bn\in\Z^d$, and a closed, shift-invariant subset $X\subset\A^{\Z^d}$ is a
\emph{subshift}. 
If $X\subset\A^{\Z^d}$ is a subshift we write
$\sigma=\sigma^X$ for the restriction of the shift-action
\eqref{eq:shift-action} to $X$. 
If $X\subset\A^{\Z^d}$ is a subshift it will sometimes be helpful to specify the
shift-action of $\Z^d$ explicitly and to write $(X,\sigma)$ instead of $X$.
A subshift $(X,\sigma)$ is called \emph{minimal} if $X$ does not
contain any nonempty, proper, closed shift-invariant subset.

For any subset $S\subset\Z^d$ we denote by $\pi_S:\A^{\Z^d}\to\A^S$ the
projection map which restricts every $x\in\A^{\Z^d}$ to $S$. 
A \emph{pattern} is a function $p:S\to\A$ for some finite subset
$S\subset\Z^d$.
% A subshift $X\subset\A^{\Z^d}$ is a \emph{shift of finite type (SFT)} if there
% exists a finite set $S\subset\Z^d$ such that
% \begin{equation}\label{eq:SFT}
%     X = \{x\in\A^{\Z^d} \mid \pi_S\cdot\sigma^\bn(x)\in\pi_S(X)
%         \text{ for every } \bn\in\Z^d\}.
% \end{equation}
A subshift $X\subset\A^{\Z^d}$ is a 
\emph{shift of finite type} (SFT) if there exists a finite set $\F$
of \emph{forbidden} patterns such that
\begin{equation}\label{eq:SFT}
    X = \{x\in\A^{\Z^d} \mid \pi_S\cdot\sigma^\bn(x)\notin\F
    \text{ for every } \bn\in\Z^d \text{ and } S\subset\Z^d\}.
\end{equation}
In this case, we write $X=\SFT(\F)$.

\subsection{Wang tiles}

A \emph{Wang tile} $\tau=\tile{$a$}{$b$}{$c$}{$d$}$ is a unit square with colored edges
formally represented as a tuple of four colors $(a,b,c,d)\in I\times J\times
I\times J$
where $I$, $J$ are
two finite sets (the vertical and horizontal colors respectively). 
For each Wang tile $\tau=(a,b,c,d)$, we denote by
$\scright(\tau)=a$,
$\sctop(\tau)=b$,
$\scleft(\tau)=c$,
$\scbottom(\tau)=d$
the colors of the right, top, left and bottom edges of $\tau$
\cite{wang_proving_1961,MR0297572}.

Let $\T$ be a set of Wang tiles. A \emph{tiling of $\Z^2$ by $\T$} is an assignation $x$
of tiles to each position of $\Z^2$ so that contiguous edges have the same
color, that is, it is a function $x:\Z^2\to\T$ satisfying
\begin{align}
    \scright\circ x(\bn)&=\scleft\circ x(\bn+\be_1)\label{eq:1}\\
    \sctop\circ x(\bn)&=\scbottom\circ x(\bn+\be_2)\label{eq:2}
\end{align}
for every $\bn\in\Z^2$.
We denote by $\Omega_\T\subset\T^{\Z^2}$ the set of all Wang tilings of $\Z^2$
by $\T$ and we call it the \emph{Wang shift} of $\T$. It is a SFT of
the form \eqref{eq:SFT}. 

A set of Wang tiles $\T$ \emph{tiles} the plane if $\Omega_\T\neq\varnothing$
and \emph{does not tile} the plane if $\Omega_\T=\varnothing$.
A tiling $x\in\Omega_\T$ is \emph{periodic} if there is a nonzero period
$\bn\in\Z^2\setminus\{(0,0)\}$ such that $x=\sigma^\bn(x)$
and otherwise it is said \emph{nonperiodic}.
A set of Wang tiles $\T$ is \emph{periodic} if there is a tiling
$x\in\Omega_\T$ which is periodic. A Wang tile set $\T$ is \emph{aperiodic} if
$\Omega_\T\neq\varnothing$ and every tiling $x\in\Omega_\T$ is nonperiodic.
As explained in the first page of \cite{MR0297572} 
(see also \cite[Prop. 5.9]{MR3136260}),
if $\T$ is periodic, then there is a tiling $x$ by $\T$ with two linearly
independent translation vectors (in particular a tiling $x$ with vertical
and horizontal translation vectors).

We say that two Wang tile sets $\T$ and $\S$ are
\emph{equivalent} if there exist two bijections $i:I\to I'$ $j:J\to J'$ such
that
\begin{equation*}
    \S = \{(i(a),j(b),i(c),j(d)) \mid (a,b,c,d) \in \T\}.
\end{equation*}

If $\T$ is a set of Wang tiles, then we define the \emph{dual tile set} $\T^*$
as its image under a reflection through the positive diagonal, i.e.,
\begin{equation*}
    \T^* = \left\{(b,a,d,c)=\tile{$b$}{$a$}{$d$}{$c$}\,\middle|\,
                  (a,b,c,d)=\tile{$a$}{$b$}{$c$}{$d$}\in\T\right\}.
\end{equation*}

\subsection{Fusion of Wang tiles}

Now, we introduce a fusion operation on Wang tiles that can be adjacent in a
tiling.
Let 
$u=\tile{Y}{B}{X}{A}$ and
$v=\tile{Z}{D}{W}{C}$
be two Wang tiles.
For $i\in\{1,2\}$,
we define a binary operation on Wang tiles denoted $\boxminus^i$ as
\begin{equation*}
u\boxminus^1 v=\tile{Z}{BD}{X}{AC}
    \quad
    \text{ if }
    Y = W
    \qquad
    \text{ and }
    \qquad
u\boxminus^2 v=
    \raisebox{-3mm}{
\begin{tikzpicture}[scale=0.9]
% tile at position (x,y)=(0, 0)
\draw (0, 0) -- (1, 0);
\draw (0, 0) -- (0, 1);
\draw (1, 1) -- (1, 0);
\draw (1, 1) -- (0, 1);
\node[rotate=90,font=\footnotesize] at (0.8, 0.5) {YZ};
\node[rotate=0,font=\footnotesize] at (0.5, 0.8) {D};
\node[rotate=90,font=\footnotesize] at (0.2, 0.5) {XW};
\node[rotate=0,font=\footnotesize] at (0.5, 0.2) {A};
\end{tikzpicture}}
    \quad \text{ if } B = C.
\end{equation*}
If $i=1$ and $Y\neq W$ or if
$i=2$ and $B\neq C$, we say that $u\boxminus^i v$ is \emph{not well-defined}.
If $u\boxminus^i v$ is well-defined for some $i\in\{1,2\}$, it means that $u$
and $v$ can appear at position $\bn$ and $\bn+\be_i$ in a tiling for some
$\bn\in\Z^d$. When appropriate, we also use the following more visual notation
for denoting $u\boxminus^i v$:
\begin{equation*}
u\boxminus^1 v = 
\begin{array}{|cc|}\hline u & v\\\hline\end{array}
    \qquad
    \text{ and }
    \qquad
u\boxminus^2 v = 
\begin{array}{|c|}
\hline
v \\
u \\
\hline
\end{array}.
\end{equation*}
For each $i\in\{1,2\}$, one can define a new tile set from 
two Wang tile sets $\T$ and $\S$ as
\begin{equation*}
    \T\boxminus^i\S = \{u\boxminus^i v \text{ well-defined }\mid
                        u\in\T,v\in\S\}.
\end{equation*}
The fusion operation together with taking the dual of a tile set satisfy the
following equations:
\begin{equation*}
    (\T\boxminus^1\S)^* = \T^*\boxminus^2\S^*
    \qquad
    \text{ and }
    \qquad
    (\T\boxminus^2\S)^* = \T^*\boxminus^1\S^*.
\end{equation*}

\subsection{Transducer representation of Wang tiles}

A transducer $M$ is a labeled directed graph whose nodes are called
\emph{states} and edges are called \emph{transitions}. The transitions are
labeled by pairs $a|b$ of letters. The first letter $a$ is the input symbol and
the second letter $b$ is the output symbol. There is no initial nor final state.
A transducer $M$ computes a relation $\rho(M)$ between bi-infinite sequences of
letters. 

As observed in \cite{MR1417578} and extensively used in
\cite{jeandel_aperiodic_2015}, any finite set of Wang tiles may be interpreted
as a transducer. To a given tile set $\T$, the states of the corresponding
transducer $M_\T$ are the colors of the vertical edges. The colors of
horizontal edges are the input and output symbols. There is a transition from
state $s$ to state $t$ with label $a|b$ if and only if there is a tile
$(t,b,s,a)\in\T$ whose left, right, bottom and top edges are colored by $s$,
$t$, $a$ and $b$, respectively:
\begin{equation*}
    M_\T = \left\{ s\xrightarrow{a|b} t : 
           \tile{$t$}{$b$}{$s$}{$a$}=(t,b,s,a) \in \T \right\}.
\end{equation*}
The transducer $M_\U$ for the tile set $\U$ is shown in Figure~\ref{fig:T7transducer}.
\begin{figure}[h]
\begin{center}
    \includegraphics[width=.8\linewidth]{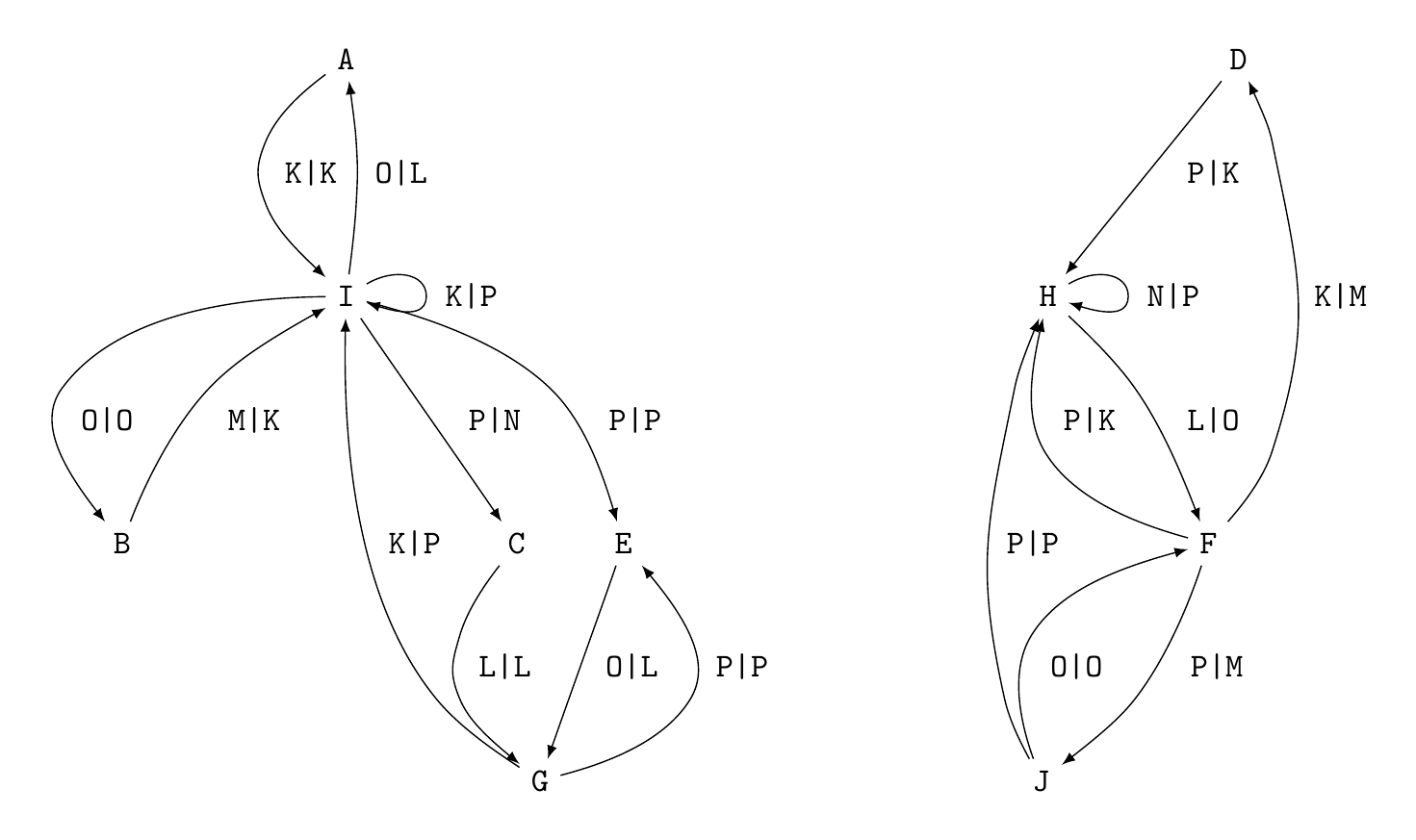}
\end{center}
\caption{The transducer $M_\U$. Each tile of $\U$ corresponds to a transition.
For example, the first tile of $\U$ is associated to $J\xrightarrow{O|O}F$,
    etc.}
\label{fig:T7transducer}
\end{figure}
Sequences $x$ and $y$ are in the relation $\rho(M_\T)$ if and only
if there exists a row of tiles, with matching vertical edges, whose bottom edges
form sequence $x$ and top edges sequence $y$. 
For example, the sequence of bottom colors seen on the
lowest row of the tiling shown at Figure~\ref{fig:U_image_5th_of4}
is $KOKPOKOKPOKPO$.
Starting in state $G$ and reading that word as input in the transducer
$M_\U$ we follow the transitions:
\[
G\xrightarrow{K|P}
I\xrightarrow{O|L}
A\xrightarrow{K|K}
I\xrightarrow{P|P}
E\xrightarrow{O|L}
G\xrightarrow{K|P}
I\xrightarrow{O|L}
A\xrightarrow{K|K}
I\xrightarrow{P|P}
E\xrightarrow{O|L}
G\xrightarrow{K|P}
I\xrightarrow{P|P}
E\xrightarrow{O|L}
G.
\]
We finish in state $G$ and we get $PLKPLPLKPLPPL$ as output which corresponds
to the sequence of top colors of the lowest row of the tiling at
Figure~\ref{fig:U_image_5th_of4}.
Then, starting in state $F$ and reading the previously output word
$PLKPLPLKPLPPL$ as input we compute the second row of the tiling at
Figure~\ref{fig:U_image_5th_of4}, etc. In general, there is a
one-to-one correspondence between valid tilings of the plane, and
the iterated execution of the transducer.

Notice that the result of the usual composition of transducers
corresponds to the transducer of $\T\boxminus^2\S$:
\begin{equation}\label{eq:boxminus_with_transducer}
    M_\T\circ M_{\S} = M_{\T\boxminus^2\S}.
\end{equation}
As done in \cite{jeandel_aperiodic_2015},
the result of the composition can be filtered by 
recursively removing any source or sink state from the transducer
reducing the size of the tile set.
This is very helpful for doing computations
(see the proof Lemma~\ref{lem:50tiles-2x2-in-OmegaU}).

\subsection{$d$-dimensional word}

In this section, we recall the definition of $d$-dimensional word that appeared
in \cite{MR2579856} and we keep the notation $u\odot^i v$ they proposed
for the concatenation. 

If $i\leq j$ are integers, then $\llbracket i, j\rrbracket$ denotes the interval of integers $\{i, i+1, \dots, j\}$.
Let $\bn=(n_1,\dots,n_d)\in\N^d$ and $\A$ be an alphabet.
We denote by $\A^{\bn}$ the set of functions
\begin{equation*}
    u:
\llbracket 0,n_1-1\rrbracket
\times
\cdots
\times
\llbracket 0,n_d-1\rrbracket
\to\A.
\end{equation*}
An element $u\in\A^\bn$ is called a
\emph{$d$-dimensional word $u$ of shape $\bn=(n_1,\dots,n_d)\in\N^d$}
on the alphabet~$\A$.
The set of all finite $d$-dimensional word is 
$\A^{*^d}=\{\A^\bn\mid\bn\in\N^d\}$.
A $d$-dimensional word of shape $\be_k+\sum_{i=1}^d\be_i$ is called a
\emph{domino in the direction $\be_k$}.
When the context is clear, we write $\A$ instead of $\A^{(1,\dots,1)}$.
When $d=2$, we represent a $d$-dimensional word $u$ of shape $(n_1,n_2)$ as a
matrix with Cartesian coordinates:
\begin{equation*}
    u=
    \left(\begin{array}{ccc}
        u_{0,n_2-1} &\dots   & u_{n_1-1,n_2-1} \\
        \dots   &\dots   & \dots \\
        u_{0,0} &\dots   & u_{n_1-1,0}
    \end{array}\right).
\end{equation*}
Let $\bn,\bm\in\N^d$ and $u\in\A^\bn$ and $v\in\A^\bm$.
If there exists an index $i$ such that the
shapes $\bn$ and $\bm$ are equal except at index $i$,
then the \emph{concatenation of $u$ and $v$ in the direction $\be_i$} 
is \emph{well-defined}: it is
the 
$d$-dimensional word $u\odot^i v$ of shape $(n_1,\dots,n_{i-1},n_i+m_i,n_{i+1},\dots,n_d)\in\N^d$
defined as
\begin{equation*}
    (u\odot^i v) (\ba) = 
\begin{cases}
    u(\ba)          & \text{if}\quad 0 \leq a_i < n_i,\\
    v(\ba-n_i\be_i) & \text{if}\quad n_i \leq a_i < n_i+m_i.
\end{cases}
\end{equation*}
If the shapes $\bn$ and $\bm$ are not equal except at
index $i$, we say that the concatenation of $u\in\A^\bn$ and $v\in\A^\bm$ in
the direction $\be_i$ is \emph{not well-defined}.

Let $\bn,\bm\in\N^d$ and $u\in\A^\bn$ and $v\in\A^\bm$.
We say that $u$ \emph{occurs in $v$ at position $\bp\in\N^d$} if
$v$ is large enough, i.e., $\bm-\bp-\bn\in\N^d$ and
\[
    v(\ba+\bp) = u(\ba)
\]
for all $\ba=(a_1,\dots,a_d)\in\N^d$ such that 
$0\leq a_i<n_i$ with $1\leq i\leq d$.
If $u$ occurs in $v$ at some position, then we say that $u$ is a
$d$-dimensional \emph{subword} or \emph{factor} of $v$.

\subsection{$d$-dimensional morphisms}
In this section, we generalize the definition of $d$-dimensional morphisms
\cite{MR2579856} to the case where the domain and codomain are different as for
$S$-adic systems.

Let $\A$ and $\B$ be two alphabets.
Let $X\subseteq\A^{*^d}$.
A function $\omega:X\to\B^{*^d}$ is a \emph{$d$-dimensional
morphism} if for every
$i$ with $1\leq i\leq d$,
and every $u,v\in X$ such that $u\odot^i v\in X$ is well-defined
we have
that the concatenation $\omega(u)\odot^i \omega(v)$
in direction $\be_i$ is well-defined and
\begin{equation*}
    \omega(u\odot^i v) = \omega(u)\odot^i \omega(v).
\end{equation*}
The next lemma can be deduced from the definition.
It says that when $d\geq 2$ every $d$-dimensional morphism
defined on the whole space $X=\A^{*^d}$ is uniform in the sense that it maps
every letter to a $d$-dimensional word of the same shape. These are called block-substitutions in \cite{2017_frank_introduction}.

\begin{lemma}
If $d\geq 2$ and $\omega:\A^{*^d}\to\B^{*^d}$ is a $d$-dimensional morphism,
then there exists a shape $\bn\in\N^d$ such that $\omega(a)\in\A^\bn$ for
every letter $a\in\A$.
\end{lemma}

Therefore, to consider non-uniform $d$-dimensional morphism when $d\geq 2$, we
need to restrict the domain to a strict subset $X\subsetneq\A^{*^d}$.
For example, 
let $d=2$ and $\A=\{a,b,x,y\}$.
The map $\omega:\A\to\A^{*^2}$ defined by
$\omega: a\mapsto aaa, b\mapsto bbb, x\mapsto xx, y\mapsto yy$
is not a well-defined $2$-dimensional morphism $\A^{*^2} \to\A^{*^2}$,
but 
it is a well-defined $2$-dimensional morphism $X\to\A^{*^2}$,
for the strict subset $X\subsetneq\A^{*^2}$ such that every column
of every $w\in X$ is over $\{a,b\}$ or $\{x,y\}$.
%\[
%    X = \left\{w\in\A^{*^2} \mid 
%        \forall i \{w_{ij}\mid j\in\Z\} = \{a,b\}
%        \text{ or }\{x,y\}\right\}
%\]
For instance, we have
\[
    \omega:
    \begin{array}{c|c|c|c}
        x & a & b & y\\
        y & a & a & y\\
        x & b & a & x\\
        y & b & b & y\\
        y & b & a & x\\
        x & a & b & x\\
    \end{array}
    \mapsto
    \begin{array}{c|c|c|c}
        xx & aaa & bbb & yy\\
        yy & aaa & aaa & yy\\
        xx & bbb & aaa & xx\\
        yy & bbb & bbb & yy\\
        yy & bbb & aaa & xx\\
        xx & aaa & bbb & xx\\
    \end{array}.
\]
A $d$-dimensional morphism 
$\omega:X \to\B^{*^d}$ 
with $X\subseteq\A^{*^d}$
can be extended to a $d$-dimensional morphism
$\omega:Y\to\B^{\Z^d}$
with $Y\subseteq\A^{\Z^d}$.

In \cite{MR2579856} and \cite[p.144]{MR1014984}, they consider the case $\A=\B$
and they restrict the domain of $d$-dimensional morphisms to the language they
generate.

Suppose now that $\A=\B$.
We say that a $d$-dimensional morphism $\omega:\A\to\A^{*^d}$ is
\emph{expansive}
if for every $a\in\A$ and $K\in\N$,
there exists $m\in\N$ such that 
$\min(\shape(\omega^m(a)))>K$.
We say that $\omega$ is \emph{primitive}
if there exists $m\in\N$ such that
for every $a,b\in\A$ the letter $b$ occurs in $\omega^m(a)$.

The definition of prolongable substitutions \cite[Def.
1.2.18--19]{MR2742574} can be adapted in the case of $d$-dimensional morphisms.
Let $d$-dimensional morphism $\omega:X \to\A^{*^d}$ with $X\subseteq\A^{*^d}$.
Let $s\in\{+1,-1\}^d$.
We say that $\omega$ is \emph{prolongable on letter $a\in\A$} in the
hyperoctant of sign $s$ if the letter $a$ appears in the appropriate corner of
its own image $\omega(a)$ more precisely at position
$p=(p_1,\dots,p_d)\in\N^d$ where
\[
p_i = 
\begin{cases}
0       & \text{if}\quad s_i = +1,\\
n_i-1   & \text{if}\quad s_i = -1.
\end{cases}
\]
where $\bn=(n_1,\dots,n_d)\in\N^d$ is the shape of $\omega(a)$.
If $\omega$ is prolongable on letter $a\in\A$ in the hyperoctant of
sign $s$ 
and if
$\lim_{m\to\infty}\min(\shape(\omega^m(a)))=\infty$, then
$\lim_{m\to\infty}\omega^m(a)$ is a well-defined $d$-dimensional infinite word
$s_0\N\times\dots\times s_{d-1}\N\to\A$.

\subsection{$d$-dimensional language}

A subset $L\subseteq\A^{*^d}$ is called a $d$-dimensional \emph{language}. The
\emph{factorial closure} of a language $L$ is
\begin{equation*}
    \overline{L}^{Fact}
    = \{u\in\A^{*^d} \mid u\text{ is a $d$-dimensional subword of some } 
                           v\in L\}.
\end{equation*}
A language $L$ is \emph{factorial} if $\overline{L}^{Fact}=L$.
All languages considered in this contribution are factorial.
Given a tiling $x\in\A^{\Z^d}$, the \emph{language} $\L(x)$ defined by $x$ is
\begin{equation*}
    \L(x) = \{u\in\A^{*^d} \mid u\text{ is a $d$-dimensional subword of } x\}.
\end{equation*}
The \emph{language} of a subshift $X\subseteq\A^{\Z^d}$ is
    $\L_X = \cup_{x\in X} \L(x)$.
Conversely, given a factorial language $L\subseteq\A^{*^d}$ we define the subshift
\begin{equation*}
    \X_L = \{x\in\A^{\Z^d}\mid \L(x)\subseteq L\}.
\end{equation*}
A language $L\subseteq\A^{*^d}$ is \emph{forbidden} in a subshift
$X\subset\A^{\Z^d}$ if $L\cap\L_X=\varnothing$.
By extension, a $d$-dimensional subword $u\in\A^{*^d}$ is \emph{forbidden} in a
subshift $X\subset\A^{\Z^d}$ if the singleton language $\{u\}$ is forbidden in
$X$.

We say that a word $u\in\A^\bn$,
with $\bn=(n_1,\dots,n_d)\in\N^d$,
admits a \emph{surrounding of radius $r\in\N$}
in a language $L\subseteq\A^{*^d}$ (resp. in a subshift $X\subset\A^{\Z^d}$)
if there exists $w\in\A^{\bn+2(r,\dots,r)}$ 
such that $w\in L$ (resp. $w\in\L_X$) and
$u$ occurs in $w$ at position $(r,\dots,r)$.

Now we consider notions of languages related to morphisms.
Given a $d$-dimensional morphism $\omega:\A \to\B^{*^d}$
and a language $L\subseteq\A^{*^d}$ of $d$-dimensional words, then
we define the image of the language $L$ under $\omega$ as
the language
\begin{equation*}
\overline{\omega(L)}^{Fact}
    = \{u\in\B^{*^d} \mid u\text{ is a $d$-dimensional subword of }
                  \omega(v) \text{ with } v\in L\}
    \subseteq \B^{*^d}
\end{equation*}
and the image of a subshift $X\subseteq\A^{\Z^d}$ under $\omega$
as the subshift
\begin{equation*}
\overline{\omega(X)}^{\sigma}
    = \{\sigma^\bk\omega(x)\in\B^{\Z^d} \mid \bk\in\Z^d, x\in X\}
    \subseteq \B^{\Z^d}.
\end{equation*}
The next lemma shows that the fact that a subshift is the image of another
subshift under a $d$-dimensional morphism can be restated equivalently in terms
of their languages.

\begin{lemma}\label{lem:existence-omega-representation}
    Let $\omega:X\to Y$ be a $d$-dimensional morphism for some subshifts
    $X\subseteq\A^{\Z^d}$ and $Y\subseteq\B^{\Z^d}$. The following conditions
    are equivalent:
\begin{enumerate}[\rm (i)]
    \item $Y=\overline{\omega(X)}^{\sigma}$,
    \item $\L_Y=\overline{\omega(\L_X)}^{Fact}$.
\end{enumerate}
\end{lemma}

\begin{proof}
    (i) $\implies$ (ii)
    By definition of $\omega$, we have $\omega(\L_X)\subseteq\L_Y$
    and $\overline{\omega(\L_X)}^{Fact}\subseteq\L_Y$ since $\L_Y$ is a
    factorial language.
    Let $u\in\L_Y$, i.e., 
    $u$ is a factor of some word $y\in Y$.
    There exists $\bk\in\Z^d$ and $x\in X$ such that $y=\sigma^\bk\omega(x)$.
    Therefore $u$ is a subword of $\omega(v)$ for some $v\in\L_X$
    and we conclude $\L_Y \subseteq \overline{\omega(\L_X)}^{Fact}$.

    (ii) $\implies$ (i)
    We have $\omega(X)\subset Y$.
    Since $Y$ is shift-invariant, we also have
    $\overline{\omega(X)}^{\sigma}\subseteq Y$.
    Let $y\in Y$. Thus $y$ is in the subshift generated by the language
    $\L_Y=\overline{\omega(\L_X)}^{Fact}$.
    %$\L_Y=\omega(\L_X)$. 
    Since $X$ is closed, there exists $x\in X$ and
    $\bk\in\Z^d$ such that $y=\sigma^\bk\omega(x)$.
    Therefore
    $Y\subseteq\overline{\omega(X)}^{\sigma}$.
\end{proof}

\subsection{Self-similar subshifts}

In this section we consider languages and subshifts defined from substitutions
leading to self-similar structures.
A subshift $X\subseteq\A^{\Z^d}$ (resp. a language $L\subseteq\A^{*^d}$)
is \emph{self-similar}
if there exists an expansive
$d$-dimensional morphism $\omega:\A\to\A^{*^d}$ such that
$X=\overline{\omega(X)}^\sigma$
(resp.  $L=\overline{\omega(L)}^{Fact}$).

Self-similar languages and subshifts can be constructed by iterative
application of the morphism $\omega$ starting with the letters.
The \emph{language} $\L_\omega$ defined by an expansive $d$-dimensional
morphism $\omega:\A \to\A^{*^d}$ is
\begin{equation*}
    \L_\omega = \{u\in\A^{*^d} \mid u\text{ is a $d$-dimensional subword of }
    \omega^n(a) \text{ for some } a\in\A\text{ and } n\in\N \}.
\end{equation*}
It satisfies
$\L_\omega=\overline{\omega(\L_\omega)}^{Fact}$
and thus is self-similar.
The \emph{substitutive shift} $\X_\omega=\X_{\L_\omega}$
defined from the language of $\omega$ is a self-similar subshift.
If $\omega$ is primitive then $\X_\omega$ is the smallest nonempty subshift
$X\subseteq\A^{\Z^d}$ satisfying $X=\overline{\omega(X)}^{\sigma}$.

\begin{remark}\label{rmk:special-case}
    In some references, the definition of self-similar is more restrictive.
For example, in \cite[p.268]{MR1637896},
the definition of self-similarity for tilings of $\R^d$ 
would be restated in terms of tilings of $\Z^d$ as:
    a subshift $X\subseteq\A^{\Z^d}$ is \emph{self-similar} if there exists an
    expansive $d$-dimensional morphism $\omega:\A\to\A^{*^d}$ such that $X =
    \X_\omega$.
    Observe that a nonempty subshift $X$ that satisfies
    $X=\overline{\omega(X)}^{\sigma}$ for some expansive and primitive
    $d$-dimensional morphism $\omega$ contains $\X_\omega$ but as illustrated
    in the following example, it needs not be~$\X_\omega$.

Consider the $1$-dimensional morphism $\mu:a\mapsto ca, b\mapsto bc, c\mapsto
    cbac$ over the alphabet $\A=\{a,b,c\}$. It is expansive and primitive.
    Consider the language
    \[
    L = \{u\in\A^{*^d} \mid u\text{ is a $d$-dimensional subword of }
    \mu^n(ab) \text{ for some } n\in\N \}
    \]
    that satisfies $L\supseteq\L_\omega$ and 
    $L=\overline{\omega(L)}^{Fact}$.
    We have $ab\in L$ but $ab\notin\L_\omega$ thus $L\supsetneq\L_\omega$.
\end{remark}

The next lemma shows that there is a unique language $L$ satisfying
    $L=\overline{\omega(L)}^{Fact}$
    and a unique subshift $X$ satisfying
    $X=\overline{\omega(X)}^{\sigma}$
whenever $\omega$ is primitive and 
the words of shape $(2,\dots,2)$ of $L$ 
or of $\L_X$ are in $\L_\omega$.

\begin{lemma}\label{lem:substitutive-equivalent-conditions}
    Let $\omega:\A\to\A^{*^d}$ be an expansive and primitive $d$-dimensional
    morphism. Let $L\subseteq\A^{*^d}$ be a $d$-dimensional language
    such that $L=\overline{\omega(L)}^{Fact}$. Then $\L_\omega\subseteq L$ and
    the following conditions are equivalent:
\begin{enumerate}[\rm (i)]
\item $L=\L_\omega$,
\item $\X_L$ is minimal,
\item $\X_L=\X_\omega$,
\item $L\cap\A^{(2,\dots,2)} = \L_\omega\cap\A^{(2,\dots,2)}$.
\end{enumerate}
\end{lemma}

\begin{proof}
    ($\L_\omega\subseteq L$). 
    Notice that, recursively, $L=\overline{\omega^m(L)}^{Fact}$ for every $m\geq1$.  For every
    $a\in\A$ and $m\geq1$, the $d$-dimensional word $\omega^m(a)$ is in the
    language $L$.

    (i) $\implies$ (ii). 
    We have that $\X_{\L_\omega}=\X_\omega$ is the substitutive shift of
    $\omega$ which is minimal since $\omega$ is primitive.
    (ii) $\implies$ (iii).
    We know that $\L_\omega\subseteq L$.
    Therefore $X_\omega\subseteq\X_L$.
    Since $\X_L$ is minimal, we conclude that $X_\omega=\X_L$.
    (iii) $\implies$ (iv).
    If $X_\omega=\X_L$ then $L=\L_\omega$.
    Therefore, $L\cap\A^{(2,\dots,2)} = \L_\omega\cap\A^{(2,\dots,2)}$.

    (iv) $\implies$ (i).
    ($L\subseteq\L_\omega$). 
    Let $z\in L$. Since $\omega$ is expansive,
    let $m\in\N$ such that the image of every letter
    $a\in\A$ by $\omega^m$ is larger than $z$, that is, 
    $\shape(\omega^m(a))\geq\shape(z)$ for all $a\in\A$.
    We have $z\in\overline{\omega^m(L)}^{Fact}$.
    By the choice of $m$, $z$ can not overlap more than two block
    $\omega^m(a)$ in the same direction. Then, there exists a word $u\in L$ of
    shape $(2,\dots,2)$ such that $z$ is a subword of $\omega^m(u)$.
    From the hypothesis, every word of shape $(2,\dots,2)$ that appear in $L$
    also appear in $\L_\omega$. Therefore $u\in\L_\omega$.
    Since $\omega$ is primitive, there exists $\ell$ such that every word $u$
    of shape $(2,\dots,2)$ that appear in $\L_\omega$ 
    is also a subword of $\omega^\ell(a)$ for every $a\in\A$.
    Therefore, $z$ is a subword of $\omega^{m+\ell}(a)$ for every $a\in\A$.
    Then $z\in\L_\omega$ and $L\subseteq\L_\omega$.
\end{proof}

\subsection{$d$-dimensional recognizability and aperiodicity}\label{sec:recognizability}

The definition of recognizability dates back to the work of Host, Quéffelec and
Mossé \cite{MR1168468}. See also \cite{akiyama_mosse_2017} who propose a
completion of the statement and proof for B. Mossé's unilateral recognizability
theorem.
The definition introduced below is based on work of Berthé, Steiner and
Yassawi \cite{berthe_recognizability_2017} on the recognizability in the case
of $S$-adic systems where more than one substitutions are involved.

Let $X\subseteq\A^{\Z^d}$ and
$\omega:X\to\B^{\Z^d}$ be a $d$-dimensional morphism.
If $y\in\overline{\omega(X)}^{\sigma}$, i.e.,
$y=\sigma^\bk\omega(x)$ for some $x\in X$ and $\bk\in\Z^d$, where $\sigma$ is
the $d$-dimensional shift map, we say that $(\bk,x)$ is a
\emph{$\omega$-representation of $y$}. We say that it is \emph{centered} if
$y_\zero$ lies inside of the image of $x_\zero$, i.e., if
$\zero\leq\bk<\shape(\omega(x_\zero))$ coordinate-wise.
We say that $\omega$ is \emph{recognizable in $X\subseteq\A^{\Z^d}$}
if each $y\in\B^{\Z^d}$ has at most one centered $\omega$-representation 
$(\bk,x)$ with $x\in X$.

% We say that $\omega$ is \emph{fully recognizable} 
% if it is recognizable in $\A^{\Z^d}$.

The unique composition property and conditions \ref{cond:C1} and
\ref{cond:C2} can be stated in terms of subshifts, $d$-dimensional morphisms
and recognizability.

\begin{proposition}\label{prop:expansive-recognizable-aperiodic}
    Let $\omega:\A\to\A^{*^d}$ be an expansive $d$-dimensional morphism.
    Let $X\subseteq\A^{\Z^d}$ be a self-similar subshift such that
    $\overline{\omega(X)}^\sigma=X$.
    If $\omega$ is recognizable in $X$, then $X$ is aperiodic.
\end{proposition}

\begin{proof}
    Suppose that there exists a periodic tiling $y\in X$ with
    period $\bp\in\Z^d\setminus\{(0,0)\}$ satisfying $\sigma^\bp y=y$.
    Since $\omega$ is expansive,
    let $m\in\N$ such that the shape of the image of every letter
    $a\in\A$ by $\omega^m$ is large enough, that is,
    $\shape(\omega^m(a))\geq\bp$ for every
    letter $a\in\A$.
    By hypothesis,
    every $y\in X$ has a $\omega$-representation. 
    Recursively, there
    exists a $\omega^m$-representation $(\bk,x)$ of $y$ satisfying
    $y=\sigma^\bk\omega^m(x)$.
    We may assume that it is centered since $X$ is shift-invariant.
    By definition of centered representation,
    for every $\bu\in\Z^d$ such that $\zero\leq\bu<\shape(\omega^m(x_\zero))$,
    $(\bu,x)$ is a centered $\omega^m$-representation of 
    $\sigma^\bu\omega^m(x)=\sigma^{\bu-\bk}y$.
    By the choice of $m$, there exists $\bu\in\Z^d$ such that
    $\zero\leq\bu<\shape(\omega^m(x_\zero))$
    and
    $\zero\leq\bu+\bp<\shape(\omega^m(x_\zero))$.
    Therefore
    $(\bu,x)$ is a centered $\omega^m$-representation of 
    $\sigma^\bu\omega^m(x)=\sigma^{\bu-\bk}y$
    and
    $(\bu+\bp,x)$ is a centered $\omega^m$-representation of 
    $\sigma^{\bu+\bp}\omega^m(x)
    =\sigma^{\bu+\bp-\bk}y
    =\sigma^{\bu-\bk}\sigma^\bp y
    =\sigma^{\bu-\bk}y$.
    Therefore, $\omega^m$ is not recognizable
    which implies that $\omega$ is not recognizable
    which is a contradiction.
    We conclude that there is no periodic tiling $y\in X$.
\end{proof}

Note that the converse of
Proposition~\ref{prop:expansive-recognizable-aperiodic} for tilings was proved
in \cite{MR1637896} as a generalization of a result of Mossé \cite{MR1168468}
who showed that recognizability and aperiodicity are equivalent for primitive
substitutive sequences.

\begin{lemma}\label{lem:aperiodic-implies-aperiodic}
    Let $\omega:X\to Y$ be some $d$-dimensional morphism between two
    subshifts $X$ and $Y$.
%\begin{enumerate}[\rm (i)]
%\item 
    If $Y$ is aperiodic, then $X$ is aperiodic.
%\item If $X$ is aperiodic and $\omega$ is recognizable in $X$,
%    then $\overline{\omega(X)}^\sigma$ is aperiodic.
%\end{enumerate}
\end{lemma}

\begin{proof}
    %(i) 
    If $X$ contains a periodic tiling $x$, then
    $\omega(x)\in Y$ is periodic.
%
%    (ii)
%    Let $y\in\overline{\omega(X)}^\sigma$.
%    Then, there exist $\bk\in\Z^d$ and $x\in X$ such that 
%    $(\bk, x)$ is a centered $\omega$-representation of $y$, i.e.,
%    $y=\sigma^\bk\omega(x)$.
%    Suppose by contradiction that $y$ has a nontrivial period
%    $\bp\in\Z^d\setminus\zero$.
%    Since $y =\sigma^\bp y =\sigma^{\bp+\bk}\omega(x)$,
%    we have that
%    $(\bp+\bk, x)$ is a $\omega$-representation of $y$.
%    Since $\omega$ is recognizable, this representation is not centered.
%    Therefore there exists $\bq\in\Z^d\setminus\zero$ such that
%    $y_\zero$ lies in the image of $x_\bq=(\sigma^\bq x)_\zero$.
%    Therefore there exists $\bk'\in\Z^d$ such that
%    $(\bk',\sigma^\bq x)$ is a centered $\omega$-representation of $y$.
%    Since $\omega$ is recognizable, we conclude that
%    $\bk=\bk'$ and $x=\sigma^\bq x$. Then $x\in X$ is periodic which is a
%    contradiction.
\end{proof}

\section{A sufficient condition for recognizability and surjectivity}
\label{sec:sufficient-condition}

The goal of this section is to show that under some hypothesis made on a Wang
tile set $\T$, namely the existence of markers, there exists another set $\S$
of Wang tiles and a nontrivial recognizable $2$-dimensional morphism
$\Omega_\S\to\Omega_\T$ that is onto up to a shift.  More precisely, every Wang
tiling by $\T$ is up to a shift the image under a nontrivial $2$-dimensional
morphism $\omega$ of a unique Wang tiling in $\Omega_\S$.

\subsection{A sufficient condition for recognizability}

The next lemma provides a sufficient condition for recognizability of
$d$-dimensional morphism. It is weak in the sense that it applies to morphism
whose images of letters are letters or a domino in a given direction, but it is
sufficient for our needs.

\begin{lemma}\label{lem:recognizable}
    Let $d\geq1$ and $i$ such that $1\leq i\leq d$.
Let $X\subseteq\A^{\Z^d}$ and
$\omega:X\to\B^{\Z^d}$ be a $d$-dimensional morphism such that
the image of letters are letters or dominoes in the direction $\be_i$.
If $\omega|_\A$ is injective and there exists a subset $M\subset \B$ such that
\begin{equation}\label{eq:markers-on-right}
    \omega(\A)
    \subseteq (\B\setminus M)
    \cup \left((\B\setminus M)\odot^i M\right),
\end{equation}
or
\begin{equation}\label{eq:markers-on-left}
    \omega(\A)
    \subseteq (\B\setminus M)
    \cup \left(M\odot^i (\B\setminus M)\right),
\end{equation}
then $\omega$ is recognizable in $X$.
\end{lemma}

Informally, the letters $m\in M$ have the role of \emph{markers}: when we see
them in the image under $\omega$, we know that they must appear as the left
(resp. right) part of a domino.
A formal definition of markers is given in the next section
(Definition~\ref{def:markers}).

\begin{proof}
    Let $(\bk,x)$ and $(\bk',x')$ be two
    centered $\omega$-representations of $y\in\B^{\Z^d}$ with
    $\bk,\bk'\in\Z^d$ and $x,x'\in X$.
    We want to show that they are equal.

    Since the image of a letter under $\omega$ is a letter or a domino in the
    direction $\be_i$, then $\bk,\bk'\in\{\zero,\be_i\}$.
    If $y_\zero\in M$, then $y_\zero$ appears as the left or right part of a
    domino and thus $\bk=\bk'=\be_i$
    if Equation~\eqref{eq:markers-on-right} holds
    or $\bk=\bk'=\zero$
    if Equation~\eqref{eq:markers-on-left} holds.

    Suppose now that $y_\zero\in \B\setminus M$. 
    If Equation~\eqref{eq:markers-on-right} holds,
    then $\bk=\bk'=\zero$.
    Suppose that Equation~\eqref{eq:markers-on-left} holds.
    By contradiction, suppose that $\bk\neq\bk'$ 
    and assume without lost of generality that
    $\bk=\zero$ and $\bk'=\be_i$.
    This means that $\omega(x'_\zero)=y_{-\be_i} \odot^i y_\zero$ is
    a domino in the direction $\be_i$. 
    Since $y_\zero\in \B\setminus M$, we must have that $y_{-\be_i}\in M$
    is a left marker.
    This is impossible as
    $\omega(x_{-\be_i})=y_{-\be_i}\in \B\setminus M$
    or
    $\omega(x_{-\be_i})=y_{-2\be_i} \odot^i y_{-\be_i}
    \in M\odot^i (\B\setminus M)$.
    Therefore, we must have $\bk=\bk'$ and
    $\omega(x)=\omega(x')$.

    Suppose by contradiction that $x\neq x'$.
    Let $\ba=(a_1,\dots,a_d)\in\Z^d$ be some minimal vector with respect to
    $\Vert\ba\Vert_\infty$ such that $x_\ba\neq x'_\ba$.
    Injectivity of $\omega$ implies that $\omega(x_\ba)$ and
    $\omega(x'_\ba)$ must have different shapes.
    Suppose without lost of generality that
    $\omega(x_\ba)\in\B $ and
    $\omega(x'_\ba)=b\odot^i c\in\B\odot^i\B$.
    We need to consider two cases: $a_i\geq 0$ and $a_i<0$.

    Suppose $a_i\geq 0$.
    We must have that Equation~\eqref{eq:markers-on-right} holds.
    We have $\omega(x_\ba)=b\in\B\setminus M$ and $c\in M$.
    But then $\omega(x_{\ba+\be_i})=c$
    or $\omega(x_{\ba+\be_i})=c\odot^i d$ for some 
    $c\in\B\setminus M$ and $d\in\B$ which is a contradiction.

    Suppose $a_i<0$.
    We must have that Equation~\eqref{eq:markers-on-left} holds.
    We have $\omega(x_\ba)=c\in\B\setminus M$ and $b\in M$.
    But then $\omega(x_{\ba-\be_i})=b$
    or $\omega(x_{\ba+\be_i})=d\odot^i b$ for some 
    $b\in\B\setminus M$ and
    $d\in\B$ which is a contradiction.
    We conclude that $x=x'$.
\end{proof}

\subsection{A sufficient condition for surjectivity up to a shift}

% Suppose in a Wang tiling by the tile set $\T$, there exists a partition
% $\T=K\sqcup L\sqcup R$ and a direction $\be_i\in\{\be_1,\be_2\}$ such that that every time
% some tile in $R$ appears at position $\ba\in\Z^2$ the tile at position
% $\ba-\be_i$ is always in $\T\setminus R$ and that every time some tile in $L$
% appears at position $\ba\in\Z^2$ the tile at position $\ba+\be_i$ is always in
% $\T\setminus L$.

Theorem~\ref{thm:exist-homeo} is the key result in this contribution
since it is used in Section~\ref{sec:desubstitionofU} to desubstitute tilings
in $\Omega_\U$ 
into tilings in $\Omega_\V$ for some Wang tile set $\V$
and in
Section~\ref{sec:desubstitionofV} to desubstitute tilings of $\Omega_\V$
into tilings in $\Omega_\W$ for some Wang tile set $\W$.
It can be seen as a $2$-dimensional generalization of the notion of derived
sequence introduced in \cite{MR1489074} for the $1$-dimensional substitutive
case. While for the $1$-dimensional substitutive
case the derived sequence is obtained by 
considering the return words to a single letter, here it gives sufficient
condition for the existence of a derived tiling by considering the return
words to a \emph{subset of letters} $\T\setminus M$ for some
$M\subset\T$.

\begin{definition}\label{def:markers}
    Let $\T$ be a Wang tile set 
    and let $\Omega_\T$ be its Wang shift.
    A nonempty proper subset $M\subset\T$ is called \emph{markers in the
    direction $\be_1$} if
    \begin{equation}
        M\odot^1 M, \qquad
        M\odot^2 (\T\setminus M), \qquad
        (\T\setminus M) \odot^2 M
    \end{equation}
    are forbidden in $\Omega_\T$.
    It is called \emph{markers in the direction $\be_2$} if
    \begin{equation}
        M\odot^2 M, \qquad
        M\odot^1 (\T\setminus M), \qquad
        (\T\setminus M) \odot^1 M
    \end{equation}
    are forbidden in $\Omega_\T$.
\end{definition}

The markers in the direction $\be_1$ (resp. $\be_2$) appear as nonadjacent
columns (resp. rows) of tiles in a tiling.
The existence of markers allows to desubstitute tilings uniquely by
$2$-dimensional morphism that are essentially $1$-dimensional.

\begin{theorem}\label{thm:exist-homeo}
    Let $\T$ be a Wang tile set 
    and let $\Omega_\T$ be its Wang shift.
    If there exists a subset
    $M\subset\T$ 
    of markers in the direction 
    $\be_i\in\{\be_1,\be_2\}$,
    then there exists
    a Wang tile set $\S$
    and a $2$-dimensional morphism
    $\omega:\Omega_\S\to\Omega_\T$
    such that 
    \begin{equation}
        \omega(\S)\subseteq (\T\setminus M)\cup 
        \left((\T\setminus M)\odot^i M\right)
    \end{equation}
    which is recognizable in the Wang shift $\Omega_\S$
    and surjective up to a shift, i.e.,
    $\omega(\Omega_\S)\cup\sigma^{\be_i}\omega(\Omega_\S)=\Omega_\T$.
%The substitution $\omega$ can be computed using
%the method below.
%Algorithm~\ref{alg:find-recognizable-sub-from-markers}.
    %Then $\omega:\Omega_\S\to\Omega_\T$ is recognizable in $\Omega_\S$
    %and $\omega(\Omega_\S)\cup\sigma^{\be_i}\omega(\Omega_\S)=\Omega_\T$.
\end{theorem}

\begin{proof}
    For every $r\in\N$, we define the set of dominoes in the direction
    $\be_i$ that admits a surrounding of radius $r$ in $\Omega_\T$:
    \[
    D_{\T,\be_i,r}
        = \left\{u\odot^i v \,\middle|\, u\odot^iv \text{ admits a surrounding
        of radius $r$ in } \Omega_\T\right\}.
    \]
    Notice that since $M$ is a set of markers, there exists $r\in\N$ such that
    $D_{\T,\be_i,r}\cap(M\odot^i M)=\varnothing$.
    Now let $r\in\N$ be any nonnegative integer and
    let $P_r$, $K_r$ be the following sets:
    \begin{align*}
        P_r &= \left\{u\boxminus^i v \,\middle|\, 
            u\in\T\setminus M, v\in M 
            \text{ and } 
            u\odot^i v \in D_{\T,\be_i,r}\right\},\\
        K_r &= \left\{u\in\T\setminus M\,\middle|\, \text{ there exists }
        v\in\T\setminus
            M \text{ such that } u\odot^iv \in D_{\T,\be_i,r}\right\}.
    \end{align*}
    Let $\S=K_r\cup P_r$ be a Wang tile set
    and $\Omega_\S$ be its Wang shift.
    Let $\omega:\Omega_\S\to\Omega_\T$ be the $2$-dimensional morphism
    defined by
    \begin{equation}\label{eq:omegaStoT}
        \omega:
    \begin{cases}
        u\boxminus^i v
        \quad \mapsto \quad
        u\odot^i v
        & \text{ if }
        u\boxminus^i v \in P_r,\\
        u \quad \mapsto \quad
        u & \text{ if } u \in K_r.
    \end{cases}
    \end{equation}

    From now on, we suppose that the markers $M$ are in the direction
    $\be_i=\be_1$, the argument being symmetric for $i=2$.
    First we show that $\omega$ is well-defined.
    Let $x\in\Omega_\S$ be a Wang tiling.
    We show that $\omega(x)$ is a valid Wang tiling.
    It is sufficient to prove it for all dominoes $p\odot^jq$ that appear in
    $\omega(x)$ with $p,q\in\T$ and $j\in\{1,2\}$.
    Any domino $p\odot^jq$ that appears in $\omega(x)$ appears in the
    image under $\omega$ of a letter or
    of a domino in $x$.
    The first case can happen only if $j=1$ and
    there exists $u\boxminus^1v\in P_r$ such that
    $p\odot^1q=u\odot^1v=\omega(u\boxminus^1v)$, in which case
    the right color of $p$ is equal to the left color of $q$ since
    $r(p)=r(u)=\ell(v)=\ell(q)$.
    In the second case, $p$ and $q$ belong to the image of distinct letters.
    The case $j=1$ (horizontal dominoes) is easy since
    $\omega$ preserve the right and left colors of 
    tiles, i.e., $r(\omega(s))=r(s)$
    and $\ell(\omega(s))=\ell(s)$ for all $s\in\S$
    so that $r(p)=\ell(q)$.
    Consider now $j=2$ (vertical dominoes).
    Any vertical domino $p\odot^2q$ that appears in $\omega(x)$ appears in the
    image under $\omega$ of a vertical domino $p'\odot^2q'$ in $x$.
    There are two cases to consider.
    If $p',q'\in K_r$, then
    $t(p)=t(\omega(p'))=t(p')=b(q')=b(\omega(q'))=b(q')$.
    If $p',q'\in P_r$, 
    then there exists $a,b,c,d\in\T$ such that
    $p\odot^2q$ appears in $(a\odot^1b)\odot^2(c\odot^1d)=\omega(p'\odot^2q')$.
    Then 
    \[
        (t(a),t(b)) = 
        t(a\odot^1b) = 
        t(\omega(p'))=t(p')=b(q')=b(\omega(q'))
        =b(c\odot^1d)
        =(b(c), b(d)).
    \]
    Then $t(a)=b(c)$ and $t(b)=b(d)$
    from which we conclude that $t(p)=b(q)$.
    Therefore $\omega(x)$ is a valid Wang tiling of the plane and belongs to
    $\Omega_\T$.

    Now we show that $\omega$ is surjective up to a shift.
    Let $y\in\Omega_\T$ be a Wang tiling which can be seen as a function
    $y:\Z^2\to\T$.
    Consider the set $y^{-1}(M)\subset\Z^2$ of positions of marker tiles
    in~$y$. From the definition of markers in the direction $\be_i=\be_1$, 
    markers appear in nonadjacent columns in a tiling.
    Formally, there exists
    a set $H\subset\Z$ such that $y^{-1}(M)=H\times\Z$ and
    $1\notin H-H$.
    Since $1\notin H-H$, there exists a strictly increasing sequence
    $(a_k)_{k\in\Z}$ such that $\Z\setminus H=\{a_k\mid k\in\Z\}$.
    We assume that $a_0=0$ if $0\in\Z\setminus H$ and $a_0=-1$ if $0\in H$
    which makes the sequence $(a_k)_{k\in\Z}$ uniquely defined.
    Let
\begin{equation*}
\begin{array}{rccl}
    x:&\Z^2 & \to & \S\\
    &(k,\ell) & \mapsto &
\begin{cases}
    y(a_k,\ell)                          & \text{ if } a_k+1\notin H,\\
    y(a_k,\ell)\boxminus^1 y(a_k+1,\ell) & \text{ if } a_k+1\in H.
\end{cases}
\end{array}
\end{equation*}
    The function $x$ is well-defined since $x(k,\ell)\in\S$ for all
    $(k,\ell)\in\Z^2$.
    Indeed, if $a_k+1\notin H$, then $x(k,\ell)= y(a_k,\ell)\in\T\setminus M$.
    Also $y(a_k+1,\ell)\in\T\setminus M$.
    Since $y$ is a valid Wang tiling $y(a_k,\ell)\odot^1 y(a_k+1,\ell)$ admits
    arbitrarily large surrounding in $\Omega_\T$.
    Therefore, $x(k,\ell)=y(a_k,\ell)\in K_\infty\subseteq K_r\subset\S$.
    Suppose now that $a_k+1\in H$.
    Then $x(k,\ell)=y(a_k,\ell)\boxminus^1 y(a_k+1,\ell)\in(\T\setminus
    M)\boxminus^1 M$.
    Since $y$ is a valid Wang tiling $x(k,\ell)\in P_\infty\subseteq
    P_r\subset\S$.
    Moreover, we deduce that $x\in\Omega_\S$, that is it is a valid Wang tiling
    of the plane, from the fact that $y\in\Omega_\T$ is a valid Wang tiling of
    the plane.

    We may now finish the proof of surjectivity.
If $0\notin H$, 
then the tiling $y$ is exactly the image under $\omega$ of the tiling $x$ that
we constructed: $y=\omega(x)$.
If $0\in H$, then 
the tiling $y$ is a shift of the image under $\omega$ of the tiling $x$:
$y=\sigma^{\be_1}\omega(x)$.
Thus $\omega$ is surjective up to a shift.

The function $\omega$ is of the form
\begin{equation*}
    \omega(\S)
    \subseteq (\T\setminus M)
    \cup \left((\T\setminus M)\odot^1 M\right)
\end{equation*}
and its restriction on $\S$ is injective by construction.
    Therefore, we conclude from Lemma~\ref{lem:recognizable} that
    $\omega$ is recognizable in $\Omega_\S$.
\end{proof}

In Equation~\eqref{eq:omegaStoT}, given two Wang tiles $u$ and $v$ such that
$u\boxminus^i v$ is well-defined for $i\in\{1,2\}$, we consider a
$2$-dimensional morphism of the form
\begin{equation*}
u\boxminus^i v
\quad
\mapsto
\quad
u\odot^i v
\end{equation*}
which can be seen as
\begin{center}
\begin{tikzpicture}
[scale=0.900000000000000]
% tile at position (x,y)=(0, 0)
\draw (0, 0) -- (1, 0);
\draw (0, 0) -- (0, 1);
\draw (1, 1) -- (1, 0);
\draw (1, 1) -- (0, 1);
\node[rotate=0,font=\footnotesize] at (0.8, 0.5) {Z};
\node[rotate=0,font=\footnotesize] at (0.5, 0.8) {BD};
\node[rotate=0,font=\footnotesize] at (0.2, 0.5) {X};
\node[rotate=0,font=\footnotesize] at (0.5, 0.2) {AC};
\node at (1.5,.5) {$\mapsto$};
\node at (3.0,0.5) {\begin{tikzpicture}[scale=0.900000000000000]
% tile at position (x,y)=(0, 0)
\draw (0, 0) -- (1, 0);
\draw (0, 0) -- (0, 1);
\draw (1, 1) -- (1, 0);
\draw (1, 1) -- (0, 1);
\node[rotate=0,font=\footnotesize] at (0.8, 0.5) {Y};
\node[rotate=0,font=\footnotesize] at (0.5, 0.8) {B};
\node[rotate=0,font=\footnotesize] at (0.2, 0.5) {X};
\node[rotate=0,font=\footnotesize] at (0.5, 0.2) {A};
% tile at position (x,y)=(1, 0)
\draw (1, 0) -- (2, 0);
\draw (1, 0) -- (1, 1);
\draw (2, 1) -- (2, 0);
\draw (2, 1) -- (1, 1);
\node[rotate=0,font=\footnotesize] at (1.8, 0.5) {Z};
\node[rotate=0,font=\footnotesize] at (1.5, 0.8) {D};
\node[rotate=0,font=\footnotesize] at (1.2, 0.5) {Y};
\node[rotate=0,font=\footnotesize] at (1.5, 0.2) {C};
\end{tikzpicture}};
\end{tikzpicture}
\qquad
    \raisebox{5mm}{\text{ or }}
\qquad
\begin{tikzpicture}
[scale=0.900000000000000]
% tile at position (x,y)=(0, 0)
\draw (0, 0) -- (1, 0);
\draw (0, 0) -- (0, 1);
\draw (1, 1) -- (1, 0);
\draw (1, 1) -- (0, 1);
\node[rotate=90,font=\footnotesize] at (0.8, 0.5) {YZ};
\node[rotate=0,font=\footnotesize] at (0.5, 0.8) {D};
\node[rotate=90,font=\footnotesize] at (0.2, 0.5) {XW};
\node[rotate=0,font=\footnotesize] at (0.5, 0.2) {A};
\node at (1.5,.5) {$\mapsto$};
\node at (2.5,0.5) {\begin{tikzpicture}[scale=0.900000000000000]
% tile at position (x,y)=(0, 0)
\draw (0, 0) -- (1, 0);
\draw (0, 0) -- (0, 1);
\draw (1, 1) -- (1, 0);
\draw (1, 1) -- (0, 1);
\node[rotate=0,font=\footnotesize] at (0.8, 0.5) {Y};
\node[rotate=0,font=\footnotesize] at (0.5, 0.8) {B};
\node[rotate=0,font=\footnotesize] at (0.2, 0.5) {X};
\node[rotate=0,font=\footnotesize] at (0.5, 0.2) {A};
% tile at position (x,y)=(0, 1)
\draw (0, 1) -- (1, 1);
\draw (0, 1) -- (0, 2);
\draw (1, 2) -- (1, 1);
\draw (1, 2) -- (0, 2);
\node[rotate=0,font=\footnotesize] at (0.8, 1.5) {Z};
\node[rotate=0,font=\footnotesize] at (0.5, 1.8) {D};
\node[rotate=0,font=\footnotesize] at (0.2, 1.5) {W};
\node[rotate=0,font=\footnotesize] at (0.5, 1.2) {B};
\end{tikzpicture}};
\end{tikzpicture}
\end{center}
whether $i=1$ or $i=2$.

\begin{remark}
In Equation~\eqref{eq:omegaStoT}, the set $P_r$ 
can be computed with the help of transducers since
the fusion operation $\boxminus^i$ on Wang tiles can be seen as the product of
transducers (see Equation~\eqref{eq:boxminus_with_transducer}).
The algorithmic aspects of Theorem~\ref{thm:exist-homeo} 
will be considered in a forthcoming work \cite{labbe_structure_2018}.
\end{remark}

\begin{remark}
    The existence of a $2$-dimensional morphism $\omega:\Omega_\S\to\Omega_\T$
    which is recognizable in the Wang shift $\Omega_\S$ and
    surjective up to a shift implies the existence of a
    \emph{homeomorphism} between $\Omega_\S$ and $\Omega_\T$ where each tile
    have specific real width and height and tilings are defined in $\R^2$
    rather than in $\Z^2$.
    We have chosen in this contribution the point of view of tilings of $\Z^2$
    which makes some aspects look easier and other more intricate.
\end{remark}

% in the next paper
% \subsection{An algorithm}
% 
% ... TODO

\section{Desubstitution of $\Omega_\U$}\label{sec:desubstitionofU}

In this section, we prove that tilings in $\Omega_\U$ can be desubstituted
uniquely into tilings of $\Z^2$ over another set $\V$ of Wang tiles. The proof
follows from the computation of the set of vertical dominoes appearing in
$\Omega_\U$ (Lemma~\ref{lem:dominoes-direction2-U}) whose proof is done with
Sage.

The tile set $\U$ seen in Figure~\ref{fig:the-tile-set-U} is defined over an
alphabet 
$\{A, B, C, D, E, F, G, H, I, J\}$
of 10 vertical colors
and an alphabet
$\{K, L, M, N, O, P\}$
of 6 horizontal colors.
We identify each tile $u_i$ in $\U$ with an index $i\in\{0,\dots,18\}$ in the
following way and we draw this index in the center of each tile to ease their
identification.
\begin{equation*}
\begin{array}{l}
    \U= \left\{ 
u_0=\hspace{-2mm}\raisebox{-3mm}{\UTileO},
u_1=\hspace{-2mm}\raisebox{-3mm}{\UTileI},
u_2=\hspace{-2mm}\raisebox{-3mm}{\UTileII},
u_3=\hspace{-2mm}\raisebox{-3mm}{\UTileIII},
u_4=\hspace{-2mm}\raisebox{-3mm}{\UTileIV},
u_5=\hspace{-2mm}\raisebox{-3mm}{\UTileV},
\right.\\
\hspace{0mm}\left.
u_6=\hspace{-2mm}\raisebox{-3mm}{\UTileVI},
u_7=\hspace{-2mm}\raisebox{-3mm}{\UTileVII},
u_8=\hspace{-2mm}\raisebox{-3mm}{\UTileVIII},
u_9=\hspace{-2mm}\raisebox{-3mm}{\UTileIX},
u_{10}=\hspace{-2mm}\raisebox{-3mm}{\UTileX},
u_{11}=\hspace{-2mm}\raisebox{-3mm}{\UTileXI},
u_{12}=\hspace{-2mm}\raisebox{-3mm}{\UTileXII},
\right.\\
\hspace{5mm}\left.
u_{13}=\hspace{-2mm}\raisebox{-3mm}{\UTileXIII},
u_{14}=\hspace{-2mm}\raisebox{-3mm}{\UTileXIV},
u_{15}=\hspace{-2mm}\raisebox{-3mm}{\UTileXV},
u_{16}=\hspace{-2mm}\raisebox{-3mm}{\UTileXVI},
u_{17}=\hspace{-2mm}\raisebox{-3mm}{\UTileXVII},
u_{18}=\hspace{-2mm}\raisebox{-3mm}{\UTileXVIII}
\right\}.
\end{array}
\end{equation*}
The transducer representation of the tile set $\U$ 
was shown in Figure~\ref{fig:T7transducer}.

\begin{lemma}\label{lem:dominoes-direction2-U}
    The set of dominoes in direction $\be_2$ allowing a surrounding of radius 2
    in $\Omega_\U$ is
    \begin{align*}
        D_{\U,\be_2,2} &= \{
\left(u_{0}, u_{8}\right)
,
\left(u_{1}, u_{8}\right)
,
\left(u_{1}, u_{9}\right)
,
\left(u_{1}, u_{11}\right)
,
\left(u_{2}, u_{16}\right)
,
\left(u_{3}, u_{16}\right)
,
\left(u_{4}, u_{13}\right)
,
\left(u_{5}, u_{13}\right)
,\\&\hspace{10mm}
\left(u_{6}, u_{14}\right)
,
\left(u_{6}, u_{17}\right)
,
\left(u_{7}, u_{15}\right)
,
\left(u_{8}, u_{0}\right)
,
\left(u_{8}, u_{9}\right)
,
\left(u_{8}, u_{11}\right)
,
\left(u_{9}, u_{1}\right)
,
\left(u_{9}, u_{10}\right)
,\\&\hspace{10mm}
\left(u_{10}, u_{1}\right)
,
\left(u_{11}, u_{1}\right)
,
\left(u_{11}, u_{10}\right)
,
\left(u_{12}, u_{6}\right)
,
\left(u_{13}, u_{4}\right)
,
\left(u_{13}, u_{7}\right)
,
\left(u_{13}, u_{18}\right)
,
\left(u_{14}, u_{2}\right)
,\\&\hspace{10mm}
\left(u_{14}, u_{6}\right)
,
\left(u_{14}, u_{12}\right)
,
\left(u_{15}, u_{7}\right)
,
\left(u_{15}, u_{13}\right)
,
\left(u_{15}, u_{18}\right)
,
\left(u_{16}, u_{3}\right)
,
\left(u_{16}, u_{14}\right)
,
\left(u_{16}, u_{17}\right)
,\\&\hspace{10mm}
\left(u_{17}, u_{3}\right)
,
\left(u_{17}, u_{14}\right)
,
\left(u_{18}, u_{5}\right)
\}.
    \end{align*}
\end{lemma}
% sage: len(U.dominoes_with_surrounding(radius=0))
% 79
% sage: len(U.dominoes_with_surrounding(radius=1))
% 37
% sage: len(U.dominoes_with_surrounding(radius=2))
% 35
% sage: len(U.dominoes_with_surrounding(radius=3))
% 35

The proof of Lemma~\ref{lem:dominoes-direction2-U} is done with Sage
in Section~\ref{sec:appendix}.

\begin{lemma}\label{lem:markersforU}
    The set
    $R = \{u_0, u_1, u_2, u_3, u_4, u_5, u_6, u_7\}$
    is a set of markers in the direction $\be_2$ for~$\U$.
\end{lemma}

\begin{proof}
    First we show that $R\odot^2 R$ is forbidden in $\Omega_\U$.
    The language of dominoes in direction $\be_2$ in $\Omega_\U$ is a subset
    of the set of dominoes in direction $\be_2$ allowing a surrounding of
    radius~2 in $\Omega_\U$:
    \[
        \left\{(u,v)\in\U^2\mid u\odot^2v\in\L(\Omega_\U)\right\}
         \subseteq D_{\U,\be_2,2}
    \]
    computed in Lemma~\ref{lem:dominoes-direction2-U}.
    Since $D_{\U,\be_2,2}\cap R^2=\varnothing$, we deduce that
    $(R\odot^2 R)\cap\L(\Omega_\U)=\varnothing$, that is
    $R\odot^2 R$ is forbidden in $\Omega_\U$.

    Finally, we remark that
    \begin{align*}
        \scleft(\U\setminus R)&=\scright(\U\setminus R) = \{A,B,C,E,G,I\},\\
        \scleft(R)&=\scright(R) = \{D,F,H,J\},
    \end{align*}
    which are disjoint.
    Thus
    $R\odot^1 (\U\setminus R)$ and
    $(\U\setminus R)\odot^1 R$ 
    are forbidden in $\Omega_\U$.
    We conclude that $R$
    is a set of markers in the direction $\be_2$ for $\U$.
\end{proof}

\begin{figure}[h]
\begin{center}
    \input{article1_alpha.tex}
\end{center}
    \caption{The morphism $\alpha:\Omega_\V\to\Omega_\U$.}
    \label{fig:alpha}
\end{figure}

\begin{proposition}\label{prop:wecandesubstituteU}
    There exists a tile set $\V$ of cardinality 21 and
    a $d$-dimensional morphism $\alpha:\Omega_\V\to\Omega_\U$
    that is recognizable in $\Omega_\V$
    and $\alpha(\Omega_\V)\cup\sigma^{\be_2}\alpha(\Omega_\V)=\Omega_\U$.
\end{proposition}

\newcommand\VTileO{
\begin{tikzpicture}
[scale=0.900000000000000]
\tikzstyle{every node}=[font=\tiny]
% tile at position (x,y)=(0, 0)
\node at (0.5, 0.5) {0};
\draw (0, 0) -- ++ (0,1);
\draw (0, 0) -- ++ (1,0);
\draw (1, 0) -- ++ (0,1);
\draw (0, 1) -- ++ (1,0);
\node[rotate=0] at (0.800000000000000, 0.5) {A};
\node[rotate=0] at (0.5, 0.800000000000000) {L};
\node[rotate=0] at (0.200000000000000, 0.5) {I};
\node[rotate=0] at (0.5, 0.200000000000000) {O};
\end{tikzpicture}
} % end of newcommand
\newcommand\VTileI{
\begin{tikzpicture}
[scale=0.900000000000000]
\tikzstyle{every node}=[font=\tiny]
% tile at position (x,y)=(0, 0)
\node at (0.5, 0.5) {1};
\draw (0, 0) -- ++ (0,1);
\draw (0, 0) -- ++ (1,0);
\draw (1, 0) -- ++ (0,1);
\draw (0, 1) -- ++ (1,0);
\node[rotate=0] at (0.800000000000000, 0.5) {B};
\node[rotate=0] at (0.5, 0.800000000000000) {O};
\node[rotate=0] at (0.200000000000000, 0.5) {I};
\node[rotate=0] at (0.5, 0.200000000000000) {O};
\end{tikzpicture}
} % end of newcommand
\newcommand\VTileII{
\begin{tikzpicture}
[scale=0.900000000000000]
\tikzstyle{every node}=[font=\tiny]
% tile at position (x,y)=(0, 0)
\node at (0.5, 0.5) {2};
\draw (0, 0) -- ++ (0,1);
\draw (0, 0) -- ++ (1,0);
\draw (1, 0) -- ++ (0,1);
\draw (0, 1) -- ++ (1,0);
\node[rotate=0] at (0.800000000000000, 0.5) {E};
\node[rotate=0] at (0.5, 0.800000000000000) {P};
\node[rotate=0] at (0.200000000000000, 0.5) {I};
\node[rotate=0] at (0.5, 0.200000000000000) {P};
\end{tikzpicture}
} % end of newcommand
\newcommand\VTileIII{
\begin{tikzpicture}
[scale=0.900000000000000]
\tikzstyle{every node}=[font=\tiny]
% tile at position (x,y)=(0, 0)
\node at (0.5, 0.5) {3};
\draw (0, 0) -- ++ (0,1);
\draw (0, 0) -- ++ (1,0);
\draw (1, 0) -- ++ (0,1);
\draw (0, 1) -- ++ (1,0);
\node[rotate=0] at (0.800000000000000, 0.5) {G};
\node[rotate=0] at (0.5, 0.800000000000000) {L};
\node[rotate=0] at (0.200000000000000, 0.5) {E};
\node[rotate=0] at (0.5, 0.200000000000000) {O};
\end{tikzpicture}
} % end of newcommand
\newcommand\VTileIV{
\begin{tikzpicture}
[scale=0.900000000000000]
\tikzstyle{every node}=[font=\tiny]
% tile at position (x,y)=(0, 0)
\node at (0.5, 0.5) {4};
\draw (0, 0) -- ++ (0,1);
\draw (0, 0) -- ++ (1,0);
\draw (1, 0) -- ++ (0,1);
\draw (0, 1) -- ++ (1,0);
\node[rotate=0] at (0.800000000000000, 0.5) {I};
\node[rotate=0] at (0.5, 0.800000000000000) {K};
\node[rotate=0] at (0.200000000000000, 0.5) {A};
\node[rotate=0] at (0.5, 0.200000000000000) {K};
\end{tikzpicture}
} % end of newcommand
\newcommand\VTileV{
\begin{tikzpicture}
[scale=0.900000000000000]
\tikzstyle{every node}=[font=\tiny]
% tile at position (x,y)=(0, 0)
\node at (0.5, 0.5) {5};
\draw (0, 0) -- ++ (0,1);
\draw (0, 0) -- ++ (1,0);
\draw (1, 0) -- ++ (0,1);
\draw (0, 1) -- ++ (1,0);
\node[rotate=0] at (0.800000000000000, 0.5) {I};
\node[rotate=0] at (0.5, 0.800000000000000) {K};
\node[rotate=0] at (0.200000000000000, 0.5) {B};
\node[rotate=0] at (0.5, 0.200000000000000) {M};
\end{tikzpicture}
} % end of newcommand
\newcommand\VTileVI{
\begin{tikzpicture}
[scale=0.900000000000000]
\tikzstyle{every node}=[font=\tiny]
% tile at position (x,y)=(0, 0)
\node at (0.5, 0.5) {6};
\draw (0, 0) -- ++ (0,1);
\draw (0, 0) -- ++ (1,0);
\draw (1, 0) -- ++ (0,1);
\draw (0, 1) -- ++ (1,0);
\node[rotate=0] at (0.800000000000000, 0.5) {I};
\node[rotate=0] at (0.5, 0.800000000000000) {P};
\node[rotate=0] at (0.200000000000000, 0.5) {G};
\node[rotate=0] at (0.5, 0.200000000000000) {K};
\end{tikzpicture}
} % end of newcommand
\newcommand\VTileVII{
\begin{tikzpicture}
[scale=0.900000000000000]
\tikzstyle{every node}=[font=\tiny]
% tile at position (x,y)=(0, 0)
\node at (0.5, 0.5) {7};
\draw (0, 0) -- ++ (0,1);
\draw (0, 0) -- ++ (1,0);
\draw (1, 0) -- ++ (0,1);
\draw (0, 1) -- ++ (1,0);
\node[rotate=0] at (0.800000000000000, 0.5) {I};
\node[rotate=0] at (0.5, 0.800000000000000) {P};
\node[rotate=0] at (0.200000000000000, 0.5) {I};
\node[rotate=0] at (0.5, 0.200000000000000) {K};
\end{tikzpicture}
} % end of newcommand
\newcommand\VTileVIII{
\begin{tikzpicture}
[scale=0.900000000000000]
\tikzstyle{every node}=[font=\tiny]
% tile at position (x,y)=(0, 0)
\node at (0.5, 0.5) {8};
\draw (0, 0) -- ++ (0,1);
\draw (0, 0) -- ++ (1,0);
\draw (1, 0) -- ++ (0,1);
\draw (0, 1) -- ++ (1,0);
\node[rotate=90] at (0.800000000000000, 0.5) {AF};
\node[rotate=0] at (0.5, 0.800000000000000) {O};
\node[rotate=90] at (0.200000000000000, 0.5) {IH};
\node[rotate=0] at (0.5, 0.200000000000000) {O};
\end{tikzpicture}
} % end of newcommand
\newcommand\VTileIX{
\begin{tikzpicture}
[scale=0.900000000000000]
\tikzstyle{every node}=[font=\tiny]
% tile at position (x,y)=(0, 0)
\node at (0.5, 0.5) {9};
\draw (0, 0) -- ++ (0,1);
\draw (0, 0) -- ++ (1,0);
\draw (1, 0) -- ++ (0,1);
\draw (0, 1) -- ++ (1,0);
\node[rotate=90] at (0.800000000000000, 0.5) {BF};
\node[rotate=0] at (0.5, 0.800000000000000) {O};
\node[rotate=90] at (0.200000000000000, 0.5) {IJ};
\node[rotate=0] at (0.5, 0.200000000000000) {O};
\end{tikzpicture}
} % end of newcommand
\newcommand\VTileX{
\begin{tikzpicture}
[scale=0.900000000000000]
\tikzstyle{every node}=[font=\tiny]
% tile at position (x,y)=(0, 0)
\node at (0.5, 0.5) {10};
\draw (0, 0) -- ++ (0,1);
\draw (0, 0) -- ++ (1,0);
\draw (1, 0) -- ++ (0,1);
\draw (0, 1) -- ++ (1,0);
\node[rotate=90] at (0.800000000000000, 0.5) {CH};
\node[rotate=0] at (0.5, 0.800000000000000) {P};
\node[rotate=90] at (0.200000000000000, 0.5) {IH};
\node[rotate=0] at (0.5, 0.200000000000000) {P};
\end{tikzpicture}
} % end of newcommand
\newcommand\VTileXI{
\begin{tikzpicture}
[scale=0.900000000000000]
\tikzstyle{every node}=[font=\tiny]
% tile at position (x,y)=(0, 0)
\node at (0.5, 0.5) {11};
\draw (0, 0) -- ++ (0,1);
\draw (0, 0) -- ++ (1,0);
\draw (1, 0) -- ++ (0,1);
\draw (0, 1) -- ++ (1,0);
\node[rotate=90] at (0.800000000000000, 0.5) {EH};
\node[rotate=0] at (0.5, 0.800000000000000) {K};
\node[rotate=90] at (0.200000000000000, 0.5) {GF};
\node[rotate=0] at (0.5, 0.200000000000000) {P};
\end{tikzpicture}
} % end of newcommand
\newcommand\VTileXII{
\begin{tikzpicture}
[scale=0.900000000000000]
\tikzstyle{every node}=[font=\tiny]
% tile at position (x,y)=(0, 0)
\node at (0.5, 0.5) {12};
\draw (0, 0) -- ++ (0,1);
\draw (0, 0) -- ++ (1,0);
\draw (1, 0) -- ++ (0,1);
\draw (0, 1) -- ++ (1,0);
\node[rotate=90] at (0.800000000000000, 0.5) {EH};
\node[rotate=0] at (0.5, 0.800000000000000) {K};
\node[rotate=90] at (0.200000000000000, 0.5) {ID};
\node[rotate=0] at (0.5, 0.200000000000000) {P};
\end{tikzpicture}
} % end of newcommand
\newcommand\VTileXIII{
\begin{tikzpicture}
[scale=0.900000000000000]
\tikzstyle{every node}=[font=\tiny]
% tile at position (x,y)=(0, 0)
\node at (0.5, 0.5) {13};
\draw (0, 0) -- ++ (0,1);
\draw (0, 0) -- ++ (1,0);
\draw (1, 0) -- ++ (0,1);
\draw (0, 1) -- ++ (1,0);
\node[rotate=90] at (0.800000000000000, 0.5) {EH};
\node[rotate=0] at (0.5, 0.800000000000000) {P};
\node[rotate=90] at (0.200000000000000, 0.5) {IJ};
\node[rotate=0] at (0.5, 0.200000000000000) {P};
\end{tikzpicture}
} % end of newcommand
\newcommand\VTileXIV{
\begin{tikzpicture}
[scale=0.900000000000000]
\tikzstyle{every node}=[font=\tiny]
% tile at position (x,y)=(0, 0)
\node at (0.5, 0.5) {14};
\draw (0, 0) -- ++ (0,1);
\draw (0, 0) -- ++ (1,0);
\draw (1, 0) -- ++ (0,1);
\draw (0, 1) -- ++ (1,0);
\node[rotate=90] at (0.800000000000000, 0.5) {GF};
\node[rotate=0] at (0.5, 0.800000000000000) {O};
\node[rotate=90] at (0.200000000000000, 0.5) {CH};
\node[rotate=0] at (0.5, 0.200000000000000) {L};
\end{tikzpicture}
} % end of newcommand
\newcommand\VTileXV{
\begin{tikzpicture}
[scale=0.900000000000000]
\tikzstyle{every node}=[font=\tiny]
% tile at position (x,y)=(0, 0)
\node at (0.5, 0.5) {15};
\draw (0, 0) -- ++ (0,1);
\draw (0, 0) -- ++ (1,0);
\draw (1, 0) -- ++ (0,1);
\draw (0, 1) -- ++ (1,0);
\node[rotate=90] at (0.800000000000000, 0.5) {GF};
\node[rotate=0] at (0.5, 0.800000000000000) {O};
\node[rotate=90] at (0.200000000000000, 0.5) {EH};
\node[rotate=0] at (0.5, 0.200000000000000) {O};
\end{tikzpicture}
} % end of newcommand
\newcommand\VTileXVI{
\begin{tikzpicture}
[scale=0.900000000000000]
\tikzstyle{every node}=[font=\tiny]
% tile at position (x,y)=(0, 0)
\node at (0.5, 0.5) {16};
\draw (0, 0) -- ++ (0,1);
\draw (0, 0) -- ++ (1,0);
\draw (1, 0) -- ++ (0,1);
\draw (0, 1) -- ++ (1,0);
\node[rotate=90] at (0.800000000000000, 0.5) {ID};
\node[rotate=0] at (0.5, 0.800000000000000) {M};
\node[rotate=90] at (0.200000000000000, 0.5) {AF};
\node[rotate=0] at (0.5, 0.200000000000000) {K};
\end{tikzpicture}
} % end of newcommand
\newcommand\VTileXVII{
\begin{tikzpicture}
[scale=0.900000000000000]
\tikzstyle{every node}=[font=\tiny]
% tile at position (x,y)=(0, 0)
\node at (0.5, 0.5) {17};
\draw (0, 0) -- ++ (0,1);
\draw (0, 0) -- ++ (1,0);
\draw (1, 0) -- ++ (0,1);
\draw (0, 1) -- ++ (1,0);
\node[rotate=90] at (0.800000000000000, 0.5) {ID};
\node[rotate=0] at (0.5, 0.800000000000000) {M};
\node[rotate=90] at (0.200000000000000, 0.5) {BF};
\node[rotate=0] at (0.5, 0.200000000000000) {M};
\end{tikzpicture}
} % end of newcommand
\newcommand\VTileXVIII{
\begin{tikzpicture}
[scale=0.900000000000000]
\tikzstyle{every node}=[font=\tiny]
% tile at position (x,y)=(0, 0)
\node at (0.5, 0.5) {18};
\draw (0, 0) -- ++ (0,1);
\draw (0, 0) -- ++ (1,0);
\draw (1, 0) -- ++ (0,1);
\draw (0, 1) -- ++ (1,0);
\node[rotate=90] at (0.800000000000000, 0.5) {IH};
\node[rotate=0] at (0.5, 0.800000000000000) {K};
\node[rotate=90] at (0.200000000000000, 0.5) {GF};
\node[rotate=0] at (0.5, 0.200000000000000) {K};
\end{tikzpicture}
} % end of newcommand
\newcommand\VTileXIX{
\begin{tikzpicture}
[scale=0.900000000000000]
\tikzstyle{every node}=[font=\tiny]
% tile at position (x,y)=(0, 0)
\node at (0.5, 0.5) {19};
\draw (0, 0) -- ++ (0,1);
\draw (0, 0) -- ++ (1,0);
\draw (1, 0) -- ++ (0,1);
\draw (0, 1) -- ++ (1,0);
\node[rotate=90] at (0.800000000000000, 0.5) {IH};
\node[rotate=0] at (0.5, 0.800000000000000) {K};
\node[rotate=90] at (0.200000000000000, 0.5) {ID};
\node[rotate=0] at (0.5, 0.200000000000000) {K};
\end{tikzpicture}
} % end of newcommand
\newcommand\VTileXX{
\begin{tikzpicture}
[scale=0.900000000000000]
\tikzstyle{every node}=[font=\tiny]
% tile at position (x,y)=(0, 0)
\node at (0.5, 0.5) {20};
\draw (0, 0) -- ++ (0,1);
\draw (0, 0) -- ++ (1,0);
\draw (1, 0) -- ++ (0,1);
\draw (0, 1) -- ++ (1,0);
\node[rotate=90] at (0.800000000000000, 0.5) {IJ};
\node[rotate=0] at (0.5, 0.800000000000000) {M};
\node[rotate=90] at (0.200000000000000, 0.5) {GF};
\node[rotate=0] at (0.5, 0.200000000000000) {K};
\end{tikzpicture}
} % end of newcommand

\begin{proof}
    From Lemma~\ref{lem:markersforU},
    $R = \{u_0, u_1, u_2, u_3, u_4, u_5, u_6, u_7\}$
    is a set of markers in the direction $\be_2$ for $\U$.
    Therefore, Theorem~\ref{thm:exist-homeo} applies.
    Let $P_2$ be the set
    \begin{align*}
        P_2 &= \left\{u\boxminus^2 v \,\middle|\, 
            u\in\U\setminus R, v\in R 
            \text{ and } 
            (u,v) \in D_{\U,\be_2,2}\right\}\\
           &= \left\{
\begin{array}{|c|}\hline u_{0} \\ u_{8} \\ \hline \end{array}
\,,\,
\begin{array}{|c|}\hline u_{1} \\ u_{9} \\ \hline \end{array}
\,,\,
\begin{array}{|c|}\hline u_{1} \\ u_{10} \\ \hline \end{array}
\,,\,
\begin{array}{|c|}\hline u_{1} \\ u_{11} \\ \hline \end{array}
\,,\,
\begin{array}{|c|}\hline u_{6} \\ u_{12} \\ \hline \end{array}
\,,\,
\begin{array}{|c|}\hline u_{4} \\ u_{13} \\ \hline \end{array}
\,,\,
\begin{array}{|c|}\hline u_{7} \\ u_{13} \\ \hline \end{array}
\,,\,
\begin{array}{|c|}\hline u_{2} \\ u_{14} \\ \hline \end{array}
\,,\,
\begin{array}{|c|}\hline u_{6} \\ u_{14} \\ \hline \end{array}
\,,\,
\begin{array}{|c|}\hline u_{7} \\ u_{15} \\ \hline \end{array}
\,,\,
\begin{array}{|c|}\hline u_{3} \\ u_{16} \\ \hline \end{array}
\,,\,
\begin{array}{|c|}\hline u_{3} \\ u_{17} \\ \hline \end{array}
\,,\,
\begin{array}{|c|}\hline u_{5} \\ u_{18} \\ \hline \end{array}
\right\}
\end{align*}
    according to Lemma~\ref{lem:dominoes-direction2-U}.
    Let $K_2$ be the set
    \begin{align*}
        K_2 &= \left\{u\in\U\setminus R\,\middle|\, \text{ there exists }
        v\in\U\setminus
            R \text{ such that } (u,v) \in D_{\U,\be_2,2}\right\}\\
          &= \{u_8, u_9, u_{11}, u_{13}, u_{14}, u_{15}, u_{16}, u_{17}\}.
    \end{align*}
    Let $\Omega_\V$ be the Wang shift of the Wang tile set $\V=K_2\cup
    P_2$ of cardinality~21:
\begin{equation*}\label{eq:tilesV}
\begin{array}{c}
    \V= \left\{ 
v_0=\hspace{-2mm}\raisebox{-3mm}{\VTileO},
v_1=\hspace{-2mm}\raisebox{-3mm}{\VTileI},
v_2=\hspace{-2mm}\raisebox{-3mm}{\VTileII},
v_3=\hspace{-2mm}\raisebox{-3mm}{\VTileIII},
v_4=\hspace{-2mm}\raisebox{-3mm}{\VTileIV},
v_5=\hspace{-2mm}\raisebox{-3mm}{\VTileV},
v_6=\hspace{-2mm}\raisebox{-3mm}{\VTileVI},
\right.\\
\hspace{1cm}\left.
v_7=\hspace{-2mm}\raisebox{-3mm}{\VTileVII},
v_8=\hspace{-2mm}\raisebox{-3mm}{\VTileVIII},
v_9=\hspace{-2mm}\raisebox{-3mm}{\VTileIX},
v_{10}=\hspace{-2mm}\raisebox{-3mm}{\VTileX},
v_{11}=\hspace{-2mm}\raisebox{-3mm}{\VTileXI},
v_{12}=\hspace{-2mm}\raisebox{-3mm}{\VTileXII},
v_{13}=\hspace{-2mm}\raisebox{-3mm}{\VTileXIII},
\right.\\
\hspace{1cm}\left.
v_{14}=\hspace{-2mm}\raisebox{-3mm}{\VTileXIV},
v_{15}=\hspace{-2mm}\raisebox{-3mm}{\VTileXV},
v_{16}=\hspace{-2mm}\raisebox{-3mm}{\VTileXVI},
v_{17}=\hspace{-2mm}\raisebox{-3mm}{\VTileXVII},
v_{18}=\hspace{-2mm}\raisebox{-3mm}{\VTileXVIII},
v_{19}=\hspace{-2mm}\raisebox{-3mm}{\VTileXIX},
v_{20}=\hspace{-2mm}\raisebox{-3mm}{\VTileXX}
\right\}.
\end{array}
\end{equation*}
    Let $\alpha:\Omega_\V\to\Omega_\U$ be the $2$-dimensional morphism
    defined by
    \begin{equation*}
        \alpha:
    \begin{cases}
        u\boxminus^2 v
        \quad \mapsto \quad
        u\odot^2 v
        & \text{ if }
        u\boxminus^2 v \in P_2\\
        k \quad \mapsto \quad
        k & \text{ if } k \in K_2.
    \end{cases}
    \end{equation*}
    Then $\alpha:\Omega_\V\to\Omega_\U$ is
    recognizable in $\Omega_\V$
    and $\alpha(\Omega_\V)\cup\sigma^{\be_2}\alpha(\Omega_\V)=\Omega_\U$.
\end{proof}

The morphism $\alpha$ in terms of the
$\U=\{u_i\}_{0\leq i\leq 18}$ and $\V=\{v_i\}_{0\leq i\leq 20}$ 
is in Figure~\ref{fig:alpha}.
The tile set $\V$ uses the
following alphabet of 13 vertical colors:
\begin{equation*}
\{A, B, E, G, I, AF, BF, CH, EH, GF, ID, IH, IJ\}
\end{equation*}
and the following alphabet of 5 horizontal colors:
\begin{equation*}
\{K, L, M, O, P\}.
\end{equation*}
Notice that the set of vertical left and right colors has increased in size
compared to $\U$ and that the set of horizontal bottom and top colors
has \emph{decreased} in size. Indeed, color N has disappeared.

\section{Desubstitution of $\Omega_\V$}\label{sec:desubstitionofV}

In this section, we prove that tilings in $\Omega_\V$ can be desubstituted
uniquely into tilings of $\Z^2$ over another set $\W$ of Wang tiles. The proof
follows from the computation of the set of horizontal dominoes appearing in
$\Omega_\V$ (Lemma~\ref{lem:dominoes-direction1-V}) whose proof is done with
Sage.

\begin{lemma}\label{lem:dominoes-direction1-V}
    The set of dominoes in direction $\be_1$ allowing a surrounding of radius 1
    in $\Omega_\V$ is
    \begin{align*}
        D_{\V,\be_1,1} &= \{
\left(v_{0}, v_{4}\right)
,
\left(v_{1}, v_{5}\right)
,
\left(v_{2}, v_{3}\right)
,
\left(v_{3}, v_{6}\right)
,
\left(v_{4}, v_{1}\right)
,
\left(v_{4}, v_{2}\right)
,
\left(v_{5}, v_{1}\right)
,
\left(v_{5}, v_{2}\right)
,\\&\hspace{10mm}
\left(v_{5}, v_{7}\right)
,
\left(v_{6}, v_{1}\right)
,
\left(v_{7}, v_{0}\right)
,
\left(v_{7}, v_{1}\right)
,
\left(v_{8}, v_{16}\right)
,
\left(v_{9}, v_{17}\right)
,
\left(v_{10}, v_{14}\right)
,
\left(v_{11}, v_{15}\right)
,\\&\hspace{10mm}
\left(v_{12}, v_{15}\right)
,
\left(v_{13}, v_{15}\right)
,
\left(v_{14}, v_{11}\right)
,
\left(v_{14}, v_{18}\right)
,
\left(v_{15}, v_{18}\right)
,
\left(v_{15}, v_{20}\right)
,
\left(v_{16}, v_{12}\right)
,
\left(v_{17}, v_{12}\right)
,\\&\hspace{10mm}
\left(v_{17}, v_{19}\right)
,
\left(v_{18}, v_{8}\right)
,
\left(v_{18}, v_{10}\right)
,
\left(v_{19}, v_{8}\right)
,
\left(v_{20}, v_{9}\right)
,
\left(v_{20}, v_{13}\right)
\}.
    \end{align*}
\end{lemma}

The proof of Lemma~\ref{lem:dominoes-direction1-V} is done with Sage
in Section~\ref{sec:appendix}.

\begin{lemma}\label{lem:markersforV}
    The set
    $R = \{v_0, v_1, v_3, v_8, v_9, v_{14}, v_{15}\}$
    is a set of markers in the direction $\be_1$ for~$\V$.
\end{lemma}

\begin{proof}
    First we show that $R\odot^1 R$ is forbidden in $\Omega_\V$.
    The language of dominoes in direction $\be_1$ in $\Omega_\V$ is a subset
    of the set of dominoes in direction $\be_1$ allowing a surrounding of
    radius~1 in $\Omega_\V$:
    \[
        \left\{(u,v)\in\V^2\mid u\odot^1v\in\L(\Omega_\V)\right\}
         \subseteq D_{\V,\be_1,1}
    \]
    computed in Lemma~\ref{lem:dominoes-direction1-V}.
    Since $D_{\V,\be_1,1}\cap R^2=\varnothing$, we deduce that
    $(R\odot^1 R)\cap\L(\Omega_\V)=\varnothing$, that is
    $R\odot^1 R$ is forbidden in $\Omega_\V$.

    Finally, we remark that
    \begin{align*}
    \scbottom(\V\setminus R) &= \sctop(\V\setminus R) = \{K,M,P\},\\
    \scbottom(R) &= \sctop(R) = \{L,O\},
    \end{align*}
    which are disjoint.
    Thus
    $R\odot^2 (\V\setminus R)$ and
    $(\V\setminus R)\odot^2 R$ 
    are forbidden in $\Omega_\V$.
    We conclude that $R$
    is a set of markers in the direction $\be_1$ for $\V$.
\end{proof}

\begin{figure}[h]
\begin{center}
    \input{article1_beta.tex}
\end{center}
    \caption{The morphism $\beta:\Omega_\W\to\Omega_\V$.}
    \label{fig:beta}
\end{figure}

\begin{proposition}\label{prop:wecandesubstituteV}
    There exists a tile set $\W$ of cardinality 19 and
    a $d$-dimensional morphism $\beta:\Omega_\W\to\Omega_\V$
    that is recognizable in $\Omega_\W$
    and $\beta(\Omega_\W)\cup\sigma^{\be_1}\beta(\Omega_\W)=\Omega_\V$.
\end{proposition}

\newcommand\WTileO{
\begin{tikzpicture}
[scale=0.900000000000000]
\tikzstyle{every node}=[font=\tiny]
% tile at position (x,y)=(0, 0)
\node at (0.5, 0.5) {0};
\draw (0, 0) -- ++ (0,1);
\draw (0, 0) -- ++ (1,0);
\draw (1, 0) -- ++ (0,1);
\draw (0, 1) -- ++ (1,0);
\node[rotate=0] at (0.800000000000000, 0.5) {I};
\node[rotate=0] at (0.5, 0.800000000000000) {K};
\node[rotate=0] at (0.200000000000000, 0.5) {A};
\node[rotate=0] at (0.5, 0.200000000000000) {K};
\end{tikzpicture}
} % end of newcommand
\newcommand\WTileI{
\begin{tikzpicture}
[scale=0.900000000000000]
\tikzstyle{every node}=[font=\tiny]
% tile at position (x,y)=(0, 0)
\node at (0.5, 0.5) {1};
\draw (0, 0) -- ++ (0,1);
\draw (0, 0) -- ++ (1,0);
\draw (1, 0) -- ++ (0,1);
\draw (0, 1) -- ++ (1,0);
\node[rotate=0] at (0.800000000000000, 0.5) {I};
\node[rotate=0] at (0.5, 0.800000000000000) {K};
\node[rotate=0] at (0.200000000000000, 0.5) {B};
\node[rotate=0] at (0.5, 0.200000000000000) {M};
\end{tikzpicture}
} % end of newcommand
\newcommand\WTileII{
\begin{tikzpicture}
[scale=0.900000000000000]
\tikzstyle{every node}=[font=\tiny]
% tile at position (x,y)=(0, 0)
\node at (0.5, 0.5) {2};
\draw (0, 0) -- ++ (0,1);
\draw (0, 0) -- ++ (1,0);
\draw (1, 0) -- ++ (0,1);
\draw (0, 1) -- ++ (1,0);
\node[rotate=0] at (0.800000000000000, 0.5) {A};
\node[rotate=0] at (0.5, 0.800000000000000) {PL};
\node[rotate=0] at (0.200000000000000, 0.5) {I};
\node[rotate=0] at (0.5, 0.200000000000000) {KO};
\end{tikzpicture}
} % end of newcommand
\newcommand\WTileIII{
\begin{tikzpicture}
[scale=0.900000000000000]
\tikzstyle{every node}=[font=\tiny]
% tile at position (x,y)=(0, 0)
\node at (0.5, 0.5) {3};
\draw (0, 0) -- ++ (0,1);
\draw (0, 0) -- ++ (1,0);
\draw (1, 0) -- ++ (0,1);
\draw (0, 1) -- ++ (1,0);
\node[rotate=0] at (0.800000000000000, 0.5) {G};
\node[rotate=0] at (0.5, 0.800000000000000) {PL};
\node[rotate=0] at (0.200000000000000, 0.5) {I};
\node[rotate=0] at (0.5, 0.200000000000000) {PO};
\end{tikzpicture}
} % end of newcommand
\newcommand\WTileIV{
\begin{tikzpicture}
[scale=0.900000000000000]
\tikzstyle{every node}=[font=\tiny]
% tile at position (x,y)=(0, 0)
\node at (0.5, 0.5) {4};
\draw (0, 0) -- ++ (0,1);
\draw (0, 0) -- ++ (1,0);
\draw (1, 0) -- ++ (0,1);
\draw (0, 1) -- ++ (1,0);
\node[rotate=0] at (0.800000000000000, 0.5) {B};
\node[rotate=0] at (0.5, 0.800000000000000) {KO};
\node[rotate=0] at (0.200000000000000, 0.5) {A};
\node[rotate=0] at (0.5, 0.200000000000000) {KO};
\end{tikzpicture}
} % end of newcommand
\newcommand\WTileV{
\begin{tikzpicture}
[scale=0.900000000000000]
\tikzstyle{every node}=[font=\tiny]
% tile at position (x,y)=(0, 0)
\node at (0.5, 0.5) {5};
\draw (0, 0) -- ++ (0,1);
\draw (0, 0) -- ++ (1,0);
\draw (1, 0) -- ++ (0,1);
\draw (0, 1) -- ++ (1,0);
\node[rotate=0] at (0.800000000000000, 0.5) {B};
\node[rotate=0] at (0.5, 0.800000000000000) {KO};
\node[rotate=0] at (0.200000000000000, 0.5) {B};
\node[rotate=0] at (0.5, 0.200000000000000) {MO};
\end{tikzpicture}
} % end of newcommand
\newcommand\WTileVI{
\begin{tikzpicture}
[scale=0.900000000000000]
\tikzstyle{every node}=[font=\tiny]
% tile at position (x,y)=(0, 0)
\node at (0.5, 0.5) {6};
\draw (0, 0) -- ++ (0,1);
\draw (0, 0) -- ++ (1,0);
\draw (1, 0) -- ++ (0,1);
\draw (0, 1) -- ++ (1,0);
\node[rotate=0] at (0.800000000000000, 0.5) {B};
\node[rotate=0] at (0.5, 0.800000000000000) {PO};
\node[rotate=0] at (0.200000000000000, 0.5) {I};
\node[rotate=0] at (0.5, 0.200000000000000) {KO};
\end{tikzpicture}
} % end of newcommand
\newcommand\WTileVII{
\begin{tikzpicture}
[scale=0.900000000000000]
\tikzstyle{every node}=[font=\tiny]
% tile at position (x,y)=(0, 0)
\node at (0.5, 0.5) {7};
\draw (0, 0) -- ++ (0,1);
\draw (0, 0) -- ++ (1,0);
\draw (1, 0) -- ++ (0,1);
\draw (0, 1) -- ++ (1,0);
\node[rotate=0] at (0.800000000000000, 0.5) {B};
\node[rotate=0] at (0.5, 0.800000000000000) {PO};
\node[rotate=0] at (0.200000000000000, 0.5) {G};
\node[rotate=0] at (0.5, 0.200000000000000) {KO};
\end{tikzpicture}
} % end of newcommand
\newcommand\WTileVIII{
\begin{tikzpicture}
[scale=0.900000000000000]
\tikzstyle{every node}=[font=\tiny]
% tile at position (x,y)=(0, 0)
\node at (0.5, 0.5) {8};
\draw (0, 0) -- ++ (0,1);
\draw (0, 0) -- ++ (1,0);
\draw (1, 0) -- ++ (0,1);
\draw (0, 1) -- ++ (1,0);
\node[rotate=90] at (0.800000000000000, 0.5) {IH};
\node[rotate=0] at (0.5, 0.800000000000000) {K};
\node[rotate=90] at (0.200000000000000, 0.5) {GF};
\node[rotate=0] at (0.5, 0.200000000000000) {K};
\end{tikzpicture}
} % end of newcommand
\newcommand\WTileIX{
\begin{tikzpicture}
[scale=0.900000000000000]
\tikzstyle{every node}=[font=\tiny]
% tile at position (x,y)=(0, 0)
\node at (0.5, 0.5) {9};
\draw (0, 0) -- ++ (0,1);
\draw (0, 0) -- ++ (1,0);
\draw (1, 0) -- ++ (0,1);
\draw (0, 1) -- ++ (1,0);
\node[rotate=90] at (0.800000000000000, 0.5) {ID};
\node[rotate=0] at (0.5, 0.800000000000000) {M};
\node[rotate=90] at (0.200000000000000, 0.5) {AF};
\node[rotate=0] at (0.5, 0.200000000000000) {K};
\end{tikzpicture}
} % end of newcommand
\newcommand\WTileX{
\begin{tikzpicture}
[scale=0.900000000000000]
\tikzstyle{every node}=[font=\tiny]
% tile at position (x,y)=(0, 0)
\node at (0.5, 0.5) {10};
\draw (0, 0) -- ++ (0,1);
\draw (0, 0) -- ++ (1,0);
\draw (1, 0) -- ++ (0,1);
\draw (0, 1) -- ++ (1,0);
\node[rotate=90] at (0.800000000000000, 0.5) {ID};
\node[rotate=0] at (0.5, 0.800000000000000) {M};
\node[rotate=90] at (0.200000000000000, 0.5) {BF};
\node[rotate=0] at (0.5, 0.200000000000000) {M};
\end{tikzpicture}
} % end of newcommand
\newcommand\WTileXI{
\begin{tikzpicture}
[scale=0.900000000000000]
\tikzstyle{every node}=[font=\tiny]
% tile at position (x,y)=(0, 0)
\node at (0.5, 0.5) {11};
\draw (0, 0) -- ++ (0,1);
\draw (0, 0) -- ++ (1,0);
\draw (1, 0) -- ++ (0,1);
\draw (0, 1) -- ++ (1,0);
\node[rotate=90] at (0.800000000000000, 0.5) {IJ};
\node[rotate=0] at (0.5, 0.800000000000000) {M};
\node[rotate=90] at (0.200000000000000, 0.5) {GF};
\node[rotate=0] at (0.5, 0.200000000000000) {K};
\end{tikzpicture}
} % end of newcommand
\newcommand\WTileXII{
\begin{tikzpicture}
[scale=0.900000000000000]
\tikzstyle{every node}=[font=\tiny]
% tile at position (x,y)=(0, 0)
\node at (0.5, 0.5) {12};
\draw (0, 0) -- ++ (0,1);
\draw (0, 0) -- ++ (1,0);
\draw (1, 0) -- ++ (0,1);
\draw (0, 1) -- ++ (1,0);
\node[rotate=90] at (0.800000000000000, 0.5) {AF};
\node[rotate=0] at (0.5, 0.800000000000000) {KO};
\node[rotate=90] at (0.200000000000000, 0.5) {ID};
\node[rotate=0] at (0.5, 0.200000000000000) {KO};
\end{tikzpicture}
} % end of newcommand
\newcommand\WTileXIII{
\begin{tikzpicture}
[scale=0.900000000000000]
\tikzstyle{every node}=[font=\tiny]
% tile at position (x,y)=(0, 0)
\node at (0.5, 0.5) {13};
\draw (0, 0) -- ++ (0,1);
\draw (0, 0) -- ++ (1,0);
\draw (1, 0) -- ++ (0,1);
\draw (0, 1) -- ++ (1,0);
\node[rotate=90] at (0.800000000000000, 0.5) {AF};
\node[rotate=0] at (0.5, 0.800000000000000) {KO};
\node[rotate=90] at (0.200000000000000, 0.5) {GF};
\node[rotate=0] at (0.5, 0.200000000000000) {KO};
\end{tikzpicture}
} % end of newcommand
\newcommand\WTileXIV{
\begin{tikzpicture}
[scale=0.900000000000000]
\tikzstyle{every node}=[font=\tiny]
% tile at position (x,y)=(0, 0)
\node at (0.5, 0.5) {14};
\draw (0, 0) -- ++ (0,1);
\draw (0, 0) -- ++ (1,0);
\draw (1, 0) -- ++ (0,1);
\draw (0, 1) -- ++ (1,0);
\node[rotate=90] at (0.800000000000000, 0.5) {GF};
\node[rotate=0] at (0.5, 0.800000000000000) {KO};
\node[rotate=90] at (0.200000000000000, 0.5) {ID};
\node[rotate=0] at (0.5, 0.200000000000000) {PO};
\end{tikzpicture}
} % end of newcommand
\newcommand\WTileXV{
\begin{tikzpicture}
[scale=0.900000000000000]
\tikzstyle{every node}=[font=\tiny]
% tile at position (x,y)=(0, 0)
\node at (0.5, 0.5) {15};
\draw (0, 0) -- ++ (0,1);
\draw (0, 0) -- ++ (1,0);
\draw (1, 0) -- ++ (0,1);
\draw (0, 1) -- ++ (1,0);
\node[rotate=90] at (0.800000000000000, 0.5) {GF};
\node[rotate=0] at (0.5, 0.800000000000000) {KO};
\node[rotate=90] at (0.200000000000000, 0.5) {GF};
\node[rotate=0] at (0.5, 0.200000000000000) {PO};
\end{tikzpicture}
} % end of newcommand
\newcommand\WTileXVI{
\begin{tikzpicture}
[scale=0.900000000000000]
\tikzstyle{every node}=[font=\tiny]
% tile at position (x,y)=(0, 0)
\node at (0.5, 0.5) {16};
\draw (0, 0) -- ++ (0,1);
\draw (0, 0) -- ++ (1,0);
\draw (1, 0) -- ++ (0,1);
\draw (0, 1) -- ++ (1,0);
\node[rotate=90] at (0.800000000000000, 0.5) {GF};
\node[rotate=0] at (0.5, 0.800000000000000) {PO};
\node[rotate=90] at (0.200000000000000, 0.5) {IH};
\node[rotate=0] at (0.5, 0.200000000000000) {PL};
\end{tikzpicture}
} % end of newcommand
\newcommand\WTileXVII{
\begin{tikzpicture}
[scale=0.900000000000000]
\tikzstyle{every node}=[font=\tiny]
% tile at position (x,y)=(0, 0)
\node at (0.5, 0.5) {17};
\draw (0, 0) -- ++ (0,1);
\draw (0, 0) -- ++ (1,0);
\draw (1, 0) -- ++ (0,1);
\draw (0, 1) -- ++ (1,0);
\node[rotate=90] at (0.800000000000000, 0.5) {GF};
\node[rotate=0] at (0.5, 0.800000000000000) {PO};
\node[rotate=90] at (0.200000000000000, 0.5) {IJ};
\node[rotate=0] at (0.5, 0.200000000000000) {PO};
\end{tikzpicture}
} % end of newcommand
\newcommand\WTileXVIII{
\begin{tikzpicture}
[scale=0.900000000000000]
\tikzstyle{every node}=[font=\tiny]
% tile at position (x,y)=(0, 0)
\node at (0.5, 0.5) {18};
\draw (0, 0) -- ++ (0,1);
\draw (0, 0) -- ++ (1,0);
\draw (1, 0) -- ++ (0,1);
\draw (0, 1) -- ++ (1,0);
\node[rotate=90] at (0.800000000000000, 0.5) {BF};
\node[rotate=0] at (0.5, 0.800000000000000) {MO};
\node[rotate=90] at (0.200000000000000, 0.5) {GF};
\node[rotate=0] at (0.5, 0.200000000000000) {KO};
\end{tikzpicture}
} % end of newcommand

\begin{proof}
    From Lemma~\ref{lem:markersforV},
    $R = \{v_0, v_1, v_3, v_8, v_9, v_{14}, v_{15}\}$
    is a set of markers in the direction $\be_1$ for~$\V$.
    Therefore, Theorem~\ref{thm:exist-homeo} applies.
    Let $P_1$ be the set
    \begin{align*}
        P_1 &= \left\{u\boxminus^1 v \mid 
            u\in\V\setminus R, v\in R 
            \text{ and } 
            u\odot^1v \in D_{\V,\be_1,1}
            \right\}\\
          &=\left\{ \begin{array}{l}
                        
\begin{array}{|cc|}\hline v_{2} & v_{3} \\ \hline \end{array}
\,,\,
\begin{array}{|cc|}\hline v_{4} & v_{1} \\ \hline \end{array}
\,,\,
\begin{array}{|cc|}\hline v_{5} & v_{1} \\ \hline \end{array}
\,,\,
\begin{array}{|cc|}\hline v_{6} & v_{1} \\ \hline \end{array}
\,,\,
\begin{array}{|cc|}\hline v_{7} & v_{0} \\ \hline \end{array}
\,,\,
\begin{array}{|cc|}\hline v_{7} & v_{1} \\ \hline \end{array}
\,,\,
\begin{array}{|cc|}\hline v_{10} & v_{14} \\ \hline \end{array}
\,,\,\\[2mm]
\begin{array}{|cc|}\hline v_{11} & v_{15} \\ \hline \end{array}
\,,\,
\begin{array}{|cc|}\hline v_{12} & v_{15} \\ \hline \end{array}
\,,\,
\begin{array}{|cc|}\hline v_{13} & v_{15} \\ \hline \end{array}
\,,\,
\begin{array}{|cc|}\hline v_{18} & v_{8} \\ \hline \end{array}
\,,\,
\begin{array}{|cc|}\hline v_{19} & v_{8} \\ \hline \end{array}
\,,\,
\begin{array}{|cc|}\hline v_{20} & v_{9} \\ \hline \end{array}

                        \end{array} \right\}
    \end{align*}
    according to Lemma~\ref{lem:dominoes-direction1-V}.
    Let $K_1$ be the set
    \begin{align*}
        K_1 &= \left\{u\in\V\setminus R\mid \text{ there exists } v\in\V\setminus
            R \text{ such that } (u,v) \in D_{\V,\be_1,1}\right\}\\
          &= \{u_8, u_9, u_{11}, u_{13}, u_{14}, u_{15}, u_{16}, u_{17}\}.
    \end{align*}
    Let $\Omega_\W$ be the Wang shift of the Wang tile set $\W=K_1\cup P_1$ 
    of cardinality~19:
\begin{equation*}\label{eq:tiles9}
\begin{array}{c}
\W= \left\{ 
w_0=\hspace{-2mm}\raisebox{-3mm}{\WTileO},
w_1=\hspace{-2mm}\raisebox{-3mm}{\WTileI},
w_2=\hspace{-2mm}\raisebox{-3mm}{\WTileII},
w_3=\hspace{-2mm}\raisebox{-3mm}{\WTileIII},
w_4=\hspace{-2mm}\raisebox{-3mm}{\WTileIV},
w_5=\hspace{-2mm}\raisebox{-3mm}{\WTileV},
w_6=\hspace{-2mm}\raisebox{-3mm}{\WTileVI},
\right.\\
\hspace{1cm}\left.
w_7=\hspace{-2mm}\raisebox{-3mm}{\WTileVII},
w_8=\hspace{-2mm}\raisebox{-3mm}{\WTileVIII},
w_9=\hspace{-2mm}\raisebox{-3mm}{\WTileIX},
w_{10}=\hspace{-2mm}\raisebox{-3mm}{\WTileX},
w_{11}=\hspace{-2mm}\raisebox{-3mm}{\WTileXI},
w_{12}=\hspace{-2mm}\raisebox{-3mm}{\WTileXII},
\right.\\
\hspace{1cm}\left.
w_{13}=\hspace{-2mm}\raisebox{-3mm}{\WTileXIII},
w_{14}=\hspace{-2mm}\raisebox{-3mm}{\WTileXIV},
w_{15}=\hspace{-2mm}\raisebox{-3mm}{\WTileXV},
w_{16}=\hspace{-2mm}\raisebox{-3mm}{\WTileXVI},
w_{17}=\hspace{-2mm}\raisebox{-3mm}{\WTileXVII},
w_{18}=\hspace{-2mm}\raisebox{-3mm}{\WTileXVIII}
\right\}.
\end{array}
\end{equation*}
    Let $\beta:\Omega_\W\to\Omega_\V$ be the $2$-dimensional morphism
    defined by
    \begin{equation*}
        \beta:
    \begin{cases}
        u\boxminus^1 v
        \quad \mapsto \quad
        u\odot^1 v
        & \text{ if }
        u\boxminus^1 v \in P_1\\
        k \quad \mapsto \quad
        k & \text{ if } k \in K_1.
    \end{cases}
    \end{equation*}
    Then $\beta:\Omega_\W\to\Omega_\V$
    is recognizable in $\Omega_\W$
    and $\beta(\Omega_\W)\cup\sigma^{\be_1}\beta(\Omega_\W)=\Omega_\V$.
\end{proof}

The morphism $\beta$ in terms of the
$\V=\{v_i\}_{0\leq i\leq 20}$ and $\W=\{w_i\}_{0\leq i\leq 18}$ 
is in Figure~\ref{fig:beta}.
The tile set $\W$ uses the
following alphabet of 10 vertical colors:
\begin{equation*}
\{A, B, G, I, AF, BF, GF, ID, IH, IJ\}
\end{equation*}
and the following alphabet of 6 horizontal colors:
\begin{equation*}
\{K, M, KO, MO, PL, PO\}.
\end{equation*}
Notice that the set of horizontal left and right colors has increased in size
compared to $\V$ and that the set of vertical bottom and top colors
has \emph{decreased} in size. Indeed, colors E, CH and EH have disappeared.

\section{Self-similarity and aperiodicity of $\Omega_\U$}\label{sec:aperiodicityofU}

Let $\gamma:\U\to\W$ be the map defined
by the rule $u_i\mapsto w_i$ for every $i\in\{0,1,\dots,18\}$.
On tiles, the effect is shown below where the position of tile $u_i$ in the
table on the left is the same as the position of the tile $w_i$ 
in the table on the right:
\begin{equation*}
    \gamma:
\left\{
    \arraycolsep=.5pt
\begin{array}{ccccc}
 \raisebox{-3mm}{\UTileO}    
&\raisebox{-3mm}{\UTileI}   
&\raisebox{-3mm}{\UTileII}  
&\raisebox{-3mm}{\UTileIII} 
&\raisebox{-3mm}{\UTileIV}  
    \\                   
 \raisebox{-3mm}{\UTileV}    
&\raisebox{-3mm}{\UTileVI}  
&\raisebox{-3mm}{\UTileVII} 
&\raisebox{-3mm}{\UTileVIII}
&\raisebox{-3mm}{\UTileIX}  
    \\                   
 \raisebox{-3mm}{\UTileX}    
&\raisebox{-3mm}{\UTileXI}  
&\raisebox{-3mm}{\UTileXII} 
&\raisebox{-3mm}{\UTileXIII}
&\raisebox{-3mm}{\UTileXIV} 
    \\                   
 \raisebox{-3mm}{\UTileXV}   
&\raisebox{-3mm}{\UTileXVI} 
&\raisebox{-3mm}{\UTileXVII}
&\raisebox{-3mm}{\UTileXVIII}
\end{array}
\mapsto
\begin{array}{ccccc}
 \raisebox{-3mm}{\WTileO}
&\raisebox{-3mm}{\WTileI}
&\raisebox{-3mm}{\WTileII}
&\raisebox{-3mm}{\WTileIII}
&\raisebox{-3mm}{\WTileIV}
\\
 \raisebox{-3mm}{\WTileV}
&\raisebox{-3mm}{\WTileVI}
&\raisebox{-3mm}{\WTileVII}
&\raisebox{-3mm}{\WTileVIII}
&\raisebox{-3mm}{\WTileIX}
\\
 \raisebox{-3mm}{\WTileX}
&\raisebox{-3mm}{\WTileXI}
&\raisebox{-3mm}{\WTileXII}
&\raisebox{-3mm}{\WTileXIII}
&\raisebox{-3mm}{\WTileXIV}
\\
 \raisebox{-3mm}{\WTileXV}
&\raisebox{-3mm}{\WTileXVI}
&\raisebox{-3mm}{\WTileXVII}
&\raisebox{-3mm}{\WTileXVIII}
\end{array}
\right.
\end{equation*}

\begin{lemma}\label{lem:homeoUtoW}
$\U$ and $\W$ are equivalent Wang tile sets.
The map $\gamma:\U\to\W$ defines a
$d$-dimensional morphism $\gamma:\Omega_\U\to\Omega_\W$
which is a bijection.
\end{lemma}

\begin{proof}
    We consider the following bijection between the colors of $\U$
and the colors of $\W$:
    \[
h:
\left\{
\begin{array}{ll}
 O\mapsto K& P\mapsto KO\\
 L\mapsto M& N\mapsto MO\\
           & M\mapsto PL\\
           & K\mapsto PO\\
\end{array}\right.,
\qquad
k:
\left\{
\begin{array}{ll}
J\mapsto A&E\mapsto AF\\
H\mapsto B&C\mapsto BF\\
D\mapsto G&I\mapsto GF\\
F\mapsto I&G\mapsto ID\\
          &B\mapsto IH\\
          &A\mapsto IJ\\
\end{array}\right.,
\]
where $h$ relabels the horizontal (bottom and top) colors
and
where $k$ relabels the vertical (left and right) colors.
    For every $i$ in $\{0,1,\ldots, 18\}$, we observe
    that the application of the map $h$ 
    on the bottom and left colors 
    and $k$ on the left and right colors
    transforms bijectively $u_i$ into $w_i$:
    \[
        \gamma(\U) 
        = \{(k(a),h(b),k(c),h(d))\mid(a,b,c,d)\in\U\}
        = \W.
    \]
Thus $\gamma$ is a bijection and $\U$ and $\W$ are equivalent Wang tile sets.
\end{proof}

Consider the $d$-dimensional morphism 
$\omega:\Omega_\U\to\Omega_\U$ defined as
$\omega=\alpha\circ\beta\circ\gamma$ shown in Figure~\ref{fig:omega}.

\begin{figure}[h]
\begin{center}
    \input{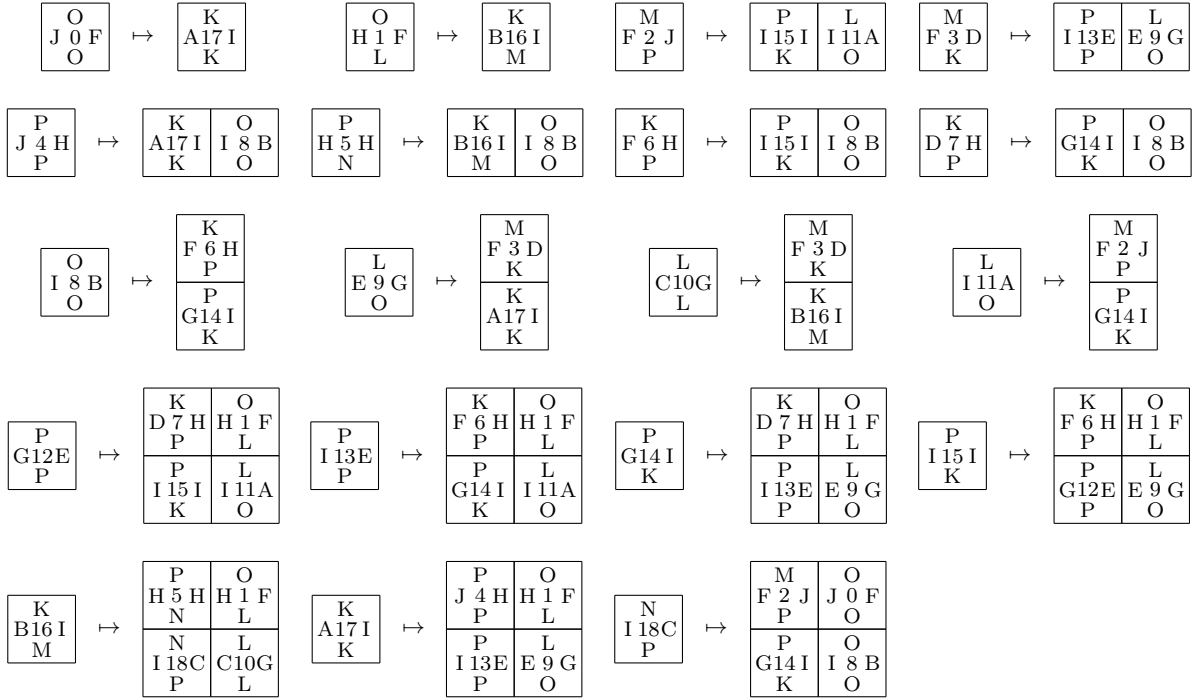}
\end{center}
    \caption{The morphism $\omega:\Omega_\U\to\Omega_\U$.
    The image of each tile $\omega(u_i)$ corresponds to the supertiles
    shown in Figure~\ref{fig:Usupertiles}.}
    \label{fig:omega}
\end{figure}

% sage: z = polygen(QQ, 'z') #z = QQ['z'].0 # same as
% sage: K = NumberField(z**2-z-1, 'phi', embedding=RR(1.6))
% sage: phi = K.gen()
% sage: M = matrix(gamma).change_ring(K)
% sage: M.charpoly().factor()
% (x - phi - 1) * (x + phi - 2) * (x - phi + 1)^3 * x^3 * (x + phi)^3 * (x - 1)^4 * (x + 1)^4
% sage: M.change_ring(ZZ).charpoly().factor()
% x^3 * (x - 1)^4 * (x + 1)^4 * (x^2 - 3*x + 1) * (x^2 + x - 1)^3
% sage: eigv = MK.eigenvectors_right()
% sage: eigv[0]
% (phi + 1, [
% (1, 2*phi, 3*phi, phi + 1, phi, -phi + 2, 1, phi - 1, 1, phi + 1, phi, phi, phi - 1, phi + 1, phi + 2, 1, phi + 1, phi, phi)
% ], 1)
% sage: eigvl = MK.eigenvectors_left()
% sage: eigvl[0]
% (phi + 1, [ (1, 1, phi - 1, 1, 1, phi - 1, phi, 1, 1, phi, phi, 1, phi, 1, phi,
% 1, phi, phi, 1) ], 1)

\begin{lemma}\label{lem:omega-is-primitive}
    The morphism $\omega$ is primitive.
    The characteristic polynomial of its incidence matrix $M$ is
    \[
    \chi_M(x) = x^{3} \cdot (x - 1)^{4} \cdot (x + 1)^{4} \cdot (x^{2} - 3x +
    1) \cdot (x^{2} + x - 1)^{3}.
    \]
    The dominant eigenvalue is
    $\varphi^2=\varphi+1=(3+\sqrt{5})/2$.
    The associated positive right eigenvector is
    \[
    \left(1, 3\varphi^3, \varphi^2, \varphi^3, \varphi, \varphi^2, \varphi^4,
    \varphi^3, \varphi^4, 2\varphi^3, \varphi^2, \varphi^3, \varphi,
    \varphi^4, \varphi^4+\varphi^2, \varphi^3, \varphi^4, \varphi^3, \varphi^2
    \right)^t
    \]
% sage: phi^2
% phi + 1
% sage: phi^3
% 2*phi + 1
% sage: phi^4
% 3*phi + 2
% sage: phi^5
% 5*phi + 3
% sage: 3*phi^3
% 6*phi + 3
% sage: 2*phi^3
% 4*phi + 2
% sage: phi^4 + phi^2
% 4*phi + 3
    and associated positive left eigenvector is
    \[
        \left(1, 1, 
        \varphi, \varphi, \varphi, \varphi, \varphi, \varphi, \varphi, \varphi, \varphi, \varphi, 
        \varphi^2, \varphi^2, \varphi^2, \varphi^2, \varphi^2, \varphi^2, \varphi^2\right).
    \]
\end{lemma}

\begin{proof}
%sage: latex(matrix(omega))
The incidence matrix $M$ of $\omega$ is
\begin{equation*}
\left(\begin{array}{rrrrrrrrrrrrrrrrrrr}
0 & 0 & 0 & 0 & 0 & 0 & 0 & 0 & 0 & 0 & 0 & 0 & 0 & 0 & 0 & 0 & 0 & 0 & 1 \\
0 & 0 & 0 & 0 & 0 & 0 & 0 & 0 & 0 & 0 & 0 & 0 & 1 & 1 & 1 & 1 & 1 & 1 & 0 \\
0 & 0 & 0 & 0 & 0 & 0 & 0 & 0 & 0 & 0 & 0 & 1 & 0 & 0 & 0 & 0 & 0 & 0 & 1 \\
0 & 0 & 0 & 0 & 0 & 0 & 0 & 0 & 0 & 1 & 1 & 0 & 0 & 0 & 0 & 0 & 0 & 0 & 0 \\
0 & 0 & 0 & 0 & 0 & 0 & 0 & 0 & 0 & 0 & 0 & 0 & 0 & 0 & 0 & 0 & 0 & 1 & 0 \\
0 & 0 & 0 & 0 & 0 & 0 & 0 & 0 & 0 & 0 & 0 & 0 & 0 & 0 & 0 & 0 & 1 & 0 & 0 \\
0 & 0 & 0 & 0 & 0 & 0 & 0 & 0 & 1 & 0 & 0 & 0 & 0 & 1 & 0 & 1 & 0 & 0 & 0 \\
0 & 0 & 0 & 0 & 0 & 0 & 0 & 0 & 0 & 0 & 0 & 0 & 1 & 0 & 1 & 0 & 0 & 0 & 0 \\
0 & 0 & 0 & 0 & 1 & 1 & 1 & 1 & 0 & 0 & 0 & 0 & 0 & 0 & 0 & 0 & 0 & 0 & 1 \\
0 & 0 & 0 & 1 & 0 & 0 & 0 & 0 & 0 & 0 & 0 & 0 & 0 & 0 & 1 & 1 & 0 & 1 & 0 \\
0 & 0 & 0 & 0 & 0 & 0 & 0 & 0 & 0 & 0 & 0 & 0 & 0 & 0 & 0 & 0 & 1 & 0 & 0 \\
0 & 0 & 1 & 0 & 0 & 0 & 0 & 0 & 0 & 0 & 0 & 0 & 1 & 1 & 0 & 0 & 0 & 0 & 0 \\
0 & 0 & 0 & 0 & 0 & 0 & 0 & 0 & 0 & 0 & 0 & 0 & 0 & 0 & 0 & 1 & 0 & 0 & 0 \\
0 & 0 & 0 & 1 & 0 & 0 & 0 & 0 & 0 & 0 & 0 & 0 & 0 & 0 & 1 & 0 & 0 & 1 & 0 \\
0 & 0 & 0 & 0 & 0 & 0 & 0 & 1 & 1 & 0 & 0 & 1 & 0 & 1 & 0 & 0 & 0 & 0 & 1 \\
0 & 0 & 1 & 0 & 0 & 0 & 1 & 0 & 0 & 0 & 0 & 0 & 1 & 0 & 0 & 0 & 0 & 0 & 0 \\
0 & 1 & 0 & 0 & 0 & 1 & 0 & 0 & 0 & 0 & 1 & 0 & 0 & 0 & 0 & 0 & 0 & 0 & 0 \\
1 & 0 & 0 & 0 & 1 & 0 & 0 & 0 & 0 & 1 & 0 & 0 & 0 & 0 & 0 & 0 & 0 & 0 & 0 \\
0 & 0 & 0 & 0 & 0 & 0 & 0 & 0 & 0 & 0 & 0 & 0 & 0 & 0 & 0 & 0 & 1 & 0 & 0
\end{array}\right)
\end{equation*}
whose $7$-th power $M^7$ is positive. Therefore $M$ is a primitive matrix
and the morphism is primitive.
The computation of characteristic polynomial, eigenvalues and eigenvectors of $M$ is done using Sage
in Section~\ref{sec:appendix}.
\end{proof}

We now prove the main result using the composition-decomposition approach.
Informally, the morphism $\alpha\beta:\Omega_\W\to\Omega_\U$ allows to group
tiles in any tilings admitted by $\U$ in a unique way into supertiles of the
form $\alpha\beta(w_i)$ for $i\in\{0,1,\dots,18\}$, thus satisfying condition 
\ref{cond:C1}.
The morphism $\gamma:\Omega_\U\to\Omega_\W$ implies that
the markings on the supertiles $\W$
imply a matching condition for the supertiles which is exactly equivalent to
that originally specified for the tiles $\U$, thus satisfying condition
\ref{cond:C2}.

\begin{proposition}\label{prop:pre-main}
The $d$-dimensional morphism 
$\omega:\Omega_\U\to\Omega_\U$ defined as
$\omega=\alpha\circ\beta\circ\gamma$ is
expansive, recognizable in $\Omega_\U$
and satisfies
\begin{equation}\label{eq:omega-is-kinda-onto}
    \Omega_\U = \omega(\Omega_\U)
    \cup\sigma^{\be_1}\omega(\Omega_\U)
    \cup\sigma^{\be_2}\omega(\Omega_\U)
    \cup\sigma^{\be_1+\be_2}\omega(\Omega_\U).
\end{equation}
\end{proposition}

\begin{proof}
From Proposition~\ref{prop:wecandesubstituteU},
Proposition~\ref{prop:wecandesubstituteV} and
Lemma~\ref{lem:homeoUtoW},
we have the sequence of recognizable $2$-dimensional morphisms:
\[
    \Omega_\U
    \xleftarrow{\alpha}
    \Omega_\V
    \xleftarrow{\beta}
    \Omega_\W
    \xleftarrow{\gamma}
    \Omega_\U.
\]
and the equations
$\alpha(\Omega_\V)\cup\sigma^{\be_2}\alpha(\Omega_\V)=\Omega_\U$
and
$\beta(\Omega_\W)\cup\sigma^{\be_1}\beta(\Omega_\W)=\Omega_\V$.
Consider the $d$-dimensional morphism $\omega=\alpha\beta\gamma$ which
is recognizable in $\Omega_\U$.
Since $\alpha\sigma^{\be_1}= \sigma^{\be_1}\alpha$, we deduce Equation
\eqref{eq:omega-is-kinda-onto}.
As shown in  Lemma~\ref{lem:omega-is-primitive}, $\omega$ is primitive
and sends at least one letter to a $2$-dimensional word of shape $(2,2)$.
Thus it is expansive.
\end{proof}

We can now prove the aperiodicity of $\U$ from the
recognizability of $\omega$.

\begin{corollary}\label{cor:OmegaU-is-aperiodic}
    $\overline{\omega(\Omega_\U)}^{\sigma}=\Omega_\U$,
    $\Omega_\U$ is self-similar and aperiodic.
\end{corollary}

\begin{proof}
    From Proposition~\ref{prop:pre-main}, $\omega$ is expansive and
    recognizable in $\Omega_\U$.
    By taking the closure under the shift $\sigma$ on both sides of
    Equation~\eqref{eq:omega-is-kinda-onto}, we deduce that
    $\Omega_\U=\overline{\Omega_\U}^\sigma
    =\overline{\omega(\Omega_\U)}^{\sigma}$.
    Thus $\Omega_\U$ is self-similar.
    %$\L(\Omega_\U) = \omega(\L(\Omega_\U))$.
    We conclude aperiodicity of $\Omega_\U$ from
    Proposition~\ref{prop:expansive-recognizable-aperiodic}.
\end{proof}

The next lemma shows that the subwords of shape $(2,2)$ in 
$\Omega_\U$ and $\X_\omega$ are the same.

\begin{lemma}\label{lem:50tiles-2x2-in-OmegaU}
    We have
    $\L(\X_\omega)\cap \U^{(2,2)} = \L(\Omega_\U)\cap \U^{(2,2)}=S$
    where $S$ is
    \begin{equation*}\arraycolsep=1pt
        \left\{\begin{array}{cccccccc}
\left(\begin{array}{rr}
u_{8} & u_{16} \\
u_{0} & u_{3}
\end{array}\right)
,&
\left(\begin{array}{rr}
u_{8} & u_{16} \\
u_{1} & u_{2}
\end{array}\right)
,&
\left(\begin{array}{rr}
u_{8} & u_{16} \\
u_{1} & u_{3}
\end{array}\right)
,&
\left(\begin{array}{rr}
u_{9} & u_{14} \\
u_{1} & u_{6}
\end{array}\right)
,&
\left(\begin{array}{rr}
u_{11} & u_{17} \\
u_{1} & u_{6}
\end{array}\right)
,&
\left(\begin{array}{rr}
u_{16} & u_{8} \\
u_{2} & u_{0}
\end{array}\right)
,&
\left(\begin{array}{rr}
u_{16} & u_{13} \\
u_{2} & u_{4}
\end{array}\right)
,&
\left(\begin{array}{rr}
u_{16} & u_{15} \\
u_{3} & u_{7}
\end{array}\right)
,\\[4mm]
\left(\begin{array}{rr}
u_{13} & u_{9} \\
u_{4} & u_{1}
\end{array}\right)
,&
\left(\begin{array}{rr}
u_{13} & u_{9} \\
u_{5} & u_{1}
\end{array}\right)
,&
\left(\begin{array}{rr}
u_{14} & u_{8} \\
u_{6} & u_{1}
\end{array}\right)
,&
\left(\begin{array}{rr}
u_{14} & u_{11} \\
u_{6} & u_{1}
\end{array}\right)
,&
\left(\begin{array}{rr}
u_{14} & u_{13} \\
u_{6} & u_{5}
\end{array}\right)
,&
\left(\begin{array}{rr}
u_{17} & u_{8} \\
u_{6} & u_{1}
\end{array}\right)
,&
\left(\begin{array}{rr}
u_{17} & u_{13} \\
u_{6} & u_{5}
\end{array}\right)
,&
\left(\begin{array}{rr}
u_{15} & u_{8} \\
u_{7} & u_{1}
\end{array}\right)
,\\[4mm]
\left(\begin{array}{rr}
u_{15} & u_{11} \\
u_{7} & u_{1}
\end{array}\right)
,&
\left(\begin{array}{rr}
u_{0} & u_{3} \\
u_{8} & u_{16}
\end{array}\right)
,&
\left(\begin{array}{rr}
u_{9} & u_{14} \\
u_{8} & u_{16}
\end{array}\right)
,&
\left(\begin{array}{rr}
u_{11} & u_{17} \\
u_{8} & u_{16}
\end{array}\right)
,&
\left(\begin{array}{rr}
u_{1} & u_{2} \\
u_{9} & u_{14}
\end{array}\right)
,&
\left(\begin{array}{rr}
u_{1} & u_{6} \\
u_{9} & u_{14}
\end{array}\right)
,&
\left(\begin{array}{rr}
u_{10} & u_{12} \\
u_{9} & u_{14}
\end{array}\right)
,&
\left(\begin{array}{rr}
u_{1} & u_{6} \\
u_{10} & u_{12}
\end{array}\right)
,\\[4mm]
\left(\begin{array}{rr}
u_{1} & u_{6} \\
u_{10} & u_{14}
\end{array}\right)
,&
\left(\begin{array}{rr}
u_{1} & u_{3} \\
u_{11} & u_{17}
\end{array}\right)
,&
\left(\begin{array}{rr}
u_{10} & u_{14} \\
u_{11} & u_{17}
\end{array}\right)
,&
\left(\begin{array}{rr}
u_{6} & u_{1} \\
u_{12} & u_{9}
\end{array}\right)
,&
\left(\begin{array}{rr}
u_{4} & u_{1} \\
u_{13} & u_{9}
\end{array}\right)
,&
\left(\begin{array}{rr}
u_{7} & u_{1} \\
u_{13} & u_{9}
\end{array}\right)
,&
\left(\begin{array}{rr}
u_{18} & u_{10} \\
u_{13} & u_{9}
\end{array}\right)
,&
\left(\begin{array}{rr}
u_{2} & u_{0} \\
u_{14} & u_{8}
\end{array}\right)
,\\[4mm]
\left(\begin{array}{rr}
u_{2} & u_{4} \\
u_{14} & u_{13}
\end{array}\right)
,&
\left(\begin{array}{rr}
u_{6} & u_{1} \\
u_{14} & u_{11}
\end{array}\right)
,&
\left(\begin{array}{rr}
u_{6} & u_{5} \\
u_{14} & u_{18}
\end{array}\right)
,&
\left(\begin{array}{rr}
u_{12} & u_{9} \\
u_{14} & u_{8}
\end{array}\right)
,&
\left(\begin{array}{rr}
u_{7} & u_{1} \\
u_{15} & u_{11}
\end{array}\right)
,&
\left(\begin{array}{rr}
u_{13} & u_{9} \\
u_{15} & u_{8}
\end{array}\right)
,&
\left(\begin{array}{rr}
u_{18} & u_{10} \\
u_{15} & u_{11}
\end{array}\right)
,&
\left(\begin{array}{rr}
u_{3} & u_{7} \\
u_{16} & u_{13}
\end{array}\right)
,\\[4mm]
\left(\begin{array}{rr}
u_{3} & u_{7} \\
u_{16} & u_{15}
\end{array}\right)
,&
\left(\begin{array}{rr}
u_{14} & u_{11} \\
u_{16} & u_{8}
\end{array}\right)
,&
\left(\begin{array}{rr}
u_{14} & u_{18} \\
u_{16} & u_{13}
\end{array}\right)
,&
\left(\begin{array}{rr}
u_{14} & u_{13} \\
u_{16} & u_{15}
\end{array}\right)
,&
\left(\begin{array}{rr}
u_{14} & u_{18} \\
u_{16} & u_{15}
\end{array}\right)
,&
\left(\begin{array}{rr}
u_{17} & u_{13} \\
u_{16} & u_{15}
\end{array}\right)
,&
\left(\begin{array}{rr}
u_{3} & u_{7} \\
u_{17} & u_{13}
\end{array}\right)
,&
\left(\begin{array}{rr}
u_{14} & u_{11} \\
u_{17} & u_{8}
\end{array}\right)
,\\[4mm]
\left(\begin{array}{rr}
u_{14} & u_{18} \\
u_{17} & u_{13}
\end{array}\right)
,&
\left(\begin{array}{rr}
u_{5} & u_{1} \\
u_{18} & u_{10}
\end{array}\right)
\end{array}\right\}
    \end{equation*}
\end{lemma}

The proof of Lemma~\ref{lem:50tiles-2x2-in-OmegaU} is done using Sage
in Section~\ref{sec:appendix}.

\begin{proposition}\label{prop:OmegaU-is-self-similar}
    $\Omega_\U$ is minimal. More precisely,
    $\Omega_\U=\X_\omega$ where $\omega=\alpha\beta\gamma$.
\end{proposition}

\begin{proof}
    From Corollary~\ref{cor:OmegaU-is-aperiodic},
    we have 
    $\overline{\omega(\Omega_\U)}^{\sigma}=\Omega_\U$.
    From Lemma~\ref{lem:existence-omega-representation},
    this is equivalent to $\L(\Omega_\U) = \overline{\omega(\L(\Omega_\U))}^{Fact}$.
    From Lemma~\ref{lem:50tiles-2x2-in-OmegaU} 
    we have
    $\L(\X_\omega)\cap \U^{(2,2)} = \L(\Omega_\U)\cap \U^{(2,2)} $.
    From Lemma~\ref{lem:omega-is-primitive}
    and Proposition~\ref{prop:pre-main}, $\omega$ is primitive and expansive.
    From Lemma~\ref{lem:substitutive-equivalent-conditions},
    we conclude that $\L(\X_\omega) = \L(\Omega_\U)$,
    $\X_\omega= \Omega_\U$
    and $\Omega_\U$ is minimal.
\end{proof}

\begin{proof}[Proof of Theorem~\ref{thm:main}]
    It is proved that $\U$ is self-similar and aperiodic in
    Corollary~\ref{cor:OmegaU-is-aperiodic} and
    that $\Omega_\U$ is minimal in
    Proposition~\ref{prop:OmegaU-is-self-similar}.
\end{proof}

We deduce also that $\V$ is aperiodic from
Lemma~\ref{lem:aperiodic-implies-aperiodic} since $\alpha$ sends periodic
tilings onto periodic tilings.

\begin{proposition}
The vector of frequency of the tiles $\{u_i\}_{0\leq i\leq 18}$ in
    any tiling of plane by $\U$ is
    \[
\frac{1}{2}
    \left(
 \frac{1}{\varphi^8},
 \frac{3}{\varphi^5},
 \frac{1}{\varphi^6},
 \frac{1}{\varphi^5},
 \frac{1}{\varphi^7},
 \frac{1}{\varphi^6},
 \frac{1}{\varphi^4},
 \frac{1}{\varphi^5},
 \frac{1}{\varphi^4},
 \frac{2}{\varphi^5},
 \frac{1}{\varphi^6},
 \frac{1}{\varphi^5},
 \frac{1}{\varphi^7},
 \frac{1}{\varphi^4},
 \frac{1}{\varphi^4}+\frac{1}{\varphi^6},
 \frac{1}{\varphi^5},
 \frac{1}{\varphi^4},
 \frac{1}{\varphi^5},
 \frac{1}{\varphi^6}\right)^t.
    \]
% sage: (1/2/phi^6).n()
% 0.0278640450004204
% sage: (1/2/phi^5*2).n()
% 0.0901699437494745
% sage: (1/2/phi^5*3).n()
% 0.135254915624213
% sage: (1/2/phi^4).n()
% 0.0729490168751576
% sage: (1/2/phi^5).n()
% 0.0450849718747373
% sage: (1/2/phi^8).n()
% 0.0106431181261044
% sage: (1/2/phi^7).n()
% 0.0172209268743160
% sage: (1/2*(1/phi^4+1/phi^6)).n()
% 0.100813061875577
where
\[
\frac{1}{2\varphi^4}\approx
0.0729,\quad
\frac{1}{2\varphi^5}\approx
0.0451,\quad
\frac{1}{2\varphi^6}\approx
0.0279,\quad
\frac{1}{2\varphi^7}\approx
0.0172,\quad
\frac{1}{2\varphi^8}\approx
0.0106,
\]
and
\[
\frac{2}{2\varphi^5}\approx
0.0902,\quad
\frac{3}{2\varphi^5}\approx
0.1353,\quad
\frac{1}{2\varphi^4}+\frac{1}{2\varphi^6}\approx
0.1008.
\]
\end{proposition}

\begin{proof}
The sum of the values of the right eigenvector given in Lemma~\ref{lem:omega-is-primitive}
is $2\varphi^8=42\varphi+26$.
Therefore, the frequency vector is $\frac{1}{2\varphi^8}$ times the right eigenvector.
\end{proof}

We remark that the morphism $\omega$ on $\U$ is not
prolongable on any Wang tile but its square $\omega^2$ is prolongable on
many of them (see Figure~\ref{fig:omega_prolongable}).

\begin{figure}[h]
\begin{center}
\begin{tikzpicture}
[auto,scale=0.9]

\node (A0) at (0,0) {
\begin{tikzpicture}
\tikzstyle{every node}=[font=\tiny]
% tile at position (x,y)=(0, 0)
\node at (0.5, 0.5) {16};
\draw (0, 0) -- ++ (0,1);
\draw (0, 0) -- ++ (1,0);
\draw (1, 0) -- ++ (0,1);
\draw (0, 1) -- ++ (1,0);
\node[rotate=0] at (0.8, 0.5) {I};
\node[rotate=0] at (0.5, 0.8) {K};
\node[rotate=0] at (0.2, 0.5) {B};
\node[rotate=0] at (0.5, 0.2) {M};
\end{tikzpicture}
};

\node (A1) at (3.0,0) {
\begin{tikzpicture}
\tikzstyle{every node}=[font=\tiny]
% tile at position (x,y)=(0, 0)
\node at (0.5, 0.5) {18};
\draw (0, 0) -- ++ (0,1);
\draw (0, 0) -- ++ (1,0);
\draw (1, 0) -- ++ (0,1);
\draw (0, 1) -- ++ (1,0);
\node[rotate=0] at (0.8, 0.5) {C};
\node[rotate=0] at (0.5, 0.8) {N};
\node[rotate=0] at (0.2, 0.5) {I};
\node[rotate=0] at (0.5, 0.2) {P};
% tile at position (x,y)=(0, 1)
\node at (0.5, 1.5) {5};
\draw (0, 1) -- ++ (0,1);
\draw (0, 1) -- ++ (1,0);
\draw (1, 1) -- ++ (0,1);
\draw (0, 2) -- ++ (1,0);
\node[rotate=0] at (0.8, 1.5) {H};
\node[rotate=0] at (0.5, 1.8) {P};
\node[rotate=0] at (0.2, 1.5) {H};
\node[rotate=0] at (0.5, 1.2) {N};
% tile at position (x,y)=(1, 0)
\node at (1.5, 0.5) {10};
\draw (1, 0) -- ++ (0,1);
\draw (1, 0) -- ++ (1,0);
\draw (2, 0) -- ++ (0,1);
\draw (1, 1) -- ++ (1,0);
\node[rotate=0] at (1.8, 0.5) {G};
\node[rotate=0] at (1.5, 0.8) {L};
\node[rotate=0] at (1.2, 0.5) {C};
\node[rotate=0] at (1.5, 0.2) {L};
% tile at position (x,y)=(1, 1)
\node at (1.5, 1.5) {1};
\draw (1, 1) -- ++ (0,1);
\draw (1, 1) -- ++ (1,0);
\draw (2, 1) -- ++ (0,1);
\draw (1, 2) -- ++ (1,0);
\node[rotate=0] at (1.8, 1.5) {F};
\node[rotate=0] at (1.5, 1.8) {O};
\node[rotate=0] at (1.2, 1.5) {H};
\node[rotate=0] at (1.5, 1.2) {L};
\end{tikzpicture}};

\node (A2) at (7.0,0) {
\begin{tikzpicture}
\tikzstyle{every node}=[font=\tiny]
% tile at position (x,y)=(0, 0)
\node at (0.5, 0.5) {14};
\draw (0, 0) -- ++ (0,1);
\draw (0, 0) -- ++ (1,0);
\draw (1, 0) -- ++ (0,1);
\draw (0, 1) -- ++ (1,0);
\node[rotate=0] at (0.8, 0.5) {I};
\node[rotate=0] at (0.5, 0.8) {P};
\node[rotate=0] at (0.2, 0.5) {G};
\node[rotate=0] at (0.5, 0.2) {K};
% tile at position (x,y)=(0, 1)
\node at (0.5, 1.5) {2};
\draw (0, 1) -- ++ (0,1);
\draw (0, 1) -- ++ (1,0);
\draw (1, 1) -- ++ (0,1);
\draw (0, 2) -- ++ (1,0);
\node[rotate=0] at (0.8, 1.5) {J};
\node[rotate=0] at (0.5, 1.8) {M};
\node[rotate=0] at (0.2, 1.5) {F};
\node[rotate=0] at (0.5, 1.2) {P};
% tile at position (x,y)=(0, 2)
\draw[fill=lightgray] (0, 2) rectangle (1, 3);
\node at (0.5, 2.5) {16};
\draw (0, 2) -- ++ (0,1);
\draw (0, 2) -- ++ (1,0);
\draw (1, 2) -- ++ (0,1);
\draw (0, 3) -- ++ (1,0);
\node[rotate=0] at (0.8, 2.5) {I};
\node[rotate=0] at (0.5, 2.8) {K};
\node[rotate=0] at (0.2, 2.5) {B};
\node[rotate=0] at (0.5, 2.2) {M};
% tile at position (x,y)=(1, 0)
\node at (1.5, 0.5) {8};
\draw (1, 0) -- ++ (0,1);
\draw (1, 0) -- ++ (1,0);
\draw (2, 0) -- ++ (0,1);
\draw (1, 1) -- ++ (1,0);
\node[rotate=0] at (1.8, 0.5) {B};
\node[rotate=0] at (1.5, 0.8) {O};
\node[rotate=0] at (1.2, 0.5) {I};
\node[rotate=0] at (1.5, 0.2) {O};
% tile at position (x,y)=(1, 1)
\node at (1.5, 1.5) {0};
\draw (1, 1) -- ++ (0,1);
\draw (1, 1) -- ++ (1,0);
\draw (2, 1) -- ++ (0,1);
\draw (1, 2) -- ++ (1,0);
\node[rotate=0] at (1.8, 1.5) {F};
\node[rotate=0] at (1.5, 1.8) {O};
\node[rotate=0] at (1.2, 1.5) {J};
\node[rotate=0] at (1.5, 1.2) {O};
% tile at position (x,y)=(1, 2)
\node at (1.5, 2.5) {8};
\draw (1, 2) -- ++ (0,1);
\draw (1, 2) -- ++ (1,0);
\draw (2, 2) -- ++ (0,1);
\draw (1, 3) -- ++ (1,0);
\node[rotate=0] at (1.8, 2.5) {B};
\node[rotate=0] at (1.5, 2.8) {O};
\node[rotate=0] at (1.2, 2.5) {I};
\node[rotate=0] at (1.5, 2.2) {O};
% tile at position (x,y)=(2, 0)
\draw[fill=lightgray] (2, 0) rectangle (3, 1);
\node at (2.5, 0.5) {16};
\draw (2, 0) -- ++ (0,1);
\draw (2, 0) -- ++ (1,0);
\draw (3, 0) -- ++ (0,1);
\draw (2, 1) -- ++ (1,0);
\node[rotate=0] at (2.8, 0.5) {I};
\node[rotate=0] at (2.5, 0.8) {K};
\node[rotate=0] at (2.2, 0.5) {B};
\node[rotate=0] at (2.5, 0.2) {M};
% tile at position (x,y)=(2, 1)
\node at (2.5, 1.5) {3};
\draw (2, 1) -- ++ (0,1);
\draw (2, 1) -- ++ (1,0);
\draw (3, 1) -- ++ (0,1);
\draw (2, 2) -- ++ (1,0);
\node[rotate=0] at (2.8, 1.5) {D};
\node[rotate=0] at (2.5, 1.8) {M};
\node[rotate=0] at (2.2, 1.5) {F};
\node[rotate=0] at (2.5, 1.2) {K};
% tile at position (x,y)=(2, 2)
\draw[fill=lightgray] (2, 2) rectangle (3, 3);
\node at (2.5, 2.5) {16};
\draw (2, 2) -- ++ (0,1);
\draw (2, 2) -- ++ (1,0);
\draw (3, 2) -- ++ (0,1);
\draw (2, 3) -- ++ (1,0);
\node[rotate=0] at (2.8, 2.5) {I};
\node[rotate=0] at (2.5, 2.8) {K};
\node[rotate=0] at (2.2, 2.5) {B};
\node[rotate=0] at (2.5, 2.2) {M};
\end{tikzpicture}};

%\node at (1.5,0) {$\mapsto$};
%\node at (3.5,0) {$\mapsto$};

\draw[->] (A0) edge node {$\omega$} (A1);
\draw[->] (A1) edge node {$\omega$} (A2);

\end{tikzpicture}
\end{center}
\caption{The gray background color identify tiles where the morphism 
    $\omega^2:\Omega_\U\to\Omega_\U$ is prolongable on the initial tile
    $u_{16}$.}
\label{fig:omega_prolongable}
\end{figure}
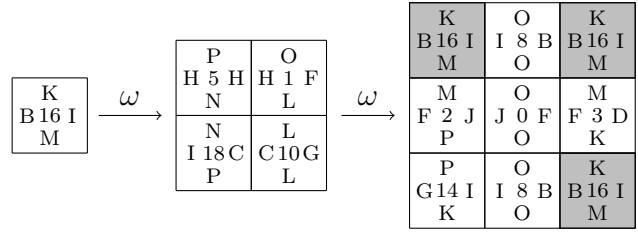

\section{A stone inflation}\label{sec:stone}

In this short section, we illustrate how the substitution $\omega$ can be seen
as a \emph{stone inflation} \cite[p.  147--148]{MR3136260}.
We consider the following correspondence between colors of the tile set $\U$
and length of edges:
\begin{align*}
    &\{A,B,C,E,G,I\} &\mapsto &&\text{ vertical edge of length } 1\\
    &\{D,F,H,J\}     &\mapsto &&\text{ vertical edge of length } \varphi^{-1}\\
    &\{K,M,N,P\} &\mapsto &&\text{ horizontal edge of length } 1\\
    &\{L,O\}     &\mapsto &&\text{ horizontal edge of length }\varphi^{-1}
\end{align*}
To each tile in $\U$ we associate a rectangle $\subset\R^2$ according to the rule
\[
\textsc{rectangle}(u_i) = 
\begin{cases}
[0,\varphi^{-1}]\times[0,\varphi^{-1}]  & \text{ if } i \in \{0,1\},\\
[0,1]\times[0,\varphi^{-1}]  & \text{ if } i \in \{2,3,4,5,6,7\},\\
[0,\varphi^{-1}]\times[0,1]  & \text{ if } i \in \{8,9,10,11\},\\
[0,1]\times[0,1]             & \text{ if } i \in \{12,13,14,15,16,17,18\}.\\
\end{cases}
\]
The relative area of each rectangle is proportional to the left eigenvector of the incidence matrix of $\omega$ (Lemma~\ref{lem:omega-is-primitive}).
Then the substitution $\omega:\Omega_\U\to\Omega_\U$ defines a
stone inflation on the finite set of tiles $\{\textsc{rectangle}(u_i)\}_{0\leq
i\leq 18}$ with expansion factor $\varphi=(1+\sqrt{5})/2$
(see Figure~\ref{fig:stoneinflation}).

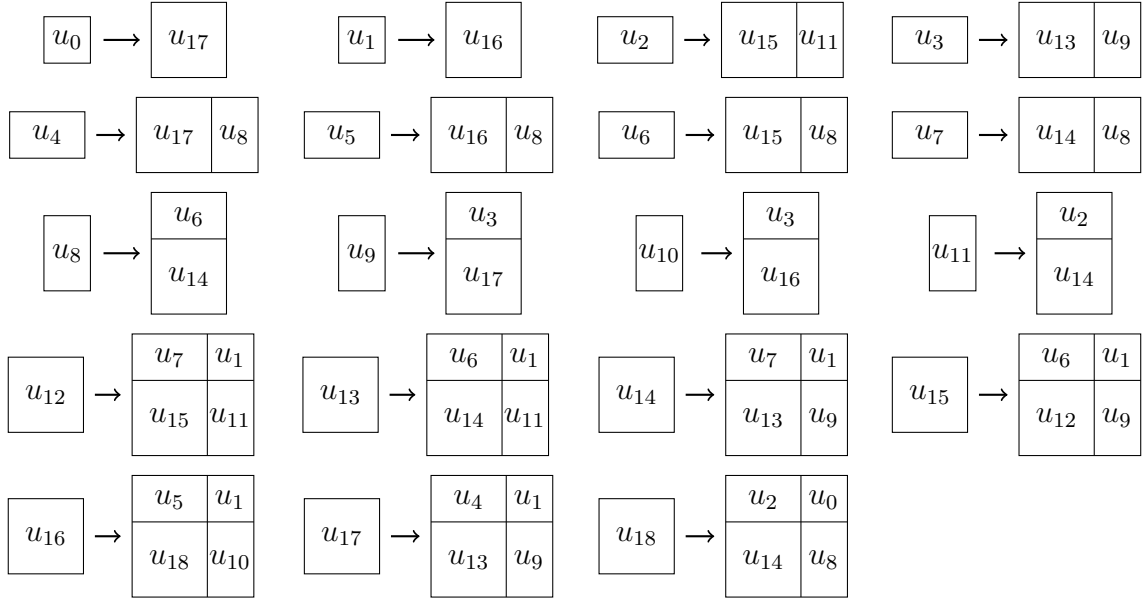
\begin{figure}[h]
\begin{center}
\newcommand\SMALLSQUARE[2]{
\begin{tikzpicture}
\node (A) at (0,0) {
\begin{tikzpicture}
    \draw (0,0) rectangle node {$u_{#1}$} (1/\p,1/\p);
\end{tikzpicture}};
\node (B) at (\p,0) {
\begin{tikzpicture}
\draw (0,0) rectangle (1,1);
\node at (.5,.5)  {$u_{#2}$};
\end{tikzpicture}};
\draw[-to,thick] (A) to (B);
\end{tikzpicture}}

\newcommand\HDOMINO[3]{
\begin{tikzpicture}
\node (A) at (0,0) {
\begin{tikzpicture}
\draw (0,0) rectangle node {$u_{#1}$} (1,1/\p);
\end{tikzpicture}};
\node (B) at (2,0) {
\begin{tikzpicture}
\draw (0,0) rectangle (\p,1);
\draw (1,0) -- (1,1);
\node at (.5,.5) {$u_{#2}$};
\node at (1.3,.5) {$u_{#3}$};
\end{tikzpicture}};
\draw[-to,thick] (A) to (B);
\end{tikzpicture}}

\newcommand\VDOMINO[3]{
\begin{tikzpicture}
\node (A) at (0,0) {
\begin{tikzpicture}
\draw (0,0) rectangle node {$u_{#1}$} (1/\p,1);
\end{tikzpicture}};
\node (B) at (\p,0) {
\begin{tikzpicture}
\draw (0,0) rectangle (1,\p);
\draw (0,1) -- (1,1);
\node at (.5,1.3) {$u_{#3}$};
\node at (.5,.5)  {$u_{#2}$};
\end{tikzpicture}};
\draw[-to,thick] (A) to (B);
\end{tikzpicture}}

\newcommand\BIGSQUARE[5]{
\begin{tikzpicture}
\node (A) at (0,0) {
\begin{tikzpicture}
    \draw (0,0) rectangle node {$u_{#1}$} (1,1);
\end{tikzpicture}};
\node (B) at (2,0) {
\begin{tikzpicture}
\draw (0,0) rectangle (\p,\p);
\draw (1,0) -- (1,\p);
\draw (0,1) -- (\p,1);
\node at (.5,.5) {$u_{#2}$};
\node at (.5,1.3) {$u_{#3}$};
\node at (1.3,.5) {$u_{#4}$};
\node at (1.3,1.3) {$u_{#5}$};
\end{tikzpicture}};
\draw[-to,thick] (A) to (B);
\end{tikzpicture}}

\begin{tikzpicture}
\matrix{
\node{\SMALLSQUARE{0}{17}}; &
\node{\SMALLSQUARE{1}{16}}; &
\node{\HDOMINO{2}{15}{11}}; &
\node{\HDOMINO{3}{13}{9}}; \\[-3mm]
\node{\HDOMINO{4}{17}{8}}; &
\node{\HDOMINO{5}{16}{8}}; &
\node{\HDOMINO{6}{15}{8}}; &
\node{\HDOMINO{7}{14}{8}}; \\[-3mm]
\node{\VDOMINO{8}{14}{6}}; &
\node{\VDOMINO{9}{17}{3}}; &
\node{\VDOMINO{10}{16}{3}}; &
\node{\VDOMINO{11}{14}{2}}; \\[-3mm]
\node{\BIGSQUARE{12}{15}{7}{11}{1}}; &
\node{\BIGSQUARE{13}{14}{6}{11}{1}}; &
\node{\BIGSQUARE{14}{13}{7}{9}{1}}; &
\node{\BIGSQUARE{15}{12}{6}{9}{1}}; \\[-3mm]
\node{\BIGSQUARE{16}{18}{5}{10}{1}}; &
\node{\BIGSQUARE{17}{13}{4}{9}{1}}; &
\node{\BIGSQUARE{18}{14}{2}{8}{0}}; \\
};
\end{tikzpicture}
\end{center}
\caption{The morphism $\omega$ seen as a stone inflation.}
\label{fig:stoneinflation}
\end{figure}

%\newpage
\section{Proofs of lemmas based on Sage}\label{sec:appendix}

This section gathers proofs 
of Lemma~\ref{lem:dominoes-direction2-U},
Lemma~\ref{lem:dominoes-direction1-V},
Lemma~\ref{lem:omega-is-primitive} and
Lemma~\ref{lem:50tiles-2x2-in-OmegaU}
made using version 8.2 of Sage \cite{sagemathv8.2}
together with the optional Sage package \texttt{slabbe-0.4.2}
\cite{labbe_slabbe_0_4_2_2018} which can be installed by running the command
\texttt{sage -pip install slabbe}.
The computations will be faster if the Gurobi linear program solver
\cite{gurobi} is installed and available in Sage.
First we import the necessary libraries from \texttt{slabbe}:
{\footnotesize
\begin{verbatim}
    sage: from slabbe import WangTileSet, TikzPicture
\end{verbatim}
}
\noindent
We create the tile set \texttt{U}:
{\footnotesize
\begin{verbatim}
    sage: L = ['FOJO', 'FOHL', 'JMFP', 'DMFK', 'HPJP', 'HPHN', 'HKFP', 'HKDP', 'BOIO',
    ....: 'GLEO', 'GLCL', 'ALIO', 'EPGP', 'EPIP', 'IPGK', 'IPIK', 'IKBM', 'IKAK', 'CNIP']
    sage: U = WangTileSet(L); U
    Wang tile set of cardinality 19
\end{verbatim}}
\noindent
This allows to create the transducer graph of $\U$ shown at
Figure~\ref{fig:T7transducer}.
{\footnotesize
\begin{verbatim}
    sage: G = U.to_transducer_graph(); G
    Looped digraph on 10 vertices
    sage: TikzPicture.from_graph(G).pdf()
\end{verbatim}}

\begin{proof}[Proof of Lemma~\ref{lem:dominoes-direction2-U}]
    We compute the number of dominoes in the direction $\be_2$ allowing a
    surrounding of radius $1,2,3$ in $\Omega_\U$. We obtain 35 dominoes in the
    direction $\be_2$  allowing a surrounding of radius $2$ and $3$.
{\footnotesize
\begin{verbatim}
    sage: [len(U.dominoes_with_surrounding(i=2,radius=r)) for r in [1,2,3]]
    [37, 35, 35]
    sage: U.dominoes_with_surrounding(i=2,radius=2)
    {(0, 8), (1, 8), (1, 9), (1, 11), (2, 16), (3, 16), (4, 13), (5, 13), (6, 14),
     (6, 17), (7, 15), (8, 0), (8, 9), (8, 11), (9, 1), (9, 10), (10, 1), (11, 1),
     (11, 10), (12, 6), (13, 4), (13, 7), (13, 18), (14, 2), (14, 6), (14, 12), (15, 7),
     (15, 13), (15, 18), (16, 3), (16, 14), (16, 17), (17, 3), (17, 14), (18, 5)}
\end{verbatim}}
\end{proof}

\noindent
Now we create the tile set \texttt{V}:
{\footnotesize
\begin{verbatim}
    sage: L = [('A','L','I','O'), ('B','O','I','O'), ('E','P','I','P'), 
    ....:  ('G','L','E','O'), ('I','K','A','K'), ('I','K','B','M'),
    ....:  ('I','P','G','K'), ('I','P','I','K'), ('AF','O','IH','O'),
    ....:  ('BF','O','IJ','O'), ('CH','P','IH','P'), ('EH','K','GF','P'),
    ....:  ('EH','K','ID','P'), ('EH','P','IJ','P'), ('GF','O','CH','L'),
    ....:  ('GF','O','EH','O'), ('ID','M','AF','K'), ('ID','M','BF','M'),
    ....:  ('IH','K','GF','K'), ('IH','K','ID','K'), ('IJ','M','GF','K')]
    sage: V = WangTileSet(L); V
    Wang tile set of cardinality 21
\end{verbatim}}

\begin{proof}[Proof of Lemma~\ref{lem:dominoes-direction1-V}]
    We compute the number of dominoes in the direction $\be_1$ allowing a
    surrounding of radius $1,2$ in $\Omega_\V$. We obtain 30 dominoes in the
    direction $\be_1$ allowing a surrounding of radius $1$ and $2$.
{\footnotesize
\begin{verbatim}
    sage: [len(V.dominoes_with_surrounding(i=1,radius=r)) for r in [1,2]]
    [30, 30]
    sage: V.dominoes_with_surrounding(i=1,radius=1)
    {(0, 4), (1, 5), (2, 3), (3, 6), (4, 1), (4, 2), (5, 1), (5, 2), (5, 7), (6, 1),
     (7, 0), (7, 1), (8, 16), (9, 17), (10, 14), (11, 15), (12, 15), (13, 15),
     (14, 11), (14, 18), (15, 18), (15, 20), (16, 12), (17, 12), (17, 19), (18, 8),
     (18, 10), (19, 8), (20, 9), (20, 13)}
\end{verbatim}}
\end{proof}

\begin{proof}[Proof of Lemma~\ref{lem:omega-is-primitive}]
We create the matrix \texttt{M} in Sage.
The characteristic polynomial, the eigenvalues and the eigenvectors of
\texttt{M} can be computed easily in Sage. 
After changing the ring of the
matrix to the number field containing $\varphi$, we obtain the result stated in
the proposition.
%\begin{verbatim}
%    sage: M.charpoly()
%    sage: z = polygen(QQ, 'z')
%    sage: K.<phi> = NumberField(z**2-z-1, 'phi', embedding=RR(1.6))
%    sage: MK = M.change_ring(K)
%    sage: MK.eigenvalues()
%    sage: MK.eigenvectors_right()
%    sage: MK.eigenvectors_left()
%\end{verbatim}
First we create the substitution $\omega$:
{\footnotesize
\begin{verbatim}
    sage: from slabbe import Substitution2d
    sage: da = {0: [[11]], 1: [[8]], 2: [[13]], 3: [[9]], 4: [[17]], 5: [[16]], 6: [[14]],
    ....:       7: [[15]], 8: [[11, 1]], 9: [[8, 0]], 10: [[18, 5]], 11: [[12, 6]],
    ....:       12: [[13, 7]], 13: [[13, 4]], 14: [[10, 1]], 15: [[9, 1]], 16: [[17, 3]],
    ....:       17: [[16, 3]], 18: [[14, 6]], 19: [[15, 7]], 20: [[14, 2]]}
    sage: alpha = Substitution2d(da)
    sage: db = {0: [[4]], 1: [[5]], 2: [[7], [0]], 3: [[2], [3]], 4: [[4], [1]],
    ....:  5: [[5], [1]], 6: [[7], [1]], 7: [[6], [1]], 8: [[18]], 9: [[16]],
    ....:  10: [[17]], 11: [[20]], 12: [[19], [8]], 13: [[18], [8]], 14: [[12], [15]],
    ....:  15: [[11], [15]], 16: [[10], [14]], 17: [[13], [15]], 18: [[20], [9]]}
    sage: beta = Substitution2d(db)
    sage: omega = alpha*beta
\end{verbatim}}
Then we create its incidence matrix, we check it is primitive, we compute it
characteristic polynomial, its eigenvalues and left and right eigenvectors:
{\footnotesize
\begin{verbatim}
    sage: M = matrix(omega)
    sage: from slabbe.matrices import is_primitive
    sage: is_primitive(M)
    True
    sage: M.charpoly().factor()
    x^3 * (x - 1)^4 * (x + 1)^4 * (x^2 - 3*x + 1) * (x^2 + x - 1)^3
    sage: z = polygen(QQ, 'z')
    sage: K.<phi> = NumberField(z**2-z-1, 'phi', embedding=RR(1.6))
    sage: MK = M.change_ring(K)
    sage: MK.eigenvalues()
    [phi + 1, -phi + 2, phi - 1, phi - 1, phi - 1, 0, 0, 0, -phi, -phi, -phi, 
        1, 1, 1, 1, -1, -1, -1, -1]
    sage: MK.eigenvectors_right()[0][1][0]
    (1, 6*phi + 3, phi + 1, 2*phi + 1, phi, phi + 1, 3*phi + 2, 2*phi + 1,
     3*phi + 2, 4*phi + 2, phi + 1, 2*phi + 1, phi, 3*phi + 2, 4*phi + 3, 2*phi + 1,
     3*phi + 2, 2*phi + 1, phi + 1)
    sage: MK.eigenvectors_left()[0][1][0]
    (1, 1, phi, phi, phi, phi, phi, phi, phi, phi, phi, phi,
     phi + 1, phi + 1, phi + 1, phi + 1, phi + 1, phi + 1, phi + 1)
\end{verbatim}}\qedhere
\end{proof}

% \begin{proof}[Proof of Lemma~\ref{lem:50tiles-2x2-in-OmegaU}]
% We do the proof using Sage. In this case, we reduce the problem to the exact
%     cover problem whose solutions can be enumerated efficiently with the
%     solver available in Sage based on the dancing links algorithm \cite{knuth_dancing_2000}.
% {\small
% \begin{verbatim}
% sage: from slabbe import WangTileSolver
% sage: tiles = ['GLCL', 'GLEO', 'FOHL', 'BOIO', 'ALIO', 'FOJO', 'CNIP', 'HPJP',
% ....: 'HPHN', 'EPIP', 'IPIK', 'HKDP', 'EPGP', 'HKFP', 'IPGK', 'JMFP', 'IKBM',
% ....: 'IKAK', 'DMFK']
% sage: tiles = map(tuple, tiles)
% sage: W = WangTileSolver(tiles, 2, 2)
% sage: W.number_of_solutions()
% 139
% \end{verbatim}
%     }
% %Each of the 139 solutions are shown in Figure~\ref{fig:self139solutions2x2}.
% \end{proof}

\begin{proof}
    [Proof of Lemma~\ref{lem:50tiles-2x2-in-OmegaU}]
From Corollary~\ref{cor:OmegaU-is-aperiodic},
we know that $\L(\Omega_\U) = \omega(\L(\Omega_\U))$.
Therefore from Lemma~\ref{lem:substitutive-equivalent-conditions},
    we have $\L_\omega\subseteq\L(\Omega_\U)$.

    Now we show using Sage that $\L(\Omega_\U)\cap \U^{(2,2)}\subseteq S$.
    First we compute $U\boxminus^1\U$. It contains 35 tiles (after recursively
    removing any source or sink state).
    Then we compute $(U\boxminus^1\U)\boxminus^2(U\boxminus^1\U)$ for which
    the resulting transducer after recursively
    removing any source or sink state contains 55
    transitions (or tiles).
    Among them, 5 can not be surrounded, that is, there is no valid tiling of a
    $3\times 3$ rectangle with them in the middle.
    For each tile $t$ of the remaining 50 tiles, we compute all valid ways 
    of writing $t=(u\boxminus^1 v)\boxminus^2(w\boxminus^1 z)$
    with $u,v,w,z\in\U$. It turns out that there is a unique way in each case
    leading to a set $S\subset\U^{(2,2)}$ of 50 subwords of shape $(2,2)$
    satisfying $\L(\Omega_\U)\cap \U^{(2,2)}\subseteq S$.
    The following takes 4s if using \verb|solver='Gurobi'| \cite{gurobi}
    and 4 min if using \verb|solver='dancing_links'|:
{\footnotesize
\begin{verbatim}
    sage: tilings = U.tiling_with_surrounding(2,2,radius=1,solver='Gurobi')
    sage: len(tilings)
    50
    sage: S = sorted(t.table() for t in tilings)
    sage: [matrix.column([col[::-1] for col in s]) for s in S]
    [
    [ 8 16]  [ 8 16]  [ 8 16]  [ 9 14]  [11 17]  [16  8]  [16 13]  [16 15]
    [ 0  3], [ 1  2], [ 1  3], [ 1  6], [ 1  6], [ 2  0], [ 2  4], [ 3  7],
    
    [13  9]  [13  9]  [14  8]  [14 11]  [14 13]  [17  8]  [17 13]  [15  8]
    [ 4  1], [ 5  1], [ 6  1], [ 6  1], [ 6  5], [ 6  1], [ 6  5], [ 7  1],
    
    [15 11]  [ 0  3]  [ 9 14]  [11 17]  [ 1  2]  [ 1  6]  [10 12]  [ 1  6]
    [ 7  1], [ 8 16], [ 8 16], [ 8 16], [ 9 14], [ 9 14], [ 9 14], [10 12],
    
    [ 1  6]  [ 1  3]  [10 14]  [ 6  1]  [ 4  1]  [ 7  1]  [18 10]  [ 2  0]
    [10 14], [11 17], [11 17], [12  9], [13  9], [13  9], [13  9], [14  8],
    
    [ 2  4]  [ 6  1]  [ 6  5]  [12  9]  [ 7  1]  [13  9]  [18 10]  [ 3  7]
    [14 13], [14 11], [14 18], [14  8], [15 11], [15  8], [15 11], [16 13],
    
    [ 3  7]  [14 11]  [14 18]  [14 13]  [14 18]  [17 13]  [ 3  7]  [14 11]
    [16 15], [16  8], [16 13], [16 15], [16 15], [16 15], [17 13], [17  8],
    
    [14 18]  [ 5  1]
    [17 13], [18 10]
    ]
\end{verbatim}}

Now we show using Sage that $S\subseteq\L_\omega$.
We compute the set $\L_\omega\cap \U^{(2,2)}$ using Sage.
{\footnotesize
\begin{verbatim}
    sage: F = omega.list_2x2_factors()
    sage: len(F)
    50
    sage: sorted(F) == S
    True
\end{verbatim}}
Thus we have shown 
$S\subseteq\L_\omega\cap \U^{(2,2)}
  \subseteq\L(\Omega_\U)\cap \U^{(2,2)}\subseteq S$.
\end{proof}

\section{Perspectives}

In this contribution, we have chosen to study the substitutive structure
of $\Omega_\U$ as $\alpha\circ\beta\circ\gamma$ where $\alpha$ involves
vertical dominoes, $\beta$ involves horizontal dominoes and $\gamma$ is a
bijection. We could have done it the other way around as
$\alpha'\circ\beta'\circ\gamma'$ where $\alpha'$ involves horizontal dominoes,
$\beta'$ involves vertical dominoes and $\gamma'$ is some other bijection as
shown below. It would be interesting to investigate all of these possible sequences
of substitutions describing $\Omega_\U$.
\begin{center}
\begin{tikzpicture}[xscale=1.5,yscale=.4,auto]
    \node (U) at (0,0) {$\Omega_\U$};
    \node (V) at (1,1) {$\Omega_\V$};
    \node (V2) at (1,-1) {$\Omega_{\V'}$};
    \node (W) at (2,1) {$\Omega_\W$};
    \node (W2) at (2,-1) {$\Omega_{\W'}$};
    \node (U2) at (3,0) {$\Omega_\U$};
    \draw[to-,very thick] (U) to node {$\alpha$} (V);
    \draw[to-,very thick] (V) to node {$\beta$} (W);
    \draw[to-,very thick] (W) to node {$\gamma$} (U2);
    \draw[to-,very thick] (U) to node[swap] {$\alpha'$} (V2);
    \draw[to-,very thick] (V2) to node[swap] {$\beta'$} (W2);
    \draw[to-,very thick] (W2) to node[swap] {$\gamma'$} (U2);
\end{tikzpicture}
\end{center}

In \cite[p. 595]{MR857454}, an aperiodic set of 16 Wang tiles is deduced from
the Ammann tiling A2 \cite{MR1156132,akiyama_note_2012} without much details.
A substitution showing its self-similar structure is also given with a
structure very close to the $2$-dimensional morphism
$\omega:\Omega_\U\to\Omega_\U$. Is there a link between Ammann tilings A2 and
$\Omega_\U$? Can we find a concave hexagonal shape with decorations related
to $\Omega_\U$? Can we factor the substitution as product of morphisms sending letters to letters or dominoes? The previous question is
related to the notion of elementary substitutions in D0L systems.

Many one-dimensional properties of $\Omega_\U$ can be extracted from the
morphism $\omega$. For example, considering the tiles $u_1$, $u_5$ and $u_6$,
one can see that the effect of $\omega^2$ on the indices of tiles is
% sage: latex(omega^2)
\[
1\mapsto \left(\begin{array}{rr}
5 & 1 \\
18 & 10
\end{array}\right),
5\mapsto \left(\begin{array}{rrr}
5 & 1 & 6 \\
18 & 10 & 14
\end{array}\right),
6\mapsto \left(\begin{array}{rrr}
6 & 1 & 6 \\
12 & 9 & 14
\end{array}\right)
\]
which contains the one-dimensional morphism:
$\mu:1 \mapsto 51,5 \mapsto 516,6 \mapsto 616$.
This can be used to show that there exist tilings in $\Omega_\U$ that contains
rows using only the tiles in the set $\{u_1, u_5, u_6\}$.
% sage: m = WordMorphism('1->51,5->516,6->616')
% sage: m.fixed_point('5')
% word: 5165161651651616516165165161651651616516...
% sage: m.latex_layout('oneliner')
% sage: latex(m)

As we have seen $\omega^2$ is prolongable on many letters thus it admits many
fixed point $x\in\Omega_\U\subset\U^{\Z^2}$.  Such fixed point were called
\emph{shape-symmetric} $d$-dimensional word by Arnaud Maes when the image of
every letter on the diagonal are squares.  In \cite{MR2579856} (see also
\cite[Chapter 4]{charlier_abstract_2009}), it is proved that a multidimensional
infinite word $x:\N^d\to\A$ over a finite alphabet $\A$ is the image by a coding
of a shape-symmetric infinite word if and only if $x$ is $S$-automatic for some
abstract numeration system $S$ built on a regular language containing the empty
word. Recall that an infinite word is $S$-automatic if, for all $n\geq0$, its
$(n+1)$-th letter is the output of a deterministic automaton fed with the
representation of $n$ in the numeration system $S$. Can we find the numeration
system $S$ and the deterministic automaton associated to $\Omega_\U$?

As noticed by the anonymous referee, the factor complexity of the language of
$\Omega_\U$ is quadratic, the first values for the number of distinct
$2$-dimensional word of shape $(n,n)$ for $n\geq 1$ being: 19, 50, 94, 154,
229, 317, 420. Tilings generated by expansive and primitive are \emph{linearly
repetitive} (which means that there is some $C>0$ for which every patch of
radius $r$ is found somewhere in every ball of radius $Cr$) and one can show
that a linearly repetitive tiling of $\R^d$ must have factor complexity bounded
by a constant times $r^d$.

%\listoftodos

\section*{Acknowledgements}

I want to thank Michaël Rao for the talk he made (Combinatorics on words and
tilings, CRM, Montréal, April 2017) from which this work is originated.  I want
to thank Vincent Delecroix for many helpful discussions at LaBRI in Bordeaux
during the preparation of this article. I am also thankful to J\"org
Thuswaldner, Henk Bruin, for inviting me to present this work (Substitutions
and tiling spaces, University of Vienna, September 2017) and to Pierre Arnoux
and Shigeki Akiyama for the same reason (Tiling and Recurrence, CIRM,
Marseille, December 2017). I want to thank Michael Baake for very
helpful discussions on inflation rules and model sets and Shigeki Akiyama for
its enthusiasm toward this project.

The author is grateful to the comments of the referee which leaded to a great
improvement in the presentation while reducing its size and simplifying many
technical proofs into simpler ones. I also wish to thank David Renault for his
careful reading of a preliminary version and his valuable comments.

I acknowledge financial support from the Laboratoire International
Franco-Québécois de Recherche en Combinatoire (LIRCO), the Agence Nationale de
la Recherche ``Dynamique des algorithmes du pgcd : une approche Algorithmique,
Analytique, Arithmétique et Symbolique (Dyna3S)'' (ANR-13-BS02-0003) and the
Horizon  2020  European  Research  Infrastructure  project  OpenDreamKit
(676541).

%%%%%%%%%%%%%%%%
% Bibliographie %
%%%%%%%%%%%%%%%%%
%\bibliographystyle{plain} %numeros
%\bibliographystyle{alpha} %author+year
\bibliographystyle{myalpha} %initials for first names
%{\footnotesize
\bibliography{../biblio}

\begin{thebibliography}{BSTY17}

\bibitem[AGS92]{MR1156132}
R. Ammann, B. Gr\"unbaum, and G.~C. Shephard.
\newblock Aperiodic tiles.
\newblock {\em Discrete Comput. Geom.}, 8(1):1--25, 1992.

\bibitem[Aki12]{akiyama_note_2012}
S. Akiyama.
\newblock A {Note} on {Aperiodic} {Ammann} {Tiles}.
\newblock {\em Discrete Comput. Geom.}, 48(3):702--710, October 2012.

\bibitem[ATY17]{akiyama_mosse_2017}
S. Akiyama, B. Tan, and H. Yuasa.
\newblock On {B}. {Mossé}'s unilateral recognizability theorem.
\newblock December 2017.
\newblock \arxiv{1801.03536}.

\bibitem[BD14]{MR3330561}
V. Berth\'e and V. Delecroix.
\newblock Beyond substitutive dynamical systems: {$S$}-adic expansions.
\newblock In {\em Numeration and substitution 2012}, RIMS K\^oky\^uroku
  Bessatsu, B46, pages 81--123. Res. Inst. Math. Sci. (RIMS), Kyoto, 2014.

\bibitem[Ber65]{MR2939561}
R. Berger.
\newblock {\em T{he} {undecidability} {of} {the} {domino} {problem}}.
\newblock ProQuest LLC, Ann Arbor, MI, 1965.
\newblock Thesis (Ph.D.)--Harvard University.

\bibitem[BG13]{MR3136260}
M. Baake and U. Grimm.
\newblock {\em Aperiodic order. {V}ol. 1}, volume 149 of {\em Encyclopedia of
  Mathematics and its Applications}.
\newblock Cambridge University Press, Cambridge, 2013.

\bibitem[BR10]{MR2742574}
V. Berth\'e and M. Rigo, editors.
\newblock {\em Combinatorics, automata and number theory}, volume 135 of {\em
  Encyclopedia of Mathematics and its Applications}.
\newblock Cambridge University Press, Cambridge, 2010.

\bibitem[BSTY17]{berthe_recognizability_2017}
V. Berthé, W. Steiner, J. Thuswaldner, and R. Yassawi.
\newblock Recognizability for sequences of morphisms.
\newblock April 2017.
\newblock \arxiv{1705.00167}.

\bibitem[Cha09]{charlier_abstract_2009}
E. Charlier.
\newblock {\em Abstract numeration systems: {Recognizability}, decidability,
  multidimensional s-automatic words, and real numbers}.
\newblock {PhD} {Thesis}, Université de Liège, Liège, Belgique, 2009.

\bibitem[CKR10]{MR2579856}
E. Charlier, T. K\"arki, and M. Rigo.
\newblock Multidimensional generalized automatic sequences and shape-symmetric
  morphic words.
\newblock {\em Discrete Math.}, 310(6-7):1238--1252, 2010.

\bibitem[Cul96]{MR1417576}
K. Culik, II.
\newblock An aperiodic set of {$13$} {W}ang tiles.
\newblock {\em Discrete Math.}, 160(1-3):245--251, 1996.

\bibitem[Dur98]{MR1489074}
F. Durand.
\newblock A characterization of substitutive sequences using return words.
\newblock {\em Discrete Math.}, 179(1-3):89--101, 1998.

\bibitem[Fra03]{MR1961011}
N.~P. Frank.
\newblock Detecting combinatorial hierarchy in tilings using derived
  {V}orono\"\i \ tesselations.
\newblock {\em Discrete Comput. Geom.}, 29(3):459--476, 2003.

\bibitem[Fra17]{2017_frank_introduction}
N.~P. Frank.
\newblock Introduction to hierarchical tiling dynamical systems.
\newblock In {\em Tiling and Recurrence, December 4-8 2017, CIRM (Marseille
  Luminy, France)}. 2017.

\bibitem[FS14]{MR3226791}
N.~P. Frank and L. Sadun.
\newblock Fusion: a general framework for hierarchical tilings of
  {$\Bbb{R}^d$}.
\newblock {\em Geom. Dedicata}, 171:149--186, 2014.

\bibitem[GO18]{gurobi}
I. Gurobi~Optimization.
\newblock Gurobi optimizer reference manual, 2018.

\bibitem[GS87]{MR857454}
B. Gr\"unbaum and G.~C. Shephard.
\newblock {\em Tilings and patterns}.
\newblock W. H. Freeman and Company, New York, 1987.

\bibitem[JR15]{jeandel_aperiodic_2015}
E. Jeandel and M. Rao.
\newblock An aperiodic set of 11 {Wang} tiles.
\newblock June 2015.
\newblock \arxiv{1506.06492}.

\bibitem[Kar96]{MR1417578}
J. Kari.
\newblock A small aperiodic set of {W}ang tiles.
\newblock {\em Discrete Math.}, 160(1-3):259--264, 1996.

\bibitem[Knu69]{MR0286317}
D.~E. Knuth.
\newblock {\em The art of computer programming. {V}ol. 1: {F}undamental
  algorithms}.
\newblock Second printing. Addison-Wesley Publishing Co., Reading,
  Mass.-London-Don Mills, Ont, 1969.

\bibitem[Lab18a]{labbe_slabbe_0_4_2018}
S. Labbé.
\newblock S. {L}abbé's {R}esearch {C}ode ({V}ersion 0.4.2).
\newblock \url{https://pypi.python.org/pypi/slabbe/}, 2018.

\bibitem[Lab18b]{labbe_structure_2018}
S. Labbé.
\newblock Substitutive structure of {J}eandel-{R}ao aperiodic tilings.
\newblock 2018.
\newblock in preparation.

\bibitem[Mos92]{MR1168468}
B. Moss\'e.
\newblock Puissances de mots et reconnaissabilit\'e des points fixes d'une
  substitution.
\newblock {\em Theoret. Comput. Sci.}, 99(2):327--334, 1992.

\bibitem[Moz89]{MR1014984}
S. Mozes.
\newblock Tilings, substitution systems and dynamical systems generated by
  them.
\newblock {\em J. Analyse Math.}, 53:139--186, 1989.

\bibitem[Oll08]{MR2507046}
N. Ollinger.
\newblock Two-by-two substitution systems and the undecidability of the domino
  problem.
\newblock In {\em Logic and theory of algorithms}, volume 5028 of {\em Lecture
  Notes in Comput. Sci.}, pages 476--485. Springer, Berlin, 2008.

\bibitem[Rob71]{MR0297572}
R.~M. Robinson.
\newblock Undecidability and nonperiodicity for tilings of the plane.
\newblock {\em Invent. Math.}, 12:177--209, 1971.

\bibitem[{Sag}18]{sagemath}
{Sage Developers}.
\newblock {\em {S}ageMath, the {S}age {M}athematics {S}oftware {S}ystem
  ({V}ersion 8.2)}, 2018.
\newblock \url{http://www.sagemath.org}.

\bibitem[Sch01]{MR1861953}
K. Schmidt.
\newblock Multi-dimensional symbolic dynamical systems.
\newblock In {\em Codes, systems, and graphical models ({M}inneapolis, {MN},
  1999)}, volume 123 of {\em IMA Vol. Math. Appl.}, pages 67--82. Springer, New
  York, 2001.

\bibitem[Sol97]{MR1452190}
B. Solomyak.
\newblock Dynamics of self-similar tilings.
\newblock {\em Ergodic Theory Dynam. Systems}, 17(3):695--738, 1997.

\bibitem[Sol98]{MR1637896}
B. Solomyak.
\newblock Nonperiodicity implies unique composition for self-similar
  translationally finite tilings.
\newblock {\em Discrete Comput. Geom.}, 20(2):265--279, 1998.

\bibitem[Wan61]{wang_proving_1961}
H. Wang.
\newblock Proving {Theorems} by {Pattern} {Recognition} — {II}.
\newblock {\em Bell System Technical Journal}, 40(1):1--41, January 1961.

\end{thebibliography}
%}

%\input{old_stuff.tex}

\end{document}